\documentclass[11pt,reqno]{amsart}
\usepackage{type1ec}
\usepackage[utf8]{inputenc}
\usepackage[T1]{fontenc}
\usepackage{amssymb,amsfonts,amsthm}
\usepackage{graphicx}
\usepackage{tikz}              
\usetikzlibrary{calc}  
\usepackage{amsthm}
\usepackage{amssymb}
\usepackage[french,english]{babel}
\usepackage{setspace}

\usepackage{color}

\definecolor{Red}{cmyk}{0,1,1,0.2}

\usepackage{amsfonts,amsmath,layout}
\newcommand{\R}{\mathbb R}

\def\R{\mathbb R}

\def\N{\mathbb N}

\def\P{\mathbb P}

\newcommand{\be}{\begin{equation}}
\newcommand{\ee}{\end{equation}}

\def\1{{\bf 1}}

\def\ds{\displaystyle}

\newcommand{\pare}[1]{\left(#1\right)}

\newcommand{\croc}[1]{\left[#1\right]}

\newtheorem{Theorem}{Theorem}[section]
\newtheorem{Definition}[Theorem]{Definition}
\newtheorem{Proposition}[Theorem]{Proposition}
\newtheorem{Lemma}[Theorem]{Lemma}

\newtheorem{Remark}[Theorem]{Remark}

\newtheorem{Sketch of proof}[Theorem]{Sketch of proof}

\begin{document}
\title[Linear Parabolic PDE with local-time Kirchhoff's condition]{Well posedness of linear Parabolic partial differential equations posed on a star-shaped network with local time Kirchhoff's boundary condition at the vertex}
\author[Miguel Martinez \& Isaac Ohavi]{Miguel Martinez$^{\dagger}$, Isaac Ohavi$^{\star}$\\
{$^\dagger$ Universit\'e Gustave Eiffel, LAMA UMR 8050, France}\\{$^\star$ Hebrew University of
Jerusalem, Department of Statistics, Israël}}
\email{miguel.martinez@univ-eiffel.fr}
\email{isaac.ohavi@mail.huji.ac.il \& isaac.ohavi@gmail.com}
\thanks{Research partially granted by the Labex Bézout\\Research partially supported by the GIF grant 1489-304.6/2019}
\dedicatory{Version: \today}

\maketitle
\begin{abstract} The main purpose of this work is to provide an existence and uniqueness result for the solution of a linear parabolic system posed on a star-shaped network, which presents a new type of Kirchhoff's boundary transmission condition at the junction. This new type of Kirchhoff's condition~-~that we decide to call here {\it local-time Kirchhoff's condition}~-~induces a dynamical behavior with respect to an external variable that may be interpreted as a local time parameter, designed to drive the system only at the singular point of the network. The seeds of this study point towards a  theoretical inquiry of a particular generalization of Walsh's random spider motions, whose spinning measures would select the available directions according to the local time of the motion at the junction of the network.
\end{abstract}

\section{Introduction}\label{sec:intro}
We are given a terminal time condition $T>0$, an integer number $I$ (with $I\ge 2$) and star-shaped compact network: $$\displaystyle \mathcal{N}_R=\bigcup_{i=1}^I\mathcal{R}_i,$$ that consists of $I$ compact rays $\mathcal{R}_i\cong [0,R]$ ($R>0$) emanating from a junction point $\{0\}$. In this study, we investigate the well-posedness for classical solutions of the following PDE system with parameter posed on $\mathcal{N}_R$ that involves a {\it local-time Kirchhoff's boundary condition at $\{0\}$}:
\begin{eqnarray}\label{eq : pde with l}
\begin{cases}\partial_tu_i(t,x,l)-a_i(t,x,l)\partial_x^2u_i(t,x,l)
+b_i(t,x,l)\partial_xu_i(t,x,l)\\
\hspace{0,4 cm}+c_i(t,x,l)u_i(t,x,l)=f_i(t,x,l),~~(t,x,l)\in (0,T)\times (0,R)\times(0,K),\\
\text{\it Local time Kirchhoff's boundary condition:}\\
\partial_lu(t,0,l)+\displaystyle \sum_{i=1}^I \alpha_i(t,l)\partial_xu_i(t,0,l)-r(t,l)u(t,0,l)=\phi(t,l),~~(t,l)\in(0,T)\times(0,K)\\
\medskip
\partial_xu_i(t,R,l)=0,~~ (t,l)\in (0,T)\times(0,K),\\
\forall (i,j)\in[\![1,I]\!]^2,~~u_i(t,0,l)=u_j(t,0,l)=u(t,0,l),~~(t,l)\in|0,T]\times[0,K],\\
\forall i\in[\![1,I]\!],~~ u_i(t,x,K)=\psi_i(t,x),~~ (t,x)\in [0,T]\times[0,R],\\
\forall i\in[\![1,I]\!],~~ u_i(0,x,l)=g_i(x,l),~~(x,l)\in[0,R]\times[0,K].
\end{cases}
\end{eqnarray}
In order to simplify our study, we have assumed in our framework that all the rays $\mathcal{R}_i$ have the same length $R>0$, that a Neumann boundary condition holds at $x=R$, and that a Dirichlet boundary condition $\psi_i$ holds at $l=K$. Of course, a more general network setting could be treated with similar tools: one could for instance consider more general rays, and/or a mix of Neumann and Dirichlet boundary conditions or local-time Kirchhoff's boundary conditions at the vertices, unbounded rays, etc.

Let us explain the main motivation that grounds our study of the system \eqref{eq : pde with l}. Initially introduced by J. Walsh in \cite{Walsh}, Walsh’s Brownian spider motion is a continuous process on a set of $I$ rays embedded in $\R^2$ emanating from $\{0\}$. Roughly speaking, to each ray $\mathcal{R}_i$ we associate a weight $\alpha_i$ that  corresponds (very) heuristically to the probability for the process to visit $\mathcal{R}_i$ when it leaves the junction point $\{0\}$. Inside each ray and apart from the junction point, the process behaves like a Brownian motion. However, because the trajectories of the Brownian
motion are not of bounded variation, this intuitive description does not make sense: starting from $\{0\}$ the process visits all the rays at once (all the rays are visited on any arbitrary small time interval). 

As a generalisation of Walsh’s Brownian motion, diffusions on graphs were introduced in the seminal works of Freidlin and Wentzell \cite{Freidlin} for a star-shaped network $\mathcal{N}_R$ and then for general graphs in Freidlin and Sheu \cite{freidlinS}. 

Given $I$ pairs $(\sigma_{i},b_{i})_{i\in I}$ of mild coefficients of diffusion from $[0,+\infty)$ to $\R$ satisfying the following condition of ellipticity: $\forall i\in [\![1,I]\!],~\sigma_i>0$, and given $\big(\alpha_1,\ldots,\alpha_I)$ positive constants satisfying $\displaystyle \sum_{i=1}^I \alpha_i=1$, it is proved in \cite{Freidlin} that there exists a continuous Feller Markov process $\big(x(\cdot),i(\cdot)\big)$ valued in $\mathcal{N}_R$ whose generator is given by the following operator:
$$\mathcal{L}:\begin{cases}\mathcal{C}^2(\mathcal{N}_R)\to \mathcal{C}(\mathcal{N}_R),\\
f=f_i(x)\mapsto 
b_i(x)\partial_xf_i(x)+\displaystyle \frac{\sigma_i^2(x)}{2}\partial_x^2f_i(x)\end{cases},$$
with domain
$$D(\mathcal{L}):=\Big\{ f \in \mathcal{C}^2(\mathcal{N}_R),~~\displaystyle \sum_{i=1}^I\alpha_i\partial_xf_i(0)=0\Big\}.$$
In the above, for $k=0,1,2,\dots$ the $k$-th order continuous class space on the junction network $\mathcal{C}^k(\mathcal{N}_R)$ is defined as 
\[
\left \{f:~\mathcal{N}_R\rightarrow \R,\,\,(x,i)\mapsto f_i(x)\;\;\text{s.t.}\;\;\;\forall (i,j)\in [\![1,I]\!]^2,\;\;\;f_i(0) = f_j(0),\;\;\;f_i\in {\mathcal C}^k([0,R])\right \}.
\]
\begin{Remark}
As is standard in the definition of the space $\mathcal{C}^k(\mathcal{N}_R)$ (see \cite{Freidlin}, \cite{Ohavi PDE}, \cite{Lions Souganidis 1}), we do not impose the existence of a continuous $k$-th order derivative at the junction, so the notation $\mathcal{C}^k(\mathcal{N}_R)$ might be a little misleading at first. The reader should keep in mind that - in addition to the existence of $k$-th order separate derivatives in all directions - only continuity is imposed at the junction $0$. 
\end{Remark}

Thereafter, it is shown in \cite{freidlinS} that there exists a one dimensional Wiener process $W$ defined on a probability space $(\Omega,\mathcal{F},\P)$ and adapted to the natural filtration of $\big(x(\cdot),i(\cdot)\big)$, such that the process $\big(x(\cdot)\big)$ satisfies the following stochastic differential equality:
\begin{equation*}
dx(t)= b_{i(t)}(x(t))dt + \sigma_{i(t)}(x(t))dW(t)+d\ell(t) \;,~~0\leq t\leq T.
\end{equation*}
In the above equality the process $(\ell(\cdot))$ has increasing paths, starts from $0$ and satisfies:
\begin{eqnarray*}
\forall t\in [0,T],~~\int_{0}^t\mathbf{1}_{\{x(s)>0\}}d\ell(s)=0,~~\P-\text{a.s.}
\end{eqnarray*}
Moreover, the following It\^{o}'s formula is proved in \cite{freidlinS}: 
\begin{eqnarray*}
\nonumber&\displaystyle df_{i(t)}(x(t))~~=~~ \Big(b_{i(t)}(x(t))\partial_xf_{i(t)}(x(t))+ \frac{1}{2}\sigma_{i(t)}^2(x(t))\partial_{x}^2f_{i(t)}(x(t))\Big)dt+\\
&
\displaystyle \partial_xf_{i(t)}(x(t))\sigma_{i(t)}(x(t))dW(t) 
+ \sum_{i=1}^{I}\alpha_{i}\partial_xf_i(0)d\ell(t),~~\P-\text{a.s,}
\end{eqnarray*}
for any sufficiently regular $f$. 
The process  $\ell(\cdot)$ can be interpreted as the local time of the process  $\big(x(\cdot),i(\cdot)\big)$ at the junction point $\{0\}$ ; indeed, a quadratic approximation of the local time process $\ell$ is given by the convergence:
\begin{eqnarray}\label{eq : temps local}
\lim_{\varepsilon \to 0}~~\mathbb{E}^{\mathbb{P}} \Big[~~\Big|~~\Big(\frac{1}{2\varepsilon}\sum_{i=1}^{I}\int_{0}^{\cdot}\sigma^{2}_{i}(0)\mathbf{1}_{\{0\leq x(s)\leq\varepsilon,i(s)=i\}}ds\Big)~~-~~\ell(\cdot)~~\Big|_{(0,T)}^2~~\Big] ~~ =~~0.
\end{eqnarray}
In \cite{Spider}, we have built a spider diffusion process satisfying uniqueness in law, with random spinning measure $(\alpha_i)_{i\in I}$ depending on the own local time of the spider process at the junction $\{0\}$. 
In this framework, the underlying process  $\Big(\big(x(\cdot),i(\cdot)\big),\ell(\cdot)\Big)$ exists and satisfies the following It\^{o}'s rule: \begin{eqnarray}\label{eq Ito spider originale}
\nonumber &f_{i(s)}(s,x(s),\ell(s))- f_{i(t)}(t,x(t),\ell(t))=\displaystyle\int_{t}^{s} \Big[\partial_tf_{i(u)}(u,x(u),\ell(u))+\\
&\nonumber b_{i(u)}(u,x(u),\ell(u))\partial_xf_{i(u)}(u,x(u),\ell(u))+
\displaystyle \frac{\sigma_{i(u)}^2(u,x(u),\ell(u))}{2}\partial_x^2f_{i(u)}(u,x(u),\ell(u))\Big]du+\\
&\nonumber \ds\int_{t}^{s}\sigma_{i(u)}(u,x(u),\ell(u))\partial_xf_{i(u)}(u,x(u),\ell(u))dW(u)+\\
&\ds \int_{t}^{s}\big[\partial_lf(u,0,\ell(u))+
\ds \sum_{i=1}^I\alpha_i(u,\ell(u))\partial_xf_i(u,0,\ell(u))\big]d\ell(u),~~s\ge t,
\end{eqnarray} 
for sufficiently regular $f$.
In the above, we use the data set $\{\alpha_i~:~(t,l);~~(i,t,l)\in [\![1,I]\!]\times [0,T]\times[0,+\infty)\}$ that denotes the selection spinning coefficients satisfying $$\displaystyle \forall (t,l)\in  [0,T]\times[0,+\infty),\;\;\;\;\sum_{i=1}^I \alpha_i(t,l)=1.$$
The proof for the existence of such a process satisfying \eqref{eq Ito spider originale} was performed using the original construction of concatenation of solutions for martingales problems by Stroock and Varadhan in \cite{Strook}. The more difficult proof for a criterion that ensures the uniqueness of a weak solution (or uniqueness in law) was achieved using a PDE argument: this last important issue gives the justification for this contribution. 

\medskip
\medskip

The precise understanding of the diffraction of Walsh's random motions may have important applications, for instance if one is interested to describe the diffusive behavior of particles subjected to scattering (or diffraction), for which very little physical information is known. As an example, we mention \cite{Scatterin Theory}: the theory of quantum trajectories states that quantum systems can be modelled as scattering processes and that these scattering effects may occur in prescribed directions emanating from a single point. Note that light scattering is a phenomenon that has long-time attracted many scientists for its importance in advanced photonics technologies such as on-chip interconnections, refined bio-imaging, solar-cells, heat-assisted magnetic recording, etc. (For an account on all these topics and the importance of the scattering phenomenon, see for e.g. \cite{DLS}).
\medskip
\medskip

There have been several works on linear and quasilinear parabolic non degenerate equations of the form \eqref{eq : pde with l} - but without involving any dependence with respect to some 'external variable' - that present a classical formulation of the boundary Kirchhoff's condition:
$$\sum_{i=1}^I\alpha_i(t)\partial_xf_i(t,0)=0,~~t\in(0,T).$$
For linear equations, up to our knowledge, one of the most relevant work is the one which was carried out by Von Below in references \cite{Von Below, Below 3, Below 4}. Essentially, it is shown in \cite{Von Below} that - under natural smoothness and strong compatibility conditions - linear boundary value problems defined on a star-shape network that involve a linear boundary Kirchhoff's condition at the junction point are well-posed. The proof relies on a particular linear transformation that {\it mutis mutandis} permits to retrieve the classical framework of parabolic systems. Note that this approach increases the dimension of the original problem and cannot be adapted directly to the framework of this contribution -- at least to the best of our abilities. We revisit the result of \cite{Von Below} in Section \ref{sec : Von Below revisited} by presenting another path for the construction of solutions: namely, we follow the main ideas presented by the second author in \cite{Ohavi PDE} (in a Quasi-linear parabolic framework) and proceed to the proof of the convergence of the corresponding elliptic schemes, with the same method used successively in \cite{Ohavi PDE} for the existence of classical solutions in suitable H\"{o}lder spaces for non degenerate quasi-linear parabolic systems.

In \cite{Below 3} the strong maximum principle for semi linear parabolic operators with Kirchhoff's condition was proved, while in \cite{Below 4} the author studied the classical global solvability for a class of semilinear parabolic equations on ramified networks, where a time-dynamical condition is prescribed at each node of the underlying network.
Compared to the results stated in \cite{Von Below} (when no dependency on the 'external variable' $l$ is involved), our methodology permits to re-state the well-posedness of the problem in the fully linear case with a weakening on the necessary compatibility conditions for the data at the boundary and also a weakening on the required regularity of the coefficients at the junction point $\{0\}$. We investigate also refined bounds for the derivatives of the solution - especially the key time derivative term $|\partial_tu(t,0)|$ - that play a crucial role in the construction of the solution to the system \eqref{eq : pde with l}.

In the linear setting, let us mention also another approach that was developed by M.K. Fijavz, D.Mugnolo and E. Sikolya in \cite{M-K}: their idea is to combine semi-group theory with variational methods in order to understand how the spectrum of the operator relates to the structure of the network. We will not investigate these issues in this contribution. 
\medskip

Parabolic (or elliptic) equations posed on networks can also be analyzed in terms of viscosity solutions. To our knowledge, the first results on viscosity solutions for Hamilton-Jacobi equations on networks have been obtained by Schieborn in \cite{D-S these} for the Eikonal equation. Later, investigations have been discussed in many contributions on first order problems \cite{Camilli 1, Imbert Nguyen, Lions College France}, elliptic equations \cite{Lions Souganidis 1} and second order problems with vanishing diffusion at the vertex \cite{Lions Souganidis 2}. 
In contrast and to the best of our expertise, mainly because of the difficulty of this subject, second order Hamilton-Jacobi equations on networks with a non-vanishing viscosity at the vertices have seldom been studied in the literature. Indeed, because
of the discontinuities of the Hamiltonians, the classical theory of viscosity solutions
cannot be applied directly. However, very recently, the second author managed to obtain a comparison theorem (thus uniqueness) for continuous viscosity solution to some kind of Walsh’s spider Hamilton-Jacobi-Bellman system that possesses a new type of boundary condition at the vertex $\{0\}$ involving
a non linear local time Kirchhoff ’s transmission (see \cite{Ohavi visco} for details): this non linear local transmission condition may be seen as a non linear generalization of the local time Kirchhoff's condition of \eqref{eq : pde with l} that is being investigated here for the first time.  The system studied in \cite{Ohavi visco} can also be seen as the extension in the elliptic non linear framework of system \eqref{eq : pde with l}.  Note that in \cite{Ohavi visco},
the introduction of the external ’local-time’ variable $\ell$, with aid of a new technique for building test functions at the neighborhood of $\{0\}$,
is a crucial ingredient to obtain the comparison principle. In future work,
we plan to prove that the viscosity solution to some kind of {\it Walsh’s spider Hamilton-Jacobi-Bellman system }characterizes in a unique way the value function of a well designed stochastic scattering control problem, with optimal spinning measure selected from the own local time of the Walsh spider process.
\medskip
\medskip

The construction of the solution of the system \eqref{eq : pde with l} involving the local time variable $l$ is achieved by proving the convergence of a parabolic scheme that uses a discretization grid corresponding to the variable $l$ (see Section \ref{sec: preuve résultat principal}).

Using classical arguments, we prove that uniqueness for solutions of system \eqref{eq : pde with l} holds true for solutions that have enough regularity (see Theorem \ref{th : para comparison th with l}). Under mild assumptions, we will see that classical solutions of the system \eqref{eq : pde with l} belong to the class $\mathcal{C}^{1,2}$ in the interior of each edge and $\mathcal{C}^{0,1}$ in the whole domain (with respect to the time-space variables $(t,x)$). Since the variable $l$ drives dynamically the system only at the junction point $\{0\}$ with the presence of the derivative $\partial_lu(t,0,l)$ in the local time Kirchhoff's boundary condition, one can expect a regularity in the class $\mathcal{C}^1$ for $l\mapsto u(t,0,l)$ and this is indeed the case (see our main Theorem \ref{th : exis para with l} and point $iv)$ in Definition \ref{definition-espace-solutions}). Inside each ray, because of the lack of information on the dependency of the solution w.r.t. the variable $l$, we believe there is very little hope to prove the existence of a partial derivative with respect to $l$ in the classical sense. However, we manage to prove that the solution of the system admits a square integrable generalized derivative $\partial_l u$ with respect to the variable $l$ (see again Theorem \ref{th : exis para with l} and point $v)$ in Definition \ref{definition-espace-solutions}).

Recall that the roots of our study of system \eqref{eq : pde with l} are grounded to our inquiry regarding the possible construction of a Walsh-spider diffusion living on $\mathcal{N}_R$ having a spinning measure that selects directions with respect to its own local time. Having this in mind, one should remember that the local time at the junction point $\{0\}$ exists only if diffusion coefficients are non degenerate. Clearly, both problems are deeply connected.
From a PDE technical aspect pointing towards the construction of the corresponding Walsh-spider diffusion, the main challenge is to obtain an H\"{o}lder continuity of the partial functions $l\mapsto \Big(\partial_tu_i(t,x,l),\partial_xu_i(t,x,l),\partial_x^2u_i(t,x,l)\Big)$ for any $x>0$. 
We will show that such regularity is guaranteed by the central assumption on the ellipticity of the diffusion coefficients on each rays together with the mild dependency of the coefficients and free term with respect to the variable $l$. 
\medskip
\medskip

The paper is organized as follows. In Section \ref{Main results} we introduce all the necessary material needed for our purposes and we announce our main Theorem \ref{th : exis para with l}. We also state a comparison theorem (Theorem \ref{th : para comparison th with l}) that will be of constant use in the proofs. Without involving the additional {\it local time variable} '$l$' at this stage, but under somewhat weaker assumptions, we provide in Section \ref{sec : Von Below revisited} another proof of the main result obtained in \cite{Von Below} for parabolic systems that involve a standard boundary Kirchhoff's transmission condition. In particular, by adapting the same methods as those employed in \cite{Ohavi PDE}, we manage to derive interesting bounds for the solution and its partial derivatives. Finally in Section \ref{sec: preuve résultat principal}, we prove our main result concluding to the well-posedness of system \eqref{eq : pde with l}.

\section{Introduction and Main results}\label{Main results}

In this section we state our main result - Theorem \ref{th : exis para with l} - regarding the solvability of the parabolic problem \eqref{eq : pde with l} involving the {\it local-time Kirchhoff's boundary condition} at the junction point, posed on a star-shaped compact network.

\subsection{Notations and preliminary results}\label{subsec: Prel}
Let us start by introducing the main notations as well as some preliminary results.
\medskip 

Let $I\in \mathbb{N}^*$ be the number of edges and $R>0$ be the common length of each ray. The bounded star-shaped compact network $\mathcal{N}_R$ is defined by:
$$\mathcal{N}_R=\bigcup_{i =1}^I \mathcal{R}_{i}$$
where
\begin{eqnarray*} \forall i\in [\![1,I]\!]~~\mathcal{R}_{i}:=[0,R]~~\text{and}~~\forall (i,j)\in [\![1,I]\!]^2,~~i\neq j,~~\mathcal{R}_{i}\cap \mathcal{R}_{j}=\{0\}.
\end{eqnarray*}
The intersection of all the rays $(\mathcal{R}_{i})_{1 \leq i\leq I}$ is called the junction point and is denoted by $\{0\}$.\\
We identify  all the points of $\mathcal{N}_R$ by couples $(x,i)$ (with $i \in[\![1,I]\!], x\in|0,R]$), such that we have: $(x,i)\in \mathcal{N}_R$, if and only if $x\in \mathcal{R}_{i}$. The compact star-shaped network $\mathcal{N}_R$ taken without the junction point $\{0\}$ is denoted by:
\begin{equation}
\label{eq:reseau-etoile}
\mathcal{N}_R^*=\mathcal{N}_R\setminus \{0\}.
\end{equation}
The 'space domain' of the PDE system \eqref{eq : pde with l} is the following:
\begin{eqnarray*}\Omega= \overset{\circ}{\mathcal{N}_R}\times (0,K) \ni ((x,i),l),
\end{eqnarray*}
whereas the 'time-space domain' is:
$$\Omega_T=(0,T)\times \Omega \ni (t,(x,i),l),$$
where $T>0$ denotes some fixed horizon time.

Regarding the functional spaces that will be used in the sequel, we will use standard notations (see e.g. Chapter 1.1 of \cite{pde para}) adapted for our purposes that are recalled in Appendix \ref{sec : functionnal spaces} for the convenience of the reader. Before giving the definition of the class of regularity for our solution of the PDE system \eqref{eq : pde with l}, we introduce (using the notations of Appendix \ref{sec : functionnal spaces}), the following sets:
\begin{align*}
  &i)~\mathcal{C}^{\frac{\alpha}{2},\alpha,\frac{\alpha}{2}}\big([0,T]\times [0,R]\times [0,K],\,\R \big):=\mathcal{H}^{\frac{\alpha}{2},\alpha,\frac{\alpha}{2}}\big([0,T]\times [0,R]\times [0,K]\big);\\
  &ii)~\mathcal{C}^{\frac{\alpha}{2},1+\alpha,\frac{\alpha}{2}}\big([0,T]\times [0,R]\times [0,K),\,\R \big):=\Big\{f:[0,T]\times [0,R]\times [0,K]\to \R,~(t,x,l)\mapsto f(t,x,l),\\
  &\Big|~f\in \mathcal{C}^{0,1,0}\big([0,T]\times [0,R]\times [0,\overline{K}] \big),~(f,\partial_xf)\in \mathcal{H}^{\frac{\alpha}{2},\alpha,\frac{\alpha}{2}}\big([0,T]\times [0,R]\times [0,\overline{K}]\big)^2,~\forall \overline{K}\in (0,K) \Big\};\\
  &(iii)~\mathcal{C}^{\frac{\alpha}{2},1+\frac{\alpha}{2}}\big([0,T]\times [0,K),\,\R \big):=\Big\{f:[0,T]\times [0,K]\to \R,~(t,l)\mapsto f(t,l)~~\Big|\\
  &\hspace{3.6 cm}f\in \mathcal{C}^{0,1}\big([0,T]\times [0,\overline{K}] \big),~(f,\partial_lf)\in \mathcal{H}^{\frac{\alpha}{2},\frac{\alpha}{2}}\big([0,T]\times [0,\overline{K}]\big)^2,~\forall \overline{K}\in (0,K) \Big\};\\
  &(iv)~\mathcal{C}^{1+\frac{\alpha}{2},2+\alpha,\frac{\alpha}{2}}\big((0,T)\times (0,R)\times (0,K),\,\R \big):=\Big\{f:[0,T]\times [0,R]\times [0,K]\mapsto \R,\\&\hspace{2.5 cm}\Big|~f\in \mathcal{C}^{1,2,0}\big(\mathcal{O}\big),~(f,\partial_tf,\partial_xf,\partial_x^2f)\in \mathcal{H}^{\frac{\alpha}{2},\alpha,\frac{\alpha}{2}}\big(\mathcal{O}\big),~\forall O\subset \subset [0,T]\times [0,R]\times [0,K]\Big\}.
\end{align*}
In the above, the notation $\forall O\subset \subset [0,T]\times [0,R]\times [0,K]$ refers to any open set $\mathcal{O}$ separated from the boundary of $[0,T]\times [0,R]\times [0,K]$ by a strictly positive distance, namely: 
\begin{eqnarray*}
\inf~~\Big\{||x-y||_{\R^3},~~ y\in\partial \Big([0,T]\times [0,R]\times [0,K]\Big),~x \in \overline{\mathcal{O}}\Big\}~~>~~0.
\end{eqnarray*}
Let us now give the definition of the class of regularity for our solution of the PDE system \eqref{eq : pde with l}.
\begin{Definition}
\label{definition-espace-solutions}
Let $\alpha \in (0,1)$. 

We say that
$$f:\overline{\Omega}_T\to \R,\;\;\big(t,(x,i),l\big)\mapsto f_i(t,x,l)$$
is in the class $f\in \mathfrak{C}^{1+\frac{\alpha}{2},2+\alpha,\frac{\alpha}{2}}_{\{0\}} \big(\Omega_T\big)$ if:\\
(i) for all $(t,l)\in[0,T]\times[0,K]$, for all $(i,j)\in [\![1,I]\!]^2$, $f_i(t,0,l)=f_j(t,0,l)=f(t,0,l)$ {\it (continuity condition at the junction point $\{0\}$)};\\
(ii) for all $i\in [\![1,I]\!]$, the map $(t,x,l)\mapsto f_i(t,x,l)$ is in the class $\mathcal{C}^{\frac{\alpha}{2},\alpha,\frac{\alpha}{2}}\big([0,T]\times [0,R]\times [0,K],\,\R \big)\cap \mathcal{C}^{\frac{\alpha}{2},1+\alpha,\frac{\alpha}{2}}\big([0,T]\times [0,R]\times [0,K),\,\R \big)$;\\
(iii) for all $i\in [\![1,I]\!]$, the map $(t,x,l)\mapsto f_i(t,x,l)$ belongs to $\mathcal{C}^{1+\frac{\alpha}{2},2+\alpha,\frac{\alpha}{2}}\big((0,T)\times (0,R)\times (0,K),\,\R \big)$;
moreover, $(\partial_tf_i,\partial_x^2f_i)\in L_{\infty}\big((0,T)\times (0,R)\times (0,K)\big)$;\\
(iv) at the junction point $\{0\}$, the map $(t,l)\mapsto f(t,0,l)$ belongs to $\mathcal{C}^{\frac{\alpha}{2},1+\frac{\alpha}{2}}\big([0,T]\times [0,K),\,\R \big)$;\\
(v) finally, for all $i\in [\![1,I]\!]$, on each ray $\mathcal{R}_i$, $f$ admits a generalized derivative with respect to the variable $l$ in $\displaystyle\bigcap\limits_{q\in (1,+\infty)}L^q\big((0,T)\times (0,R)\times (0,K)\big)$.
\end{Definition}

In the same way, we define the classes  $\mathfrak{C}^{1,2,0}_{\{0\}}\big(\Omega_T\big)$ analogously to $\big[i)-ii)-iii)-iv)-v)\big]$ but without any additional H\"older regularity;
$\mathfrak{Lip}^{2,0}_{\{0\}}\big(\Omega\big)$ is defined analogously to $\big[i)-ii)-iii)-iv)-v)\big]$ but removing the dependence on the time variable and with Lipschitz regularity.

Let us recall a very useful lemma of interpolation. The main ingredients of its proof can be found in Lemma 2.1 of \cite{Ohavi PDE} and since this contribution is already quite long, we have decided not to give it here. The idea is to observe that the H\"{o}lder constants appearing in the statement are uniforms in the variables $(t,x,l)$, so that the proof can be achieved exchanging the roles of $l$ and $t$ in Lemma 2.1 of \cite{Ohavi PDE}. 

\begin{Lemma}\label{lm : cont deriv temps grad}
Fix $\underline{K}>0$. Assume that $u:=u(t,x,l) \in \mathcal{C}^{0,1,0}([0,T]\times[0,R]\times [0,\underline{K}])$ satisfies 
\begin{eqnarray*}
\forall (t,s,l,q)\in[0,T]^2\times [0,\underline{K}]^2\;\;\;\text{s.t.}\;\;\;\;|t-s|\leq 1,&&\!\!\!\!\!\!\!\!|l-q|\leq 1,\text{ and }~~\forall x,y\in[0,R]:\\
|u(t,x,l)- u(s,x,q)|&\leq& \nu_1|t-s|^\alpha + \nu_2|l-q|^\beta,\\
|\partial_xu(t,x,l)-\partial_xu(t,y,l)|&\leq& \nu_3|x-y|^\gamma
\end{eqnarray*}
for some given constants $\nu_1, \nu_2, \nu_3\in \R^+$ and $\alpha, \beta, \gamma \in (0,1)$.
Then
\begin{eqnarray*}
&\forall (t,s,l,q)\in[0,T]^2\times [0,\underline{K}]^2,\;\;\;\text{s.t.}\;\;\;\;|t-s|\leq 1,~~|l-q|\leq 1, \text{ and }~~\forall x\in[0,R]:\\
&|\partial_xu(t,x,l)-\partial_xu(s,x,l)|~~~\leq ~~\Big(2\nu_3\Big(\frac{\nu_1}{\gamma \nu_3}\Big)^{\frac{\gamma}{1+\gamma}}~~+~~2\nu_1\Big(\frac{\gamma\nu_3}{ \nu_1}\Big)^{-\frac{1}{1+\gamma}}\Big)|t-s|^{\frac{\alpha\gamma}{1+\gamma}},\\
&|\partial_xu(s,x,l)-\partial_xu(s,x,q)|~~~\leq ~~\Big(2\nu_3\Big(\frac{\nu_2}{\gamma \nu_3}\Big)^{\frac{\gamma}{1+\gamma}}~~+~~2\nu_2\Big(\frac{\gamma\nu_3}{ \nu_2}\Big)^{-\frac{1}{1+\gamma}}\Big)|l-q|^{\frac{\beta \gamma}{1+\gamma}}.
\end{eqnarray*}
\end{Lemma}
One of the main important technical issues when one wants to study the well posedness of the system \eqref{eq : pde with l} is to characterize the regularity of the derivatives $\big(\partial_tu,\partial_xu,\partial_x^2u\big)$ with respect to the variable $l$ of some possible generalized solution $u$. Here, we will see that the smoothness of a generalized solution (the term generalized being applied only on each ray $\mathcal{R}_i$ separately, leaving the junction $\{0\}$ out) is determined only by the smoothness of the coefficients and free terms. We state the following important lemma.

\begin{Lemma}\label{lem : regu Holder interieure}
Suppose that $u \in W^{1,2,0}_2\big((0,T)\times(0,R)\times (0,K)\big)\cap \mathcal{C}^{\frac{\alpha}{2},\alpha,\frac{\alpha}{2}}\big([0,T]\times [0,R]\times [0,K]\big)$, ($\alpha \in (0,1)$) is a generalized solution of the following parametric parabolic problem w.r.t to the variable $l\in [0,K]$:
\begin{eqnarray}
\label{eq:lemme-2.3}
\partial_t u(t,x,l)-a(t,x,l)\partial_x^2u(t,x,l)+b(t,x,l)\partial_xu(t,x,l)+c(t,x,l)u(t,x,l)-f(t,x,l)=0,
\end{eqnarray}
$ \Big($namely
$$\int_0^T\!\!\!\int_0^R\!\!\!\int_0^K \big(\partial u(t,x,l)-a(t,x,l)\partial_xu(t,x,l)+c(t,x,l)u(t,x,l)-f(t,x,l)\big)\phi(t,x,l)dtdxdl=0,$$
for any $\phi \in \mathcal{C}^{\infty}_c\big([0,T]\times[0,R]\times [0,K]\big)$ in the class of infinite differential functions with compact support strictly included in $(0,T)\times(0,R)\times (0,K) \Big)$.

Assume that the coefficient $a$ is elliptic:
$$ \forall (t,x,l)\in [0,T]\times[0,R]\times [0,K],~~a(t,x,l)\ge \underline{a}>0$$ and the coefficients and free terms $(a,b,c,f)$ have H\"{o}lder regularity in the class $\mathcal{C}^{\frac{\alpha}{2},\alpha,\frac{\alpha}{2}}\big((0,T)\times(0,R)\times (0,K)\big)$.

Then $u$ belongs to the class $\mathcal{C}^{1+\frac{\alpha}{2},2+\alpha,\frac{\alpha}{2}}\big((0,T)\times(0,R)\times (0,K)\big)$. 
\end{Lemma}
Remark that one could choose $\mathcal{C}^{\alpha,\beta,\gamma}\big((0,T)\times(0,R)\times (0,K)\big)$ as the class of regularity for the coefficients and free terms and show that the solution belongs to $\mathcal{C}^{1+\alpha,2+\beta,\gamma}\big((0,T)\times(0,R)\times (0,K)\big)$. For the reader's convenience, we have used here the classical terminology given in \cite{pde para}, pointing towards a possible extension of the well-posedness of similar systems as \eqref{eq : pde with l} to the quasi-linear framework.

%%%%%%%%%%%%%%%%%%%%%%%%%%%%%%%%%%%%%%%%%%%%%%%%%%%%%%%%%%%%%%%
%%%%%%%%%%%%DEBUT PREUVE LEMME 2.3 %%%%%%%%%%%%%%%%%%%%%%%%%%%%%%%%
%%%%%%%%%%%%%%%%%%%%%%%%%%%%%%%%%%%%%%%%%%%%%%%%%%%%%%%%%%%%%%%%

\begin{proof}

{\it Introduction: a short reminder for the interior regularity of weak solutions in the classical case}

When there is no dependency with respect to the variable $l$, results on the interior regularity of weak solutions of parabolic equation may be found for example in Theorem 12.1 III of \cite{pde para}. For the convenience of the reader, let us provide a short remainder of the main ideas that lead to the result of the classical case given in Theorem 12.1 III of \cite{pde para}, before getting into all the details of the proof. 

Let us recall that, if $u\in \mathcal{C}^{1+\frac{\alpha}{2},2+\alpha}\big((0,T)\times (0,R)\big)$ stands for some classical solution of the following parabolic problem
\begin{eqnarray}
\label{eq:para-gen}
&\partial_tu(t,x)-a(t,x)\partial_x^2u(t,x)+b(t,x)\partial_xu(t,x)+c(t,x)u(t,x)=f(t,x).
\end{eqnarray}
with smooth data given in the domain $(0,T)\times (0,R)$, then the classical Schauder's estimate reads:
\begin{equation}
\label{Schauder-estimate}
\|u\|_{\mathcal{C}^{1+\alpha,2+\alpha}(\mathcal{O} )}\leq C\big(\|f\|_{\mathcal{C}^{\frac{\alpha}{2},\alpha}(\mathcal{U}
)}+\|u\|_{L_{\infty}(\mathcal{U} )}\big),
\end{equation}
for any open smooth domains $\mathcal{O}\subset \subset \mathcal{U} \subset (0,T)\times (0,R)$ (see for e.g. \cite{Krylov-2} Section 10 Chapter 8).

Now, assume that $w$ is some continuous representative of a generalized solution of the last problem \eqref{eq:para-gen} and that $w$ belongs to the class $W^{1,2}\big((0,T)\times (0,R)\big)\cap L_{\infty}\big((0,T)\times (0,R)\big)$.

Associated to $w$, we introduce the following parabolic problem:
\begin{eqnarray*}
&\partial_tv^n(t,x)-a(t,x)\partial_x^2v^n(t,x)+b(t,x)\partial_xv^n(t,x)+c(t,x)v^n(t,x)=f(t,x),~~(t,x)\in \mathcal{U},\\
&v^n(t,x)=w^n(t,x),~~(t,x)\in  {}_L\partial\mathcal{U}
\end{eqnarray*}
where $_L\partial{\mathcal U}$ denotes the lateral boundary surface of ${\mathcal U}\subset (0,T)\times (0,R)$.

Here, in the setting of this classical parabolic Dirichlet problem, we have regularized the value of $w$ - for instance by standardly using a family of convolution kernels $(\xi^n)$ that tend weakly to the Dirac mass - in order to ensure the regularity of the solution $v^n$ at the boundary. Classical arguments guarantee that the solution $v^n$ is in the class $\mathcal{C}^{1+\frac{\alpha}{2},2+\alpha}(\mathcal{U})$. Then, the classical Schauder's estimates (written for $v^n$) combined with the use of Ascoli's theorem ensure that -- up to a subsequence -- the sequence $(v^n)$ converges locally uniformly to $v\in \mathcal{C}^{1+\frac{\alpha}{2},2+\alpha}(\mathcal{U})$ solution of
\begin{eqnarray*}
&\partial_tv(t,x)-a(t,x)\partial_x^2v(t,x)+b(t,x)\partial_xv(t,x)+c(t,x)v(t,x)=f(t,x),~~(t,x)\in \mathcal{U},\\
&v(t,x)=w(t,x),~~(t,x)\in  {}_L\partial\mathcal{U}.
\end{eqnarray*}
Moreover, by standard convergence arguments, $v$ itself satisfies a Schauder's type estimate of the form \eqref{Schauder-estimate}.

Now, if we know that there exists only one single generalized solution of this last problem, we have necessarily that $v$ is equal to $w$ on $\mathcal{U}$ (first almost everywhere and then everywhere). Due to the Schauder's estimate for $v$, for any open set $\mathcal{O}$ that is strictly included in the domain $(0,T)\times (0,R)$, $v$ has a regularity in the class $\mathcal{C}^{1+\frac{\alpha}{2},2+\alpha}(\mathcal{O})$ -- which yields in turn, by extending $\mathcal{O}$, that $w\in \mathcal{C}^{1+\frac{\alpha}{2},2+\alpha}\big((0,T)\times (0,R)\big)$. Hence, we conclude as in Theorem 12.1 III of \cite{pde para} for the classical interior regularity of the weak solution $w$.
\medskip
\medskip

In order to adapt these arguments to our setting, we see that the key point is to obtain a Schauder's type estimate for the parametric parabolic problem (that involves the variable $l$).

{\it Step 1. Proof of a Schauder's type estimate}

In the sequel and for the proof itself we consider the following data $(a,b,c,f)\in \mathcal{C}^{\frac{\alpha}{2},\alpha,\frac{\alpha}{2}}\big([0,T]\times [0,R]\times [0,K]\big)$ and we assume the ellipticity assumption for the leading coefficient $a$: $$\inf_{(t,x,l)} a(t,x,l) \ge \underline{a}>0.$$

We begin to prove that if $v\in  W^{1,2,0}_2\big((0,T)\times(0,R)\times (0,K)\big)\cap \mathcal{C}^{\frac{\alpha}{2},\alpha,\frac{\alpha}{2}}\big([0,T]\times [0,R]\times [0,K]\big)$ is such that for each $l\in [0,K]$, the map $v(\cdot, \cdot\,\,, l)$ is a {\it classical  $\mathcal{C}^{1+\frac{\alpha}{2},2+\alpha}\big((0,T)\times (0,R)\big)$ solution} of the following linear parabolic problem posed on the domain $(0,T)\times (0,R)$:
\begin{eqnarray}\label{eq : EDP parametr}
&\nonumber \partial_tv(t,x,l)-a(t,x,l)\partial_x^2v(t,x,l)+b(t,x,l)\partial_xv(t,x,l)+\\
&c(t,x,l)v(t,x,l)=f(t,x,l),~~\forall (t,x)\in (0,T)\times (0,R),
\end{eqnarray}
then the map $(t,x,l)\mapsto v(t,x,l)$ belongs to the class $\mathcal{C}^{1+\frac{\alpha}{2},2+\alpha,\frac{\alpha}{2}}\big((0,T)\times (0,R)\times (0,K)\big)$. Moreover, under the same assumptions, we provide a Schauder's type estimate: for any open set $\mathcal{O} \subset \subset (0,T)\times (0,R)\times (0,K)$ strictly separated from the boundary of the domain $(0,T)\times (0,R)\times (0,K)$ by a positive distance $\delta >0$, there exists a positive  constant $C>0$ depending only on the data $\big(\delta,\alpha,\underline{a},\|a\|_{\mathcal{C}^{\frac{\alpha}{2},\alpha,\frac{\alpha}{2}}},\|b\|_{\mathcal{C}^{\frac{\alpha}{2},\alpha,\frac{\alpha}{2}}},\|c\|_{\mathcal{C}^{\frac{\alpha}{2},\alpha,\frac{\alpha}{2}}},T\big)$, such that: 
\begin{eqnarray}\label{eq : Schauder estimate para}
\|v\|_{\mathcal{C}^{1+\alpha,2+\alpha,\frac{\alpha}{2}}(\mathcal{O} )}\leq C\big(\|f\|_{\mathcal{C}^{\frac{\alpha}{2},\alpha,\frac{\alpha}{2}}}+\|v\|_{\mathcal{C}^{\frac{\alpha}{2},\alpha,\frac{\alpha}{2}}}\big).
\end{eqnarray}

We prove \eqref{eq : Schauder estimate para}. For any $l \in [0,K]$ and any open set $\mathcal{V}$ included in the domain $[0,T]\times [0,R]\times [0,K]$, we will denote by:
$$\mathcal{V}_l:=\Big\{(t,x)\in [0,T]\times [0,R],~~(t,x,l)\in \mathcal{V}\Big\}$$
the $l$-level set of $\mathcal{V}$. Clearly, under this notation $\mathcal{V}_l$ is an open set of $[0,T]\times [0,R]$. 

Let $\mathcal{O}$ be then an open smooth domain strictly included in $(0,T)\times (0,R)\times (0,K)$ and let $\mathcal{W}$ a second open domain also
strictly included in $(0,T)\times (0,R)\times (0,K)$ such that $\mathcal{W}$ contains strictly $\mathcal{O}$ with $\mathcal{O}$ well separated from $\partial\mathcal{W}$ (we write in this case $\mathcal{O}\subset\subset \mathcal{W}$). The classical results given by the Schauder's estimates (see for e.g. \cite{Krylov-2}) lead to the existence of a constant $C(l)$ depending on the norm of the data $(a,b,c,f)(.,l)$ in the class $\mathcal{C}^{\frac{\alpha}{2},\alpha}\big([0,T]\times [0,R]\big)$, and $(\underline{a},\delta(l):=dist(\mathcal{O}_l,\mathcal{W}_l),T)$, such that:
\begin{eqnarray}\label{eq: Schauder estimat parametric}
\forall l\in[0,K],~~\|v(\cdot,l)\|_{\mathcal{C}^{1+\frac{\alpha}{2},2+\alpha}(\mathcal{O}_l)} \leq C(l)\big(\|f(\cdot, l)\|_{\mathcal{C}^{\frac{\alpha}{2},\alpha}(\mathcal{W}_l)}+\|v(\cdot,l)\|_{\mathcal{C}^0(\mathcal{W}_l)}\big).
\end{eqnarray}
 By giving a closer look to the estimations given in \cite{Krylov-2} Chapter 8 Section 10 (see also \cite{pde para} IV-\S 10, but the dependence of the constant w.r.t to the distance to the boundary is less apparent), we observe the non decreasing behavior of $C(l)$ with respect to $\delta(l):=dist(\mathcal{O}_l,\mathcal{W}_l)$; since
$$\forall l\in[0,K],~~\delta(l)\ge \gamma :=(dist(\mathcal{O},\mathcal{W}))>0,$$
we see that we are allowed to choose $C(l)=C>0$ independent of $l \in [0.K]$ in \eqref{eq: Schauder estimat parametric}.

Fix now $(l,q)\in [0,K]$ and denote by $v(\cdot, l)$ and $v(\cdot, q)$ two  solutions of the parametric problem \eqref{eq : EDP parametr}, with parameters $l$ and $q$. Remark that $v(\cdot, l)-v(\cdot, q)$ solves the following parabolic problem with unknown function $w$:

$\forall (t,x)\in (0,T)\times (0,R)$,
\begin{eqnarray}\label{eq : Diff EDP parametr}
&\nonumber \partial_t w(t,x)-a(t,x,l)\partial_x^2w(t,x)+b(t,x,l)\partial_xw(t,x)+c(t,x,l)w(t,x)=F(t,x,l,q)
\end{eqnarray}
where in the last equation the expression of the free term $F$ is given by:
\begin{eqnarray*}
&F(t,x,l,q)=-\big(a(t,x,q)-a(t,x,l)\big)\partial_x^2v(t,x,q) +\big(b(t,x,q)-b(t,x,l)\big)\partial_xv(t,x,q)+\\
&\big(c(t,x,q)-c(t,x,l)\big)v(t,x,q)+f(t,x,l)-f(t,x,q).
\end{eqnarray*}
Using the classical Schauder's estimate on the following open sets:
$$(\mathcal{O}_l\cup \mathcal{O}_q) \subset \subset (\mathcal{W}_l\cup \mathcal{W}_q),$$
we see that there exists a constant $C>0$, independent of $(l,q)$ such that:
\begin{eqnarray}\label{eq: Schauder dedoubler}
\|v(\cdot,l)-v(\cdot,q)\|_{\mathcal{C}^{1+\frac{\alpha}{2},2+\alpha}(\mathcal{O}_l\cup \mathcal{O}_q)} \leq C\big(\|F(\cdot, l,q)\|_{\mathcal{C}^{\frac{\alpha}{2},\alpha}(\mathcal{W}_l\cup \mathcal{W}_q)}+\|v(\cdot,l)-v(\cdot, q)\|_{\mathcal{C}^0(\mathcal{W}_l\cup \mathcal{W}_q)}\big).~~
\end{eqnarray}
Recall that by assumption, $v\in \mathcal{C}^{1+\frac{\alpha}{2},2+\alpha,\frac{\alpha}{2}}\big([0,T]\times [0,R]\times [0,K]\big)$. So that:
$$\|v(\cdot,l)-v(\cdot,q)\|_{\mathcal{C}^0(\mathcal{W}_l\cup \mathcal{W}_q)}\leq |l-q|^{\frac{\alpha}{2}}\|v\|_{\mathcal{C}^{\frac{\alpha}{2},\alpha,\frac{\alpha}{2}}([0,T]\times [0,R]\times [0,K])}.$$
From the assumptions of the data, remark also that
\begin{eqnarray*}
\|F(\cdot,l,q)\|_{\mathcal{C}^{\frac{\alpha}{2},\alpha}(\mathcal{W}_l\cup \mathcal{W}_q)}\leq |l-q|^{\frac{\alpha}{2}}C\|v(\cdot,q)\|_{\mathcal{C}^{1+\frac{\alpha}{2},2+\alpha}(\mathcal{W}_l\cup \mathcal{W}_q)},
\end{eqnarray*}
for some other constant $C>0$, depending only on the norm of the data.
But, applying the classical Schauder's estimates for $v(\cdot,q)$ on the domains
$$\mathcal{W}_l\cup \mathcal{W}_q\subset \subset [0,T]\times[0,R],$$
we have:
\[
\|v(\cdot,q)\|_{\mathcal{C}^{1+\frac{\alpha}{2},2+\alpha}(\mathcal{W}_l\cup \mathcal{W}_q)}\leq \bar{C}\Big(\|f\|_{\mathcal{C}^{\frac{\alpha}{2},\alpha,\frac{\alpha}{2}}([0,T]\times [0,R]\times [0,K])}+\|v\|_{\mathcal{C}^0([0,T]\times [0,R]\times [0,K])}\Big),
\]
where the constant $\bar{C}>0$ is independent of $q$. 

In turn, we obtain from \eqref{eq: Schauder dedoubler} that
\begin{align*}
&\|v(\cdot,l)-v(\cdot,q)\|_{\mathcal{C}^{1+\frac{\alpha}{2},2+\alpha}(\mathcal{O}_l\cup \mathcal{O}_q)}\\ 
&\leq M|l-q|^{\frac{\alpha}{2}}\Big(\|f\|_{\mathcal{C}^{\frac{\alpha}{2},\alpha,\frac{\alpha}{2}}([0,T]\times [0,R]\times [0,K])}+\|v\|_{\mathcal{C}^{\frac{\alpha}{2},\alpha,\frac{\alpha}{2}}([0,T]\times [0,R]\times [0,K])}\Big),
\end{align*}
where the constant $M>0$ is independent of $(l,q)$ and depends only on:
$$\big(\delta,\underline{a},\|a\|_{\mathcal{C}^{\frac{\alpha}{2},\alpha,\frac{\alpha}{2}}},\|b\|_{\mathcal{C}^{\frac{\alpha}{2},\alpha,\frac{\alpha}{2}}},\|c\|_{\mathcal{C}^{\frac{\alpha}{2},\alpha,\frac{\alpha}{2}}},T\big)$$
(Recall that $\delta >0$ is the strictly positive distance separating $\mathcal{O}$ from the boundary of $[0,T]\times [0,R]\times [0,K]$). 

We obtain therefore that:
\begin{eqnarray*}
\sup \left \{\displaystyle \frac{\|v(\cdot,l)-v(\cdot,q)\|_{\mathcal{C}^{1+\frac{\alpha}{2},2+\alpha}(\mathcal{O}_l\cup \mathcal{O}_q)}}{|l-q|^{\frac{\alpha}{2}}}~\Big |~(l,q)\in [0,K]^2,~~l\neq q\right \}\\
\hspace{1,4 cm}\leq M\Big(\|f\|_{\mathcal{C}^{\frac{\alpha}{2},\alpha,\frac{\alpha}{2}}([0,T]\times [0,R]\times [0,K])}+\|v\|_{\mathcal{C}^{\frac{\alpha}{2},\alpha,\frac{\alpha}{2}}([0,T]\times [0,R]\times [0,K])}\Big). 
\end{eqnarray*}
Now since
$$\mathcal{O}=\bigcup_{l \in [0,K]}\mathcal{O}_l,$$
we can conclude finally that $v \in \mathcal{C}^{1+\frac{\alpha}{2},2+\alpha,\frac{\alpha}{2}}\big(\mathcal{O}\big)$ and that there exists a positive constant $C>0$, depending only on the data $\big(\delta,\alpha,\underline{a},\|a\|_{\mathcal{C}^{\frac{\alpha}{2},\alpha,\frac{\alpha}{2}}},\|b\|_{\mathcal{C}^{\frac{\alpha}{2},\alpha,\frac{\alpha}{2}}},\|c\|_{\mathcal{C}^{\frac{\alpha}{2},\alpha,\frac{\alpha}{2}}},T\big)$, such that: 
\begin{eqnarray}
\label{eq:schauder-repetition}
\|v\|_{\mathcal{C}^{1+\alpha,2+\alpha,\frac{\alpha}{2}}(\mathcal{O} )}\leq C\big(\|f\|_{\mathcal{C}^{\frac{\alpha}{2},\alpha,\frac{\alpha}{2}}}+\|u\|_{\mathcal{C}^{\frac{\alpha}{2},\alpha,\frac{\alpha}{2}}}\big),
\end{eqnarray}
namely that \eqref{eq : Schauder estimate para} holds true. Thus we have proved \eqref{eq : Schauder estimate para} and the arbitrary choice of the open set $\mathcal{O}$ allows to state  finally that $v$ is in the class $\mathcal{C}^{1+\frac{\alpha}{2},2+\alpha,\frac{\alpha}{2}}\big((0,T)\times (0,R)\times (0,K)\big)$.
\medskip
\medskip

{\it Step 2. Proof of the interior regularity for weak solutions of \eqref{eq:lemme-2.3}}

We are now in position to adapt the arguments exposed in the introduction of the proof to our context, taking into account the dependency on $l$. 

Let $u \in W^{1,2,0}_2\big((0,T)\times(0,R)\times (0,K)\big)\cap \mathcal{C}^{\frac{\alpha}{2},\alpha,\frac{\alpha}{2}}\big([0,T]\times [0,R]\times [0,K]\big)$ a generalized solution of the parametric parabolic problem in the statement of the lemma.

In order to adapt the arguments exposed in the introduction of the proof to our context, we introduce naturally the following parabolic problem with parameter $l\in[0,K]$ posed on some connected open subset $\mathcal{U} = (s,s')\times (z,r)$ satisfying ${\mathcal U}\subset \subset (0,T)\times (0,R)$:
\begin{align}
\label{eq:regularization-u-v}
&\partial_tv^n(t,x,l)-a(t,x,l)\partial_x^2v^n(t,x,l)+b(t,x,l)\partial_xv^n(t,x,l)+c(t,x,l)v^n(t,x,l)\nonumber\\
&\hspace{7,4 cm}=f(t,x,l),~~(t,x)\in \mathcal{U},\nonumber\\
&v^n(t,x,l)=u^n(t,x,l),~~(t,x)\in {}_L\partial\mathcal{U},
\end{align}
with ${}_L\partial\mathcal{U} = \partial \mathcal{U}\setminus (\{s'\}\times [z,r])$ and where we have regularized at the lateral boundary ${}_L\partial\mathcal{U}$ the value of $u$ by convolution in the domain $(0,T)\times (0,R)$, more precisely:
\begin{align}\label{eq: regula u Schauder convo}
\nonumber &\forall (t,x)\in[0,T)\times[0,R],~~u^n(t,x,l)=\ds \int_{\R^2}\overline{u}(\theta,y,l)\rho_n(t-\theta,x-y)d\theta dy,~~l\in[0,K],\\
&\forall (\theta,y)\in \R^2,~~\overline{u}(\theta,y,l)=\begin{cases}
 u(\theta,y,l),~~\text{if}~~(\theta,y)\in [0,T]\times[0,R],\\
 0~~\text{else}
\end{cases},~~l\in[0,K],
\end{align}
and $\rho_n$ is a classical kernel of convolution in $\R^2$.
Note that classical estimates for solutions of standard parabolic problems ensure that for any fixed $l\in [0,K]$ there exists a unique classical smooth solution $v^n(\cdot, l)\in \mathcal{C}^{\frac{\alpha}{2},\alpha,\frac{\alpha}{2}}({\mathcal{U}})$ to the ladder Dirichlet boundary parabolic problem. Moreover, we have the uniform estimate
$\|v^n\|_{\mathcal{C}^{\frac{\alpha}{2},\alpha,\frac{\alpha}{2}}({\mathcal{U}}\times[0,K])} \leq C(n)$
where the constant $C(n)$ shows a dependency on $n$ only through the value of the parabolic boundary data $\|u^n\|_{\mathcal{C}^{\frac{\alpha}{2},\alpha,\frac{\alpha}{2}}({\mathcal{U}}\times[0,K])}$ (see for e.g. the classical results on {\it Solvability
of Problems 5.4' and 5.4} in \cite{pde para}). 
Now note that the regularized map $u^n(\cdot, l)$ belongs to the class $\mathcal{C}^\infty \big([0,T]\times[0,R]\big)$ and that it is always possible to perform the convolution in such a way that $u^n$ satisfies
\begin{eqnarray*}
\sup_{n \ge 0}\|u^n\|_{\mathcal{C}^{\frac{\alpha}{2},\alpha,\frac{\alpha}{2}}([0,T]\times[0,R]\times[0,K])} \leq \|u\|_{\mathcal{C}^{\frac{\alpha}{2},\alpha,\frac{\alpha}{2}}([0,T]\times[0,R]\times[0,K])}.
\end{eqnarray*}
Thus, the result of this discussion gives us assurance that there exists a finite constant $C:=\sup_{n\geq 0}C(n)<+\infty$ independent of $n$ such that
\begin{eqnarray}
\label{eq ; uniform born convo}
\sup\limits_{n\geq 0}\|v^n\|_{\mathcal{C}^{\frac{\alpha}{2},\alpha,\frac{\alpha}{2}}(\overline{{\mathcal{U}}}\times[0,K])} \leq C.
\end{eqnarray}

Let $\mathcal{O}:= (t_1,t_2)\times (r_1,r_2)\times (l_1,l_2)\subset \subset \mathcal{U}\times (0,K)$.

With the same arguments used to prove the Schauder's estimates \eqref{eq:schauder-repetition}, it is not hard to check that the map $(t,x,l)\mapsto v^n(t,x,l)$ has a regularity in the class $\mathcal{C}^{1+\frac{\alpha}{2},2+\alpha,\frac{\alpha}{2}}\big(\mathcal{O}\big)$ and satisfies itself a Schauder's type estimate 
\begin{eqnarray}
\label{eq:schauder-repetition-2}
\|v^n\|_{\mathcal{C}^{1+\alpha,2+\alpha,\frac{\alpha}{2}}(\mathcal{O} )}\leq M\big(\|f\|_{\mathcal{C}^{\frac{\alpha}{2},\alpha,\frac{\alpha}{2}}(\overline{\mathcal{U}}\times [0,K])}+\|v^n\|_{\mathcal{C}^{\frac{\alpha}{2},\alpha,\frac{\alpha}{2}}(\overline{\mathcal{U}}\times [0,K])}\big),
\end{eqnarray}
where as before, $M>0$ stands for some constant depending only on the data.

Hence, just like for the remainder in the introduction of the proof, combining \eqref{eq ; uniform born convo} together with Schauder's estimates \eqref{eq:schauder-repetition-2} allows to apply Ascoli's theorem: up to a subsequence $(v^{n_k})$ converges uniformly in the class $\mathcal{C}^{1+\frac{\alpha}{2},2+\alpha,\frac{\alpha}{2}}\big(\mathcal{O}\big)$ to a function $v$, which solves 
\begin{align}
\label{eq ; sol limite }
\partial_tv(t,x,l)-a(t,x,l)\partial_x^2v(t,x,l)+b(t,x,l)\partial_xv(t,x,l)+c(t,x,l)v(t,x,l)=\nonumber\\
f(t,x,l),~~(t,x)\in (t_1,t_2)\times(r_1,r_2),
%%&v(t,x,l)=u(t,x.l),~~(t,x)\in  \partial \Big((t_1,t_2)\times(r_1,r_2)\Big)
\end{align}
for any $l \in (l_1,l_2).$

This solution $v$ is also a generalized solution in the sense that:
\begin{eqnarray}
\label{eq:generalized-solution-O}
\int\!\!\int\!\!\int_{\mathcal{O}}\big(\partial_t v(t,x,l)-a(t,x,l)\partial_x^2v(t,x,l)+b(t,x,l)\partial_xv(t,x,l)+
\nonumber\\
c(t,x,l)v(t,x,l)-f(t,x,l)\big)\phi(t,x,l)dtdxdl=0,
\end{eqnarray}
for any $\phi \in \mathcal{C}^{\infty}_c\big(\overline{\mathcal{O}}\big)$ (the class of infinite differential function with compact support strictly included in $\mathcal{O}$). 

Formally speaking, the previous limit $v$ depends on the set $\mathcal{O}$:  in order to emphasize its dependence on ${\mathcal{O}}$ let us denote it $v_{\mathcal{O}}$ for a moment. 
Since $\mathcal{O}$ may be arbitrarily taken in $\mathcal{U}$, we may consider $(\mathcal{O}_p)$ an increasing sequence of open sets converging to $\mathcal{U}$ as $p$ tends to infinity. (Recall that $\mathcal{U} = (s,s')\times (z,r)$), so one may define for instance:
\begin{eqnarray*}
 \forall p \text{ large enough },~~\mathcal{O}_p=(s+\frac{1}{p},s'-\frac{1}{p})\times (z+\frac{1}{p},r-\frac{1}{p}). 
\end{eqnarray*}
The preceding shows that we can attach to this sequence a doubly indexed subsequence 

$\pare{v^{n_k^{(p)}}}_{(k,p)\in \N^\ast\times \N^\ast}$, which satisfies that for any $p$, $\lim_{k\rightarrow + \infty}v^{n_k^{(p)}} = v_{\mathcal{O}_p}$ locally uniformly in $\mathcal{C}^{1+\frac{\alpha}{2},2+\alpha,\frac{\alpha}{2}}\big(\mathcal{O}_p\big)$ with $v_{\mathcal{O}_p}$ verifying \eqref{eq:generalized-solution-O} (with $\mathcal{O}_p$ in place of $\mathcal{O}$) and constructed inductively such that the subsequence $\pare{v^{n_k^{(p+1)}}}_{k\in \N^\ast}$ is itself a subsequence of $\pare{v^{n_k^{(p)}}}_{k\in \N^\ast}$.
Proceeding to a diagonal extraction, we now consider $\pare{v^{n_p^{(p)}}}_{p\in \N^\ast}$. By construction, for any $q\in \N^\ast$,  $\pare{v^{n_p^{(p)}}}_{p\geq q}$ is a subsequence of $\pare{v^{n_k^{(q)}}}_{k\in \N^\ast}$ and as such, the subsequence $\pare{v^{n_p^{(p)}}}_{p\in \N^\ast}$  converges to $v_{\mathcal{O}_q}$ locally uniformly in $\mathcal{C}^{1+\frac{\alpha}{2},2+\alpha,\frac{\alpha}{2}}\big(\mathcal{O}_q\big)$. Since the ladder holds true for any $q$, the family $(v_{\mathcal{O}_q})_q$ has to be consistent and our subsequence $\pare{v^{n_p^{(p)}}}_{p\in \N^\ast}$ converges locally uniformly in the class $\mathcal{C}^{1+\frac{\alpha}{2},2+\alpha,\frac{\alpha}{2}}\big(\mathcal{O}\big)$ to some function $v\in \mathcal{C}^{1+\frac{\alpha}{2},2+\alpha,\frac{\alpha}{2}}\big(\mathcal{O}\big)$ which satisfies \eqref{eq:generalized-solution-O}. This convergence holds for any arbitrary $\mathcal{O} =(t_1,t_2)\times (r_1,r_2)\times (l_1,l_2)\subset \subset \mathcal{U}\times (0,K)$ and we get in fact that $v\in \mathcal{C}^{1+\frac{\alpha}{2},2+\alpha,\frac{\alpha}{2}}\big(\mathcal{U}\times (0,K)\big)$.

Moreover, for any $l\in (0,K)$, the convolution regularization $(u^n(.,l))$ defined in \eqref{eq: regula u Schauder convo}, converges pointwise to $u(.,l)$ (recall that ${\mathcal U}\subset \subset (0,T)\times (0,R)).$ In turn, \eqref{eq:regularization-u-v} shows that $(v^{n^{(p)}_p}(.,l))_p$ converges to $u(.,l)$ on the lateral surface $_{L}\partial U$ ensuring that for all $l\in (0,K)$:~
$$
v(t,x,l) = u(t,x,l),\;\;\;\forall (t,x)\in {}_{L}\partial U.
$$

We now proceed to show that
$$\forall (t,x,l)\in \mathcal{U}\times (0,K),~~u(t,x,l)=v(t,x,l),$$
which will in turn imply finally that $u \in \mathcal{C}^{1+\frac{\alpha}{2},2+\alpha,\frac{\alpha}{2}}\big(\mathcal{U}\times (0,K)\big)$.

Denote by
$$\mathcal{K}=\{l_1,\ldots l_n,\ldots\}$$
a dense countable subset of $(0,K)$. Fix $\overline{l}\in \mathcal{K}$ and let $\{\phi_{\varepsilon}\in \mathcal{C}^\infty\big([0,K]\big), \varepsilon >0\}$ denote a family of smooth functions converging in the sense of distribution to the Dirac distribution $\delta_{\overline{l}}$ as $\varepsilon \searrow 0$. 

Let $\phi\in \mathcal{C}^\infty_c\big((s,s')\times(\ell,r)\big)$ and $\varepsilon>0$. We have  
\begin{eqnarray*}
&\displaystyle \int\!\!\int\!\!\int_{\mathcal{O}} \big(\partial_t v(t,x,l)-a(t,x,l)\partial_x^2v(t,x,l)+b(t,x,l)\partial_xv(t,x,l)+
\\
&c(t,x,l)v(t,x,l)-f(t,x,l)\big)\phi_\varepsilon(l)\phi(t,x)dtdxdl=0,
\end{eqnarray*}
for arbitrary $\mathcal{O}$. By letting $\varepsilon \searrow 0$
\begin{eqnarray*}
&\displaystyle \int_{s}^{s'}\!\!\int_{\ell}^{r}\big(\partial_t v(t,x,\overline{l})-a(t,x,\overline{l})\partial_x^2v(t,x,\overline{l})+b(t,x,\overline{l})\partial_xv(t,x,\overline{l})+
\\
&c(t,x,\overline{l})v(t,x,\overline{l})-f(t,x,\overline{l})\big)\phi(t,x)dtdx=0.
\end{eqnarray*}
With the same arguments, we have also:
\begin{eqnarray*}
&\displaystyle \int_{s}^{s'}\!\!\int_{\ell}^{r}\big(\partial_t u(t,x,\overline{l})-a(t,x,\overline{l})\partial_x^2u(t,x,\overline{l})+b(t,x,\overline{l})\partial_xu(t,x,\overline{l})+
\\
&c(t,x,\overline{l})u(t,x,\overline{l})-f(t,x,\overline{l})\big)\phi(t,x)dtdx=0.
\end{eqnarray*}
Therefore $u(\cdot,\overline{l})$ and $v(\cdot,\overline{l})$ are two classical weak solutions of the same parabolic problem, on the domain $\mathcal{U} = (s,s')\times (\ell,r)$ possessing the same boundary conditions on $_{L}\partial U$. From the weak uniqueness in the class $W^{1,2}_2\big((s,s')\times (\ell,r)\big)$ that follows from our assumptions (see for instance the weak uniqueness result stated in \cite{pde para} Theorem 9.1 Chapiter IV), we deduce that:
$$u(t,x,\overline{l})=v(t,x,\overline{l}),~~dt\otimes dx ,~~\text{almost everywhere in }[s,s']\times [\ell,r].$$
This implies the existence of some negligible set
$\mathcal{N}_{\overline{l}}\subset\subset [s,s']\times [\ell,r],$
such that:
$$dt\otimes dx \big(\mathcal{N}_{\overline{l}}\big)=0,$$
and:
$$\forall (t,x)\in [s,s']\times [\ell,r]\setminus \mathcal{N}_{\overline{l}},~~u(t,x,\overline{l})=v(t,x,\overline{l}).$$
Set:
$$\mathcal{N}=\bigcup_{l_n \in \mathcal{K}}\mathcal{N}_{\overline{l}_n}.$$
We can conclude that:
$$\forall \overline{l}\in \mathcal{K},~~\forall (t,x)\in [s,s']\times [\ell,r]\setminus \mathcal{N},~~u(t,x,\overline{l})=v(t,x,\overline{l}).$$
Using now the key assumption that $u\in \mathcal{C}^{\frac{\alpha}{2},\alpha,\frac{\alpha}{2}}\big([0,T]\times[0,R]\times [0,K]\big)$, we can conclude by the continuity of both $u$ and $v$ with respect to the variables $(t,x)$ that:
$$\forall \overline{l}\in \mathcal{K},~~\forall (t,x)\in [s,s']\times [\ell,r],~~u(t,x,\overline{l})=v(t,x,\overline{l}).$$
The density of $\mathcal{K}$ in $(0,K)$ and the continuity of both $u$ and $v$ with respect to the variable $l$ yield
$$\forall (t,x,l)\in [s,s']\times [\ell,r]\times (0,K),~~u(t,x,l)=v(t,x,l),$$
ensuring that $u\in \mathcal{C}^{1+\frac{\alpha}{2},2+\alpha,\frac{\alpha}{2}}\big(\mathcal{U}\times (0,K)\big)$.

Now observe that $\mathcal{U}$ has been arbitrarily taken in $(0,T)\times (0,R)$, so that in fact $$u\in \mathcal{C}^{1+\frac{\alpha}{2},2+\alpha,\frac{\alpha}{2}}\big((0,T)\times (0,R)\times (0,K)\big),$$ which concludes the proof of the lemma.

\end{proof}

%%%%%%%%%%%%%%%%%%%%%%%%%%%%%%%%%%%%%%%%%%%%%%%%%%%%%%%%%%%%%%%
%%%%%%%%%%%%FIN PREUVE LEMME 2.3%%%%%%%%%%%%%%%%%%%%%%%%%%%%%%%%
%%%%%%%%%%%%%%%%%%%%%%%%%%%%%%%%%%%%%%%%%%%%%%%%%%%%%%%%%%%%%%%%

\subsection{Assumptions and main results}\label{subsec: Main results}

\medskip

In this subsection, we introduce the data involved in the PDE system \eqref{eq : pde with l} with the required assumptions and we state our main Theorem \ref{th : exis para with l}. Next, we proceed to the proof of a comparison theorem for the PDE system \eqref{eq : pde with l}. 
\subsubsection{Existence and uniqueness for a Parabolic PDE with Kirchhoff's local time condition}
\medskip

\textbf{For the rest of these notes, we fix $\alpha \in(0,1)$}.

We introduce the following data:
$$\mathfrak{D}:=\begin{cases}
\Big(a_i \in W^{1,\infty}\big([0,T]\times [0,R]\times [0,K],\R_+\big)\Big)_{i\in[\![1,I]\!]}\\
\Big(b_i \in W^{1,\infty}\big([0,T]\times [0,R]\times [0,K],\R\big)\Big)_{i\in[\![1,I]\!]}\\
\Big(c_i \in W^{1,\infty}\big([0,T]\times [0,R]\times [0,K],\R\big)\Big)_{i\in[\![1,I]\!]}\\
\Big(f_i \in W^{1,\infty}\big([0,T]\times [0,R]\times [0,K],\R\big)\Big)_{i\in[\![1,I]\!]}\\
\Big(\alpha_i \in W^{1,\infty}\big([0,T]\times [0,K],\R_+\big)\Big)_{i\in[\![1,I]\!]}\\ 
r\in W^{1,\infty}\big([0,T]\times [0,K],\R_+\big)\\
\phi \in W^{1,\infty}\big([0,T]\times [0,K],\R\big)\\
\psi \in \mathcal{C}^{0,1}\big([0,T]\times \mathcal{N}_R\big)\cap \mathcal{C}^{1,2}
\big((0,T)\times \overset{\circ}{\mathcal{N}^*_R}\big)\hspace{0,4 cm} \text{(with ${\mathcal{N}^*_R}$ defined in \eqref{eq:reseau-etoile})}\\
g\in \mathfrak{Lip}_{\{0\}}^{2,0}\big(\Omega\big)\hspace{0,4 cm}\text{(where the definition of $\mathfrak{Lip}_{\{0\}}^{2,0}\big(\Omega\big)$ is given after Definition \ref{definition-espace-solutions})}.
\end{cases}.
$$
We assume that the data $\mathfrak{D}$ satisfies the following assumption:
$$\textbf{Assumption } (\mathcal{H})$$
\noindent 
a) Ellipticity condition for the terms $\big(a_i,\alpha_i\big)_{i\in [\![1,I]\!]}$:
\begin{align*}
&(i)~~\exists~\underline{a}>0,~~\forall i \in [\![1,I]\!], ~~\forall (t,x,l)\in [0,T]\times [0,R]\times[0,K],~~a_i(t,x,l) \ge \underline{a},\\
&(ii)~~\exists~\underline{\alpha}>0,~~\forall i \in [\![1,I]\!], ~~\forall (t,l)\in [0,T]\times[0,K],~~\alpha_i(t,l) \ge \underline{\alpha},
\end{align*}
\noindent b) Compatibility conditions at the boundaries:
\begin{align*}
&(i)~~\partial_lg(0,l)+\sum_{i=1}^I\alpha_i(0,l)\partial_xg_i(0,l)-r(0,l)g(0,l)=\phi(0,l),~~l\in[0,K),\\
&(ii)~~\partial_xg(R,l)=0,~~l\in [0,K),\\
&(iii)~~g(x,K)=\psi_i(0,x),~~x\in [0,R].
\end{align*}
We state the main central result of this work, which asserts the unique solvability of the parabolic linear PDE system \eqref{eq : pde with l} posed on $\mathcal{N}_R$ and having a dynamical '{\it local-time Kirchhoff's boundary condition}' at the junction point $\{0\}$.

Remember the definition of $\mathfrak{C}^{1+\frac{\alpha}{2},2+\alpha,\frac{\alpha}{2}}_{\{0\}} \big(\Omega_T\big)$ given in Definition \ref{definition-espace-solutions} and the definition of $\mathfrak{C}^{1,2,0}_{\{0\}} \big(\Omega_T\big)$ given right after it.
\begin{Theorem}\label{th : exis para with l}
Assume that the data $\mathfrak{D}$ satisfies assumption $(\mathcal{H})$. Then, the system \eqref{eq : pde with l} is uniquely solvable in the class $\mathfrak{C}^{1+\frac{\alpha}{2},2+\alpha,\frac{\alpha}{2}}_{\{0\}} \big(\Omega_T\big)$. 
\end{Theorem}
Next, we give the definitions of super and sub solutions for the system \eqref{eq : pde with l}, and we prove a comparison Theorem.
\begin{Definition}\label{def : sur/sous solutions}
We say that $u\in\mathfrak{C}^{1,2,0}_{\{0\}} \big(\Omega_T\big)$ is a super solution (resp. sub solution) of the PDE system \eqref{eq : pde with l} if:
\begin{eqnarray*}
\begin{cases}\partial_tu_i(t,x,l)-a_i(t,x,l)\partial_x^2u_i(t,x,l)
+b_i(t,x,l)\partial_xu_i(t,x,l)+\\
\;\;c_i(t,x,l)u_i(t,x,l)-f_i(t,x,l)\ge 0,~~(\text{resp. }\leq 0),~~(t,x,l)\in (0,T)\times (0,R)\times(0,K)\\
\partial_lu(t,0,l)+\displaystyle \sum_{i=1}^I \alpha_i(t,l)\partial_xu_i(t,0,l)-r(t,l)u(t,0,l)-\phi(t,l)\leq 0,~~ (\text{resp. }\ge 0),\\
\hspace{10,4 cm}(t,l)\in(0,T)\times(0,K)\\
\partial_xu_i(t,R,l)\ge 0,~~ (\text{resp. }\leq 0),~~ (t,l)\in (0,T)\times(0,K)
\end{cases}.
\end{eqnarray*}
\end{Definition}
\begin{Theorem}\label{th : para comparison th with l} Comparison Theorem.\\
Assume that the data $\mathfrak{D}$ satisfies assumptions $(\mathcal{H})$. Let $u\in \mathfrak{C}^{1,2,0}_{\{0\}} \big(\Omega_T\big)$ (resp. $v\in \mathfrak{C}^{1,2,0}_{\{0\}} \big(\Omega_T\big)$) a super solution (resp. a sub solution) of system \eqref{eq : pde with l} satisfying that for all $\big(t,(x,i),l\big)\in [0,T]\times \mathcal{N}_R\times [0,K]$:
$$u_i(0,x,l)\ge v_i(0,x,l)~~\text{and}~~u_i(t,x,K)\ge v_i(t,x,K).$$
Then, for all $\big(t,(x,i),l\big)\in[0,T]\times \mathcal{N}_R\times [0,K]$: $$u_i(t,x,l)\ge v_i(t,x,l).$$
\end{Theorem}
\begin{proof}

Let $\lambda(K,R)=\lambda>C(K,R)$, where $C(K,R)$ is some constant whose expression will be given later (the definition of $C(K,R)$ is given in \eqref{eq : expr C(K,R)}).
First fix $s\in(0,T)$ and $\ell\in (0,K)$. We argue by contradiction and we assume that
\begin{eqnarray*}
\sup\Big\{\exp\big(-\lambda t-\frac{(x-\ds\frac{R}{2})^2}{2}\big)\Big(v_{i}(t,x,l)-u_{i}(t,x,l)\Big)\Big\}>0
\end{eqnarray*}
 where the supremum is taken over all $(i,t,x,l)\in [\![1,I]\!]\times [0,s]\times[0,R]\times[\ell,K]$ (under the convention $\sup(\emptyset) = 0$).

For any $(i,t,x,l)\in [\![1,I]\!]\times [0,s]\times[0,R]\times[\ell,K]$ let us set:
$$\theta_{\lambda,i}(t,x,l):=\exp\big(-\lambda t-\frac{(x-\ds\frac{R}{2})^2}{2}\big)\Big(v_{i}(t,x,l)-u_{i}(t,x,l)\Big).$$
Using the continuity and the terminal boundary conditions satisfied by $u$ and $v$ in the assumptions of the theorem, the supremum above is then reached at a point $(i_0,t_0,x_0,l_0)\in [\![1,I]\!]\times (0,s]\times [0,R]\times [\ell,K)$ that satisfies:
\begin{eqnarray}\label{eq : compa u et v absurde}
v_{i_0}(t_0,x_0,l_0)-u_{i_0}(t_0,x_0,l_0)>0.
\end{eqnarray}

Assume first that $x_0=R$. From the regularity of $x\mapsto \theta_{\lambda, i_0}(t_0, x, l_0)$ up to  $x_0=R$ induced by our assumptions, we get $$\partial_x\theta_{\lambda,i_0}(t_0,x_0,l_0)\ge 0.$$ 

Hence:
$$\partial_xv_{i_0}(t_0,x_0,l_0)\ge \partial_xu_{i_0}(t_0,x_0,l_0)+\ds \frac{R}{2}\big(v_{i_0}(t_0,x_0,l_0)-u_{i_0}(t_0,x_0,l_0))\big).$$
Using the fact that $u$ is a super solution, whereas $v$ is a sub solution, we obtain from the boundary inequalities at $x_0=R$ that
$$0\ge \frac{R}{2}\big(v_{i_0}(t_0,x_0,l_0)-u_{i_0}(t_0,x_0,l_0))\big)>0,$$
and hence a contradiction.

Suppose now that $x_0\in (0,R)$, then the optimality conditions in the directional derivatives with respect to the variables $t$ and $x$ imply:
\begin{eqnarray*}
&\partial_t\theta_{\lambda,i_0}(t_0,x_0,l_0)\ge 0\;\;\;(=0 \text{ if } t_0\neq s),~~\partial_x\theta_{\lambda,i_0}(t_0,x_0,l_0)=0,~~\partial^2_x\theta_{\lambda,i_0}(t_0,x_0,l_0)\leq 0.
\end{eqnarray*}
So that since $(t_0,x_0,l_0)\in (0,s]\times [0,R]\times [\ell, K)$,
\begin{eqnarray*}
&-2(x_0-\frac{R}{2})\big(\partial_x v_{i_0}(t_0,x_0,l_0)-\partial_x u_{i_0}(t_0,x_0,l_0)\big)\\
&+[(x_0-\frac{R}{2})-1]^2\big(v_{i_0}(t_0,x_0,l_0)-u_{i_0}(t_0,x_0,l_0)\big)+\partial_x^2v_{i_0}(t_0,x_0,l_0)-\partial_x^2u_{i_0}(t_0,x_0,l_0)\leq 0.
\end{eqnarray*}
Because $\partial_x\theta_{\lambda,i_0}(t_0,x_0,l_0)=0$ we get
\begin{eqnarray*}
-[1+(x_0-\frac{R}{2})^2]\big(v_{i_0}(t_0,x_0,l_0)-u_{i_0}(t_0,x_0,l_0)\big)+\partial_x^2v_{i_0}(t_0,x_0,l_0)-\partial_x^2u_{i_0}(t_0,x_0,l_0)\leq 0.
\end{eqnarray*}

Using now the fact that $v$ is a sub solution while $u$ is a super solution of \eqref{eq : pde with l} on the ray $\mathcal{R}_{i_0}$, we obtain using the positivity assumption on the coefficient $a_{i_0}$ $(\mathcal{H})$ a)-(i):
\begin{eqnarray*}
&0 \ge (\partial_tv_{i_0}-\partial_tu_{i_0})(t_0,x_0,l_0) -a_{i_0}(t_0,x_0,l_0)\big((\partial_x^2v_{i_0}-\partial_x^2u_{i_0})(t_0,x_0,l_0)\big)\\&
+b_{i_0}(t_0,x_0,l_0)\big((\partial_xv_{i_0}-\partial_xu_{i_0})(t_0,x_0,l_0)\big)+c_{i_0}(t_0,x_0,l_0)\big((v_{i_0}-u_{i_0})(t_0,x_0,l_0)\big)
\\&\geq \lambda \big((v_{i_0}-u_{i_0})(t_0,x_0,l_0)\big)-a_{i_0}(t_0,x_0,l_0)[1+(x_0-\frac{R}{2})^2]\big((v_{i_0}-u_{i_0})(t_0,x_0,l_0)\big)\\
&+b_{i_0}(t_0,x_0,l_0)(x_0-\frac{R}{2})\big((v_{i_0}-u_{i_0})(t_0,x_0,l_0)\big)+c_{i_0}(t_0,x_0,l_0)\big((v_{i_0}-u_{i_0})(t_0,x_0,l_0)\big)\\
&\ge\big(\lambda-C(K,R)\big)\big(v_{i_0}(t_0,x_0,l_0)-u_{i_0}(t_0,x_0,l_0)\big),
\end{eqnarray*}
where:
\begin{eqnarray}\label{eq : expr C(K,R)}
&\nonumber C(K,R):=\sup \Big\{-a_{i}(t,x,l)[1+(x-\frac{R}{2})^2]
+b_{i}(t,x,l)(x-\frac{R}{2})+c_{i}(t,x,l),~~\\
&(t,(x,i),l)\in[0,T]\times \mathcal{N}_R\times [0,K]\Big\}.
\end{eqnarray}

Therefore, using \eqref{eq : compa u et v absurde} and the defining property for $\lambda$, we obtain a contradiction.

Now assume that $x_0=0$. 
Since for all $(i,j)\in[\![1,I]\!]$, $u_i(t_0,0,l_0)=u_j(t_0,0,l_0)=u(t_0,0,l_0)$ and $v_i(t_0,0,l_0)=v_j(t_0,0,l_0)=u(t_0,0,l_0)$, using the regularity with respect to the variable $l$ of both $u$ and $v$ at $\{0\}$ (coming from condition $iv)$ in the definition of $\mathfrak{C}^{1,2,0}_{\{0\}} \big(\Omega_T\big)$), we obtain:
\begin{eqnarray*}
\exp\big(\!-\lambda t_0-\ds\frac{R^2}{8}\big)\Big((\partial_lv-\partial_lu)(t_0,0,l_0)\Big)\leq 0~~\text{so that}~~\partial_lv(t_0,0,l_0)-\partial_lu(t_0,0,l_0)\leq 0.
\end{eqnarray*}

By definition of $(t_0,x_0,l_0) = (t_0,0,l_0)$, we have also that for all $i\in[\![1,I]\!]$ and $h\in[0,R]$:
\begin{eqnarray*}
&\exp(\lambda t_0-\ds\frac{R^2}{8}\big)\Big(v(t_0,0,l_0)-u(t_0,0,l_0)\Big) \ge \exp(\lambda t_0-\ds\frac{(h-\ds \frac{R}{2})^2}{2}\big)\Big(v_i(t_0,h,l_0)-u_i(t_0,h,l_0)\Big).
\end{eqnarray*}

Therefore, applying a first order Taylor expansion with respect to the variable $x$ in the neighborhood of the junction point $\{0\}$, for any $i\in [\![1,I]\!]$
we get that \begin{eqnarray*}
& v(t_0,0,l_0)-u(t_0,0,l_0)~~\ge~~ v(t_0,0,l_0)-u(t_0,0,l_0)+\\
&h\Big(\ds\frac{R}{2}\big(v(t_0,0,l_0)-u(t_0,0,l_0)\big)+
\partial_xv_i(t_0,0,l_0)-\partial_xu_i(t_0,0,l_0)\Big)~+~h\varepsilon_i(h),
\end{eqnarray*}
with $\lim_{h\to 0}\varepsilon_i(h)=0$. Thus,
\begin{eqnarray}\label{eq : ine grad}
\forall i\in[\![1,I]\!],~~\partial_xv_i(t_0,0,l_0) ~~\leq ~~\partial_xu_i(t_0,0,l_0)-\ds\frac{R}{2}\big(v(t_0,0,l_0)-u(t_0,0,l_0)\big).
\end{eqnarray}
Now, using the ellipticity assumption on the coefficients $(\alpha_i)_{1\leq i \leq I}$ ($\mathcal{H}$) a)-(ii)), observing that the coefficient $r$ is non negative and also the fact that $v$ is a sub solution while $u$ is a super solution of \eqref{eq : pde with l} at $\{0\}$, we obtain:
\begin{eqnarray*}
&0\ge \partial_lu(t_0,0,l_0)-\partial_lv(t_0,0,l_0)+\displaystyle \sum_{i=1}^I \alpha_i(t,l)\big(\partial_xu_i(t_0,0,l_0)-\partial_xv_i(t_0,0,l_0)\big)\\
&-r(t_0,l_0)\big(u(t_0,0,l_0)-v(t_0,0,l_0)\big) \ge I \underline{\alpha}\ds\frac{R}{2}\big(v(t_0,0,l_0)-u(t_0,0,l_0)\big)>0,
\end{eqnarray*}
which yields a contradiction.

All cases lead to contradictions, resulting in the fact that for all $0\leq s<T$, for all $(t,(x,i),l)\in[0,s]\times\mathcal{N}_R\times[\ell,K]$:
\begin{eqnarray*}
\exp\big(-\lambda t-\frac{(x-\ds\frac{R}{2})^2}{2}\big)\Big(v_{i}(t,x,l)-u_{i}(t,x,l)\Big)~~\leq~~0.
\end{eqnarray*}
Using the continuity of $u$ and $v$ w.r.t variables $(t,l)$, we deduce finally that for all $(t,(x,i),l)\in[0,T]\times\mathcal{N}_R\times[0,K]$:
\begin{eqnarray*}
v_{i}(t,x,l) \leq u_{i}(t,x,l).
\end{eqnarray*}
\end{proof}

%%%%%%%%%%%%%%%%%%%%%%%%%%%%%%%%%%%%%%%%%%%%%%%%%%%%%%%
%%%%%%%%%%%%%%%%%%%%%%%%%%%%%%%%%%%%%%%%%%%%%%%%%%%%%%%
%%%%%%%%%%%%%%%%%%%%%%%%%%%%%%%%%%%%%%%%%%%%%%%%%%%%%%%%
%%%%%%%%%%%%%%%%%%%%%%%%%%%%%%%%%%%
%%%%%%%%%%%%%%%%%%%%%%%%%%%%%%%%%%%%%%%%%%%%%%%%%%%%%%%
%%%%%%%%%%%%%%%%%%%%%%%%%%%%%%%%%%%%%%%%%%%%%%%%%%%%%%
%%%%%%%%%%%%%%%%%%%%%%%%%%%%%%%%%%%%%%%%%%%%%%%%%%%%%%
%%%%%%%%%%%%%%%%%%%%%%%%%%%%%%%%%%%%%%%%%%%%%%%%%%%%%%

\section{The main result obtained by Von Below in \cite{Von Below} revisited.}
\label{sec : Von Below revisited}
Up to our knowledge, the first result obtained for linear parabolic equations posed on networks - involving Kirchhoff's type boundary conditions at the vertices -
was obtained by Von Below in \cite{Von Below}. Essentially, it is proved in this paper that a linear parabolic problem on a network with a classical Kirchhoff's boundary conditions at the junctions vertices is well-posed. The proof consists in increasing the dimension of the problem and showing that the linear parabolic problem with Kirchhoff's condition is equivalent to a well-posed initial boundary value problem but for a higher dimensional standard parabolic linear system, where the boundary conditions are transformed in such a way that the classical results on linear parabolic equations systems with Neumann boundary conditions may be applied (namely, the conditions of Chapter VII in \cite{pde para}). 

The result of \cite{Von Below} on existence is stated under natural smoothness assumptions for the coefficients, but also under second order strong compatibility conditions for the initial data $g$ at the junction point $\{0\}$, (see equation 8.3 in \cite{Von Below}). 
To be more specific, recall that if $u$ is a weak solution in the Sobolev class $W^{1,2}$ (for infinitely differentiable test functions vanishing at the lateral surface of the domain $[0,T]\times[0,R]$) of some classical linear parabolic problem with coefficients $(a,b,c,f)$ that belong to $\mathcal{C}^{1+\frac{\alpha}{2},2+\alpha}\big((0,T)\times (0,R)\big)$, then $u\in \mathcal{C}^{1+\frac{\alpha}{2},2+\alpha}\big((0,T)\times (0,R)\big)$ (see Theorem III.12.2 given in \cite{pde para}). Moreover, if the initial data of this classical parabolic problem satisfies the classical compatibility conditions at the lateral surface of $[0,T]\times[0,R]$, then the solution $u$ will have a regularity in the class $\mathcal{C}^{1+\frac{\alpha}{2},2+\alpha}\big([0,T]\times [0,R]\big)$, namely in the whole domain up to the parabolic boundary. Hence, considering conditions at some vertex with a similar look as various classical boundary conditions coming from several directions, leads naturally to the strengthening of the compatibility conditions at the vertex: this is what is imposed in \cite{Von Below} as the key in order to guarantee in each directions of the edges, there exists a regular spatial derivative at the vertices. 

In this section -- as was obtained by I.Ohavi in \cite{Ohavi PDE} in a slightly different context (Quasi-linear parabolic equation but with homogeneous coefficients) -- we show that there is no need to impose a second order compatibility condition of the initial data, but only a first order compatibility with Kirchhoff's condition at the vertex; we show that this is enough to ensure that the solution is $\mathcal{C}^1$, with respect to the variable $x$ in the whole domain $[0,T]\times [0,R]$, for all the rays $(\mathcal{R}_i)_{i\in[\![1,I]\!]}$. 

Turning to uniqueness issues, one should keep in mind that the regularity at the junction point of the time derivative and the Laplacian of the solution is not required. The classical results on well-posedeness of parabolic systems that use the same linear operator involved in our problem ensure that the linear operator is invertible on the Banach space $\mathcal{C}^{1+\frac{\alpha}{2},2+\alpha}\big([0,T]\times [0,R]\big)\times\mathcal{C}^{1+\alpha/2}\big([0,T]\big) $, (see  see Chapter IV, section 7 in \cite{pde para}) as long as the coefficients belong at least to the class $\mathcal{C}^{1+\alpha/2}\big([0,T]\big)$ up to the boundary. Notably, we will see that only Lipschitz coefficients at the junction $\{0\}$ are necessary to ensure the validity of Kirchhoff's boundary condition. This phenomenon was also observed in \cite{Ohavi PDE}.
\medskip{}
\medskip{}

Our main objective is to ensure the well posedness of the system \eqref{eq : pde with l}, where a new variable $l$ comes into play. If one wishes, as is natural, to exploit and adapt similar ideas as in the classical approach by following the techniques of proof given in \cite{pde para} and performing the same kind of transformations as in \cite{Von Below} for example, the following issues would surely have to be considered:\\
i) obtain an explicit solution on the half line with constant coefficients using the heat kernel. This relates to the joint density of the reflected Brownian motion and its local time;\\
ii) obtain the solvability with general coefficients in the half line (for e.g. as in Chapter IV, Section 7 of \cite{pde para}) but taking good care of the fact that we do not have an uniform parabolic operator in the variables $(x,l)$;\\
iii) adapt the theory of linear parabolic systems to the linear operator involved by the system, like in Chapter VII of \cite{pde para}, which would lead to very long polynomial calculations.
\medskip

As already mentioned in the Introduction, in this paper we prefer to choose another path and dig into the recent ideas of \cite{Ohavi PDE}, where the second author obtained classical solvability in H\"{o}lder spaces for quasi-linear parabolic system posed on a star-shaped network, with a homogeneous Neumann or Kirchhoff's condition denoted by $F$, by constructing and studying a convergent elliptic scheme. 

Let us now give some insights on the methodology used in \cite{Ohavi PDE}.
First, elementary arguments show that the elliptic quasi-linear problem is well posed (see \cite{Lions Souganidis 1} or Appendix B in \cite{Ohavi PDE}). The data of the system satisfies the classical assumption of uniform ellipticity, with quadratic growth in the gradient variable given in Chapter VI of \cite{pde para}, whereas the boundary condition $F$ is assumed to be increasing with respect to the gradient at the junction point $\{0\}$. The main key is to obtain first a bound for $|\partial_tu|$ in the whole domain (see Lemma 4.1 in \cite{Ohavi PDE}). Note that in the quasi-linear context of \cite{Ohavi PDE}, the price to pay was to consider homogeneous coefficients. Let us also mention that all the bounds for the solution are completely independent of Kirchhoff's boundary condition $F$ (see Lemma 4.1 and 4.2 in \cite{Ohavi PDE}). 

Following these ideas, the construction of the solution of system \eqref{eq : pde with l} will be done {\it via} a convergence approximation scheme by constructing a parabolic discretization scheme with a discrete grid w.r.t. the variable $l$:
\begin{eqnarray*}
\begin{cases}
\partial_tu_i^p(t,x)-a_i(t,x,l_p)\partial_x^2u_i^p(t,x)+b_i(t,x,l_p)\partial_xu_i^p(t,x)+c_i(t,x,l_p)u_i^p(t,x)=f_i(t,x,l_p),\\
\ds n(u^{p-1}(t,0)-u^p(t,0))+\sum_{i=1}^I\alpha_i(t,l_p)\partial_xu_i^p(t,0)-r(t,l_p)u^p(t,0)=\phi(t,l_p).
\end{cases}.
\end{eqnarray*}
Getting accurate expressions of the bounds for the derivatives of the solution of each such $l_p$-step parabolic problem is of crucial importance in order to guarantee the convergence of the sequence towards a non-exploding solution as the mesh-size of the $l$-grid tends to zero. This section is entirely devoted to the matter of getting expressions of these bounds that are {\it good enough} (see Theorem \ref{th: ex sys para}). 

With this purpose in mind, we follow the same line of arguments as in \cite{Ohavi PDE}: we construct an elliptic system designed to converge to the parabolic problem. The linear character of our system permits to simplify some of the arguments since Bernstein's estimates are no longer needed to find a bound for the gradient term; also, up to a bit of extra burdensome technicalities, the strong assumption on the homogeneity of the coefficients that was needed in \cite{Ohavi PDE} is not required here. A central key is to obtain an uniform bound of the elliptic system approximation time derivative $n|u^{k-1}(t,0)-u^k(t,0)|$ at the junction point $\{0\}$. This is done in Proposition \ref{pr : borne deriv temps en 0}, where we provide a refined uniform bound independent of the coefficients appearing on the rays: this refined bound will be crucial to ensure the convergence of our $l$-step parabolic scheme in Section \ref{sec: preuve résultat principal}.

\medskip{}
\medskip{}
\medskip{}
\medskip{}
In the remainder of this section, we use the same notations introduced in Introduction \ref{sec:intro} and Appendix \ref{sec : functionnal spaces}, removing the dependence on the "local-time" variable $l$. We consider the following data:
$$\mathcal{D}^{'}~~\begin{cases}\Big(a_i \in W^{1,\infty}\big([0,T]\times [0,R],\R_+\big)\Big)_{i\in[\![1,I]\!]}\\
\Big(b_i \in W^{1,\infty}\big([0,T]\times [0,R],\R\big)\Big)_{i\in[\![1,I]\!]}\\
\Big(c_i \in W^{1,\infty}\big([0,T]\times [0,R],\R\big)\Big)_{i\in[\![1,I]\!]}\\
\Big(f_i \in W^{1,\infty}\big([0,T]\times [0,R],\R\big)\Big)_{i\in[\![1,I]\!]}\\
\Big(\alpha_i \in W^{1,\infty}\big([0,T],\R_+\big)\Big)_{i\in[\![1,I]\!]}\\ 
\lambda \in W^{1,\infty}\big([0,T],\R_+\big),\\
\gamma \in W^{1,\infty}\big([0,T],\R\big),\\
g\in\mathcal{C}^{1}\big(\mathcal{N}_R \big)\cap \mathcal{C}^{2}_b\big(\overset{\circ}{\mathcal{N}_R^*}\big)
\end{cases}.$$
We assume furthermore that the data $\mathcal{D}^{'}$ satisfy the following assumption:
$$\textbf{Assumption } (\mathcal{H}^{'})$$
a)-Ellipticity condition for the terms $\big(a_i,\alpha_i,\big)_{i\in [\![1,I]\!]}$ and $\lambda$:
\begin{align*}
&(i)~~\exists~\underline{a}>0,~~\forall i \in [\![1,I]\!], ~~\forall (t,x)\in [0,T]\times [0,R],~~a_i(t,x) \ge \underline{a},\\
&(ii)~~\exists~\underline{\alpha}>0,~~\forall i \in [\![1,I]\!], ~~\forall t\in [0,T],~~\alpha_i(t) \ge \underline{\alpha},\\
&(iii)~~\exists~\underline{\lambda}>0,~~\forall t\in [0,T],~~\lambda(t) \ge \underline{\lambda}.
\end{align*}
b)-Compatibility conditions for the initial condition $g$:
\begin{align*}
&(i)~~-\lambda(0)g(0)+\sum_{i=1}^I\alpha_i(0)\partial_xg_i(0)=\gamma(0),\\
&(ii)~~\partial_xg(R)=0.
\end{align*}
We consider the following parabolic system posed on the star-shaped network $\mathcal{N}_R$:
\begin{eqnarray}\label{eq : pde para }
\begin{cases}\partial_tu_i(t,x)-a_i(t,x)\partial_x^2u_i(t,x)
+b_i(t,x)\partial_xu_i(t,x)+
c_i(t,x)u_i(t,x)=f_i(t,x),\\
\hspace{0,4 cm}~~(t,x)\in (0,T)\times (0,R),\\
-\lambda(t)u(t,0)+\displaystyle \sum_{i=1}^I \alpha_i(t)\partial_xu_i(t,0)=\gamma(t),~~t\in(0,T),\\
\partial_xu_i(t,R)=0,~~ t\in (0,T),\\
\forall (i,j)\in[\![1,I]\!]^2,~~u_i(t,0)=u_j(t,0)=u(t,0),~~t\in[0,T],\\
\forall i\in[\![1,I]\!],~~ u_i(0,x)=g_i(x),~~x\in[0,R].
\end{cases}
\end{eqnarray}

Define $(t_k=\frac{kT}{n})_{0 \leq k \leq n}$ by a grid on $[0,T]$ . We consider $(u_i^k)_{i\in [\![1,I]\!], k\in [\![0,n]\!]}$ the unique classical solution of the following system of elliptic equations $\pare{{\mathcal E}_k}_{k\in [\![1,n]\!]}$:~
\begin{eqnarray}\label{eq: schema ell}
{\mathcal E}_k~:~
\begin{cases}
n(u_i^k(x)-u_{i}^{k-1}(x))-a_i(t_k,x)\partial_x^2u_i^k(x) \\
\hspace{0,3 cm}+\,b_{i}(t_k,x)\partial_xu_i^k(x) + c_i(t_k,x)u_i^k(x)=f_i(t_k,x)~~ \text{ if } x \in (0,R),\\ 
\ds -\lambda(t_k)u^k(0)+\sum_{i=1}^I {\alpha}_i(t_k)\partial_x u_i^k(0)=\gamma(t_k),\\
\partial_xu_i^k(R)=0,~~\forall i\in[\![1,I]\!],\\
u_i^k(0)=u_{j}^k(0)=u^k(0),~~\forall (i,j)\in[\![1,I]\!]^2.
\end{cases}
\end{eqnarray}
where $u^0_i(x)=g_i(x)$.
\medskip
 By applying inductively classical results on elliptic partial differential equations (see for e.g. Theorem 2.1 of \cite{Lions Souganidis 1}) gives us assurance that at each step $k\in [\![1,n]\!]$ the above elliptic system \eqref{eq: schema ell} admits a unique solution $(u_i^k)_{i\in [\![1,I]\!]}$ in the class $\mathcal{C}^{2}(\mathcal{N}_R)$. A map $h$ in the class $\mathcal{C}^{2}(\mathcal{N}_R)$ is a super (resp. sub) solution corresponding to ${\mathcal E}_k$ if:
\begin{eqnarray*}
\begin{cases}
n(h_i(x)-u_{i}^{k-1}(x))-a_i(t_k,x)\partial_x^2h_i(x)+
b_{i}(t_k,x)\partial_xh_i(x)+c_{i}(t_k,x)h_i(x)-f_i(t_k,x)\ge 0,\\
\text{(resp.}~\leq 0),~~\text{ if } x \in (0,R),\\ 
\ds -\lambda(t_k)h(0)+\sum_{i=1}^I \alpha_i(t_k)\partial_x h_i(0)-\gamma(t_k)\leq 0,~~\text{(resp.}~\ge 0)\\
\partial_xh_i(R)\ge 0,~~\text{(resp.}~\leq 0),~~\forall i\in[\![1,I]\!],\\
h(0)=h_{j}(0)=h_{i}(0),~~\forall (i,j)\in[\![1,I]\!]^2.
\end{cases}
\end{eqnarray*}
The elliptic comparison theorem holds true in the class $\mathcal{C}^2(\mathcal{N}_R)$ (see Theorem 3.3 in \cite{Ohavi PDE}), that is if $f$ is a super solution and $v$ a sub solution, then $f\ge v$ in the whole domain $\mathcal{N}_R$. 

%%%%%%%%%%%%%%%%%%%%%%%%%%%%%%%%%%%%%%%%%%%%%%%%%%%%%%%%%%%%%%%%%%%%%%%%%%%%%%%%%%%%%%%%%%%%%%%%%%%%%%%%%%%%%%%%%%%%%%%%%%%%%%%%%%%%%%%%%%%%%%%%%%%%%%%%%%%%%%%%%%%%%%%%%%%%%%%%%%%%%%%%%%%%%%%%%%%%%%

\medskip

For a fixed $n\in \N^\ast$, we will denote in the sequel $\pare{{\mathcal L}_{i}^k}_{i\in [\![1,I]\!], k\in [\![1,n]\!]}$ the family of operators acting each on $\phi \in C^2([0,R])$ and defined by
\begin{align*}
{\mathcal L}_{i}^k\phi(x) &= n(\phi(x)-u_{i}^{k-1}(x))-a_i(t_k,x)\partial_x^2 \phi(x) +\,b_{i}(t_k,x)\partial_x \phi(x) + c_i(t_k,x)\phi(x) - f_i(t_k,x).
\end{align*}
Using this notation, $h$ is a super (resp. sub) solution corresponding to  ${\mathcal E}_k$ implies ${\mathcal L}_i^k h \geq 0$ (resp. $\leq 0$) for all $x\in (0,R)$.
\medskip
\subsection{Uniform bound for the solution}
\begin{Proposition}
\label{prop:uniform-bound}
For any large enough $n\in \N^\ast$ ($n\geq \pare{\lfloor |c|_{\infty}\rfloor + 1} \vee |c|^2_\infty$),
\begin{equation}
\label{eq:def-C0}
\max_{i \in [\![1,I]\!]}  \max_{k\in [\![1,n]\!]} |u_i^k|_\infty \leq C_0,
\end{equation}
with
\begin{equation}
\label{eq:def-C0-schema-elliptique}
C_0:=\pare{\frac{|\gamma|_\infty}{\underline{\lambda}}\vee |g|_\infty + |f|_ \infty}{\rm e}^{|c|_\infty + 1}.
\end{equation}
\end{Proposition}
\begin{proof}

\medskip

{\it Step 1. Analysis}.

We are going to show by induction on the variable $k\in [\![0,n]\!]$ that
$$\forall k\in [\![0,n]\!],\;\;\forall i\in [\![1,I]\!],~~|u_i^k|_\infty \leq M_{k,n}$$
for a well chosen positive sequence $\pare{M_{k,n}}_{k\in [\![0,n\!]}$ independent on the branch index $i\in [\![1,I]\!]$.

For the construction of $\pare{M_{k,n}}_{k\in [\![0,n\!]}$, the main tool is to use the elliptic comparison theorem for junction partial differential equations (see for e.g. Theorem 3.3 in \cite{Ohavi PDE} or Theorem 2.1 in \cite{Lions Souganidis 1}) iteratively for each problem ${\mathcal E}_k$ with the family $(\hat{\phi}_i^k)_{i\in [\![1,I]\!],\;k\in [\![1,n]\!]}$ of constant functions defined by
\begin{equation}
\label{def:hat-phi-constant}
\hat{\phi}_i^k~:~x\mapsto {M_{k,n}} \;\;\;\; \text{and}\;\;\;\;\;\check{\phi}_i^k~:~x\mapsto -{M_{k,n}}.
\end{equation}

Initialisation of our induction imposes that for all $i\in [\![1,I]\!]$,
$$
M_{0,n} \geq |u_i^0|_\infty = |g_i|_{\infty},$$
which is guaranteed by choosing $M_{0,n}$ such that 
\begin{equation}
\label{eq:cond-M-init}
   M_{0,n} \geq \max_{i\in [\![1,I]\!]}|g_i|_{\infty} 
\end{equation}

Let us now find sufficient conditions that ensure the heredity of the property. Fix $k\in [\![1,n]\!]$. As announced before, the idea is to make use of the comparison theorem for the problem ${\mathcal E}_k$ in combination with the induction hypothesis: $\forall i\in [\![1,I]\!],~|u_i^{k-1}|_\infty \leq M_{k-1,n}$.

Dropping the references to $i$, the induction hypothesis implies 
\begin{align*}
{\mathcal L}^k\hat{\phi}^k(x) &\geq n\,\pare{M_{k,n} - u^{k-1}(x)} + c(t_k,x){M_{k,n}} - f(t_k, x) 
\geq (n - |c|_\infty)M_{k,n} - nM_{k-1,n} - |f|_\infty
\end{align*}
so that ${\mathcal L}^k\hat{\phi}^k(x)\geq 0$ is guaranteed if the following induction relation holds
\begin{align*}
%%\label{eq:cond-M-uniform}
M_{k,n} = \frac{n}{n-|c|_\infty}\pare{M_{k-1,n} + \frac{|f|_\infty}{n}}
\end{align*}
that is solved for
\begin{align}
\label{eq:cond-M-uniform-2}
M_{k,n} = \pare{\frac{n}{n - |c|_\infty}}^k M_{0,n} + \frac{|f|_ \infty}{|c|_\infty}\pare{\pare{\frac{n}{n - |c|_\infty}}^{k} - 1}{\bf 1}_{|c|_\infty\neq 0}
\end{align}
(remember that $n\geq \lfloor |c|_{\infty}\rfloor + 1$).

Let us now turn to the boundary conditions needed to apply the comparison theorem. 

The continuity condition at the junction point $\{0\}$ of the family $(\hat{\phi}_i^k)_{i\in [\![1,I]\!],\;k\in [\![1,n]\!]}$ is clearly satisfied since 
\[\phi^k_i(0) = \phi^k_j(0) = M_{k,n}\;\;\;\forall (i,j)\in [\![1,I]\!]^2.\]
Moreover, it is also clear that
$\partial_{x}\hat{\phi}^k(R) = 0\geq 0$. For Kirchhoff's condition, we need
\begin{align*}
0\geq &-\lambda(t_k)\phi^k(0)+\sum_{i=1}^I {\alpha}_i(t_k)\partial_x \phi^k(0) - \gamma(t_k) = -\lambda(t_k)M_{k,n} -\gamma(t_k)
\end{align*}
that is guaranteed whenever
\begin{equation}
\label{eq:Kirchhoff-phi-uniform-bound}
M_{k,n} \geq \frac{|\gamma|_\infty}{\underline{\lambda}}.
\end{equation}

In conclusion of this analysis, we have shown by induction that $(\hat{\phi}_i^k)_{i\in [\![1,I]\!],\;k\in [\![0,n]\!]}$ defined in \eqref{def:hat-phi-constant} is a family of super solutions of the elliptic problems $({\mathcal E}_k)_{k\in [\![0,n]\!]}$ whenever the sequence $\pare{M_{k,n}}_{k\in [\![0,n]\!]}$ satisfies \eqref{eq:cond-M-init} together with \eqref{eq:cond-M-uniform-2}, and \eqref{eq:Kirchhoff-phi-uniform-bound} for all $k\in [\![0,n]\!]$.

{\it Step 2. Synthesis}

In regard of our previous analysis, we define our purposely designed sequence $\pare{M_{k,n}}_{k\in [\![0,n]\!]}$ by setting
\begin{align*}
M_{0,n} &= \frac{|\gamma|_\infty}{\underline{\lambda}}\vee |g|_\infty,\\
M_{k,n} &= \pare{\frac{n}{n - |c|_\infty}}^k M_{0,n} + \frac{|f|_ \infty}{|c|_\infty}\pare{\pare{\frac{n}{n - |c|_\infty}}^{k} - 1}{\bf 1}_{|c|_\infty\neq 0}.
\end{align*}
Defined likewise, the sequence $\pare{M_{k,n}}_{k\in [\![0,n]\!]}$ has been purposely constructed in order to satisfy \eqref{eq:cond-M-init}, \eqref{eq:cond-M-uniform-2} and \eqref{eq:Kirchhoff-phi-uniform-bound} for all $k\in [\![0,n]\!]$. Consequently, by applying the comparison theorem to the elliptic problems $({\mathcal E}_k)_{k\in [\![0,n]\!]}$ with
the family of super solutions $(\hat{\phi}_{i}^k)_{i\in [\![1,I]\!]}$ defined in \eqref{def:hat-phi-constant} (and similarly to the family of sub solutions $(\check{\phi}_{i}^k)_{i\in [\![1,I]\!]}$), we ensure that for all $i\in [\![1,I]\!]$ and $k\in [\![1,n]\!]$~:~
\[
-M_{k,n} \leq u_i^k(x) \leq M_{k,n},\;\;\;\; x\in [0,R].
\]
Finally, using the explicit expression of $M_{k,n}$ yields that
for all $i\in [\![1,I]\!]$, $k\in [\![1,n]\!]$~:~
\[
|u_i^k|_\infty \leq \pare{\frac{|\gamma|_\infty}{\underline{\lambda}}\vee |g|_\infty + |f|_ \infty}{\rm e}^{|c|_\infty + 1}:=C_0,
\]
for large enough $n\in \N^\ast$ ($n\geq \pare{\lfloor |c|_{\infty}\rfloor + 1} \vee |c|^2_\infty$),
which ends the proof.
\end{proof}

\subsection{Uniform bound for the approximated time derivative}

Recall that $\underline{\alpha}$ stands for the ellipticity condition satisfied by the Kirchhoff's coefficients $(\alpha_i)_{i\in [\![1,I]\!] }$ (Assumption $ 
 (\mathcal{H}^{'})-a)-ii)$. We have the following result.

\begin{Proposition}
\label{Prop:estimation-time-derivative}
For any large enough $n\in \N^\ast$ ($n\geq \pare{\lfloor |c|_{\infty}\rfloor + 1} \vee |c|^2_\infty \vee \frac{|\alpha|_ {W^{1,\infty}}}{\underline{\alpha}}$),
\begin{equation}
\label{eq:def-C1}
\max_{i \in [\![1,I]\!]}  \max_{k\in [\![1,n]\!]} n|u_i^k - u_{i}^{k-1}|_\infty \leq C_1
\end{equation}
with
\begin{equation}
\label{eq:C1}
C_1:=\pare{1 + \frac{C_0|\lambda|_{W^{1,\infty}} + |\gamma|_{W^{1,\infty}}}{\underline{\lambda}}\vee C(g)}\exp\pare{K}
\end{equation}
where we have set
\begin{equation}
\label{eq:C(g)}
C(g):=|a|_\infty|\partial_x^2 g|_\infty + |b|_\infty|\partial_x g|_\infty + |c|_\infty|g|_\infty + |f|_\infty,
\end{equation}
and
\begin{equation}
\label{eq:K}
K:=\frac{1}{\underline{a}^3 \vee 1}\pare{1 + \frac{|\alpha|}{\underline{\alpha}} + C_0 + C}\pare{|b|_{W^{1,\infty}} + |c|_{W^{1,\infty}} + |f|_{W^{1,\infty}}}
\end{equation}
where $C$ is a universal constant (namely one can choose $C = 1188$).
\end{Proposition}
\begin{proof}

 {\it Step 1. Adaptive approximation of the identity}
 
For technical reasons that will appear clearly later, we need to introduce approximations of the identity function. These approximations will be used to ensure Kirchhoff's condition for the super and sub solutions constructed in the proof.

 In order to simplify the notations, let us fix for a moment $k\in [\![1,n]\!]$ and let us drop any reference to $i\in [\![1,I]\!]$. Denote $\ds [\alpha]^k := \frac{\alpha_i(t_{k-1})}{\alpha_i(t_k)}$.

Observe that
\begin{equation}
\label{eq:inegalite-crochet-alpha}
|[\alpha]^k - 1|\leq \frac{|\alpha|_{W^{1,\infty}}}{\underline{\alpha}n}
\end{equation}
where we recall that $\underline{\alpha}$ states for the ellipticity condition satisfied by the Kirchhoff's coefficients $(\alpha_i)_{i\in [\![1,I]\!] }$ (Assumption $ 
 (\mathcal{H}^{'})-a)-ii)$.
Set $\theta >0$ a small parameter. We introduce the following interpolation polynomial 
\begin{equation}
\label{def:poly}
P_\theta^k(x) = \pare{1 - [\alpha]^k}\left \{\frac{3}{\theta^4}x^5 - \frac{8}{\theta^3}x^4 + \frac{6}{\theta^2}x^3\right \} + [\alpha]^kx
\end{equation}
that satisfies the following important facts
\begin{eqnarray}
\begin{cases}
P_\theta^k(0) = 0,\;\;\; (\partial_x P_\theta^k)(0) = [\alpha]^k,\;\;\; (\partial^2_x P_\theta^k)(0) = 0\\
P_\theta^k(\theta) = \theta,\;\;\; (\partial_x P_\theta^k)(\theta) = 1,\;\;\; (\partial^2_x P_\theta^k)(\theta) = 0.
\end{cases}
\end{eqnarray}
The polynomial $P_\theta^k$ is constructed s.t. the $\theta$-approximation of the identity $\psi_\theta^k$ defined by
\begin{equation}
\label{def:psi}
\psi_\theta^k(x) = P_\theta^k(x){\bf 1}_{x\leq \theta}  + x{\bf 1}_{x>\theta}\;\;\;\;\text{ for } x \in [0,R]
\end{equation}
is a twice-differentiable function. 

An elementary study of the polynomial $Q(x) = 3x^5 - 8x^4 + 6x^3$ shows that it takes only positive values and satisfies that $x - Q(x)\geq 0$ for any $x\in [0,1]$. In particular, rewritting $P_\theta^k$, we see that for all $x\in [0,\theta]$:~
\begin{align*}
P_\theta^k(x) &= \theta\pare{(1-[\alpha]^k)Q(x/\theta) + [\alpha]^k(x/\theta)}=\theta\croc{\underbrace{Q(x/\theta)}_{\geq 0} + \underbrace{[\alpha]^k}_{\geq 0}\underbrace{\pare{(x/\theta) - Q(x/\theta)}}_{\geq 0}}\geq 0.
\end{align*}
Hence, from the bound \eqref{eq:inegalite-crochet-alpha}, we see that it is possible to choose $\theta_0>0$ small enough so that 
$\psi_\theta^k\pare{[0,R]}\subset [0,R]$ for all $\theta\in (0,\theta_0]$.

Observe that
\begin{equation}
\label{eq:poly-en-0}
(\partial_x \psi_\theta^k)(0) = [\alpha]^k,\hspace{0,3 cm}\lim_{x\searrow 0+}|\partial_x^2 \psi_\theta^k(x)| = 0.
\end{equation}
Observe also that there is a universal constant $C>0$ (a rough computation gives $C\geq 66\times 18 = 1188$ as announced in the statement of the proposition) s.t. 
\begin{eqnarray}
\begin{cases} 
\label{eq:estimation-polynomes}
\ds |\psi_\theta^k - {\rm id}|_{\infty}\leq  C\frac{|\alpha|_{W^{1,\infty}}}{\underline{\alpha}\,n},\\
\ds |\partial_x \psi_\theta^k - 1|_{\infty}\leq C\frac{|\alpha|_{W^{1,\infty}}}{\underline{\alpha}\,n},\\
\ds |\partial_x^2 \psi_\theta^k(x)|\leq  C\frac{|\alpha|_{W^{1,\infty}}}{\underline{\alpha}\,n\,\theta}{\bf 1}_{x\leq \theta},~~\forall x\in [0,R],\\
\ds |\pare{\partial_x \psi_\theta^k}^2 - 1|_{\infty}\leq  C\frac{|\alpha|_{W^{1,\infty}}}{\underline{\alpha}\,n},
\end{cases}
~~\text{where}~~{\bf 1}_{x\leq \theta}:=\begin{cases}1,~~\text{if}~~0\leq x\leq \theta,\\
0,~~\text{if}~~\theta < x\leq R\end{cases},
\end{eqnarray}
and we made use once again of the bound \eqref{eq:inegalite-crochet-alpha} with $n \geq \frac{|\alpha|_ {W^{1,\infty}}}{\underline{\alpha}}$.

Thus, reintroducing the dependencies on $i$ in our notations, we define likewise a family  $(\psi_{i,\theta}^k)_{i\in [\![1,I]\!], k\in [\![0,n]\!]}$ of $[0,R]$ valued approximations of the identity by initializing $\psi_{i,\theta}^0$ to the identity (for all $i\in [\![1,I]\!]$), and by defining $\psi_{i,\theta}^k$ through \eqref{def:psi} when $k\in [\![1,n]\!]$.

\medskip
{\it Step 2. Analysis}

We aim at showing by induction on the variable $k$ that
\[\forall \theta \in (0,\theta_0],\,\,\forall k\in [\![1,n]\!],\,\,\forall i\in [\![1,I]\!],~~n|u_i^k-\pare{u_i^{k-1}\circ \psi_{i,\theta}^k}|_\infty \leq M^\theta_{k,n} + \varepsilon_n(\theta)\]
where $(\varepsilon_n(\theta))$ is a sequence of functions vanishing as $\theta$ goes to $0$ (non uniformly w.r.t $n$) and $\pare{M^\theta_{k,n}}_{k\in [\![1,n]\!]}$ is a well-chosen purposely designed uniformly bounded sequence of positive numbers.
\medskip

To that end, the main idea is to apply, for each $k\in [\![1,n]\!]$, the comparison theorem to $\mathcal{E}_{k}$ with a super solution of type
\begin{equation}
\label{def:hat-phi}
\hat{\phi}_{i,\theta}^k~:~x\mapsto \pare{u_i^{k-1}\circ \psi_{i,\theta}^k}(x)+\frac{M^\theta_{k,n}}{n} \;\;\;\;i\in [\![1,I]\!]
\end{equation}
\text{ and a sub solution of type}  
\[\check{\phi}_{i,\theta}^k~:~x\mapsto -\pare{u_i^{k-1}\circ \psi_{i,\theta}^k}(x)+\frac{M^\theta_{k,n}}{n}, \;\;\;\;i\in [\![1,I]\!].
\]

Dropping once again the references to the branch index $i$, we have 
\begin{align*}
{\mathcal L}^k\hat{\phi}_{\theta}^k(x) = M_{k,n}^\theta &+ n\pare{\croc{u^{k-1}\circ \psi_\theta^k}(x) - u^{k-1}(x)}\\
&\;\;\;\;\;\;- a(t_k,x)\pare{\pare{\partial_x \psi_\theta^k}^2 \croc{\pare{\partial_x^2 u^{k-1}}\circ \psi^k_\theta}(x) + \partial^2_x \psi_\theta^k \croc{\pare{\partial_x u^{k-1}}\circ \psi^k_\theta}}(x)\\
&\;\;\;\;\;\;+b(t_k,x)\partial_x \psi_\theta^k \croc{\pare{\partial_x u^{k-1}}\circ \psi^k_\theta}(x) + c(t_k,x)\croc{u^{k-1}\circ \psi^k_\theta}(x) + f(t_k, x).
\end{align*}

Note that for $k=1$, the initialization condition ${\mathcal L}^1\hat{\phi}_{\theta}^1 \geq 0$ will be ensured whenever $M_{1,n}^\theta$ satisfies
\begin{equation}
\label{eq:M-init-time-derivative}
 M_{1,n}^\theta \geq n |g(\psi_\theta^1(x)) - g(x)| + C(g),\hspace{0,3 cm}x\in (0,R).
\end{equation}
The initialization conditions needed at the boundaries will be treated later including all the cases.

Let us now fix $k\geq 2$. To avoid an overcrowd of terms in our computations, let us introduce the following set of notations (the symbol $\mathfrak{\Delta v}$ announces a kind of 'error term for $v$' that is due to the use of our approximation of the identity).
Regarding our approximation of the identity~:~
\begin{eqnarray}
\label{eq:delta-coeffs-psi}
\begin{cases}
\mathfrak{id}_\theta^k(x)&=\ds \frac{a(t_k,x)\pare{\partial_x \psi_\theta^k(x)}^2}{a(t_{k-1}, \psi^k_\theta(x))},\\
\mathfrak{\Delta o}_\theta^k(x) &= \pare{1 - \mathfrak{id}_\theta^k(x)}.
\end{cases}
\end{eqnarray}
Regarding the solution~:~
\begin{eqnarray}
\label{eq:delta-coeffs}
\begin{cases}
\mathfrak{\Delta u}_\theta^{k-2}(x) &= \pare{\croc{u^{k-2}\circ \psi^{k-1}_\theta} - \croc{u^{k-2}\circ \psi^{k}_\theta}}(x),\\
\mathfrak{\Delta \tilde{u}}_\theta^{k-2}(x) &= \pare{\croc{u^{k-2}\circ \psi^{k}_\theta} - \croc{u^{k-2}\circ \psi_\theta^{k-1} \circ \psi^{k}_\theta}}(x),\\
\mathfrak{\Delta \tilde{\tilde{u}}}_\theta^{k-2}(x)&= \pare{\croc{u^{k-2}\circ \psi^{k-1}_\theta} - u^{k-2}}(x)\\
\mathfrak{p''}_\theta^k(x) &= -a(t_k,x){\partial^2_x \psi_\theta^k(x)\croc{\pare{\partial_x u^{k-1}}\circ \psi^k_\theta}}(x).
\end{cases}
\end{eqnarray}
Regarding the coefficients of the elliptic problem~:~
\begin{eqnarray}
\label{eq:delta-coeffs-2}
\begin{cases}
\mathfrak{\Delta b}_\theta^k(x) &= \pare{b(t_k,x)\partial_x \psi_\theta^k(x) - b(t_{k-1}, \psi_\theta^k(x))\mathfrak{id}_\theta^k(x)},\\
\mathfrak{\Delta c}_\theta^k(x)&= \pare{c(t_k,x)\partial_x \psi_\theta^k(x) - c(t_{k-1}, \psi_\theta^k(x))\mathfrak{id}_\theta^k(x)},\\
\mathfrak{\Delta f}_{\theta}^k(x) &= \pare{f(t_k,x)- f(t_{k-1}, \psi_\theta^k(x))\mathfrak{id}_\theta^k(x)}.\end{cases}
\end{eqnarray}

With the help of \eqref{eq:estimation-polynomes}, observe that 
\begin{equation}
\label{eq:delta-coeffs-convergence}
\lim\limits_{\theta \searrow 0+}\sup_{k\in [\![2,n]\!]}\pare{|\mathfrak{p''}_\theta^k|_\infty + n|\mathfrak{\Delta u}_\theta^{k-2}|_\infty + n|\mathfrak{\Delta{\tilde{u}}}_\theta^{k-2}|_\infty +  n|\mathfrak{\Delta \tilde{\tilde{u}}}_\theta^{k-2}|_\infty} = 0.
\end{equation}
and also the following bound on the 'error terms'
\begin{eqnarray}
\label{eq:delta-coeffs-estimations}
\begin{cases}
\ds |\mathfrak{\Delta b}_\theta^k|_\infty\leq \ds \frac{C}{n}\pare{|b|_{W^{1,\infty}}(1 + \frac{|\alpha|}{\underline{\alpha}})} + C|b|_{\infty}\sup_{k\in [\![1,n]\!]}|\mathfrak{\Delta o}_\theta^k|_{\infty}\\
\ds |\mathfrak{\Delta c}_\theta^k|_{\infty}\leq \ds \frac{C}{n}\pare{|c|_{W^{1,\infty}}(1 + \frac{|\alpha|}{\underline{\alpha}})} + C|c|_{\infty}\sup_{k\in [\![1,n]\!]}|\mathfrak{\Delta o}_\theta^k|_{\infty}\\
\ds |\mathfrak{\Delta f}_{\theta}^k|_{\infty} \leq \ds \frac{C}{n}\pare{|f|_{W^{1,\infty}}(1 + \frac{|\alpha|}{\underline{\alpha}})} + C|f|_{\infty}\sup_{k\in [\![1,n]\!]}|\mathfrak{\Delta o}_\theta^k|_{\infty}
\end{cases}
\end{eqnarray}
with
\[
|\mathfrak{\Delta o}_\theta^k|_{\infty}\leq C\pare{\frac{|\alpha|_{W^{1, \infty}}}{\underline{\alpha}\,n} + \frac{|a|_{W^{1, \infty}}}{\underline{a}\,n}}.
\]
Our main concern is about second order terms. Using the equation, we have
\begin{align*}
&a(t_{k-1}, \psi^k_\theta(x))\croc{\pare{\partial_x^2 u^{k-1}}\circ \psi^k_\theta}(x)\\
&= n\pare{\croc{u^{k-1}\circ \psi^k_\theta}(x) - \croc{u^{k-2}\circ \psi^k_\theta}(x)}\\
&\;\;\;\;\;\;\;\;\;\;+b(t_{k-1}, \psi_\theta^k(x))\croc{\pare{\partial_x u^{k-1}}\circ \psi^k_\theta}(x) + c(t_{k-1}, \psi_\theta^k(x))\croc{u^{k-1}\circ \psi^k_\theta}(x) + f_i(t_{k-1}, \psi_k^\theta(x)).
\end{align*}
In turn, making use of the notations \eqref{eq:delta-coeffs} and \eqref{eq:delta-coeffs-2}, we get 
\begin{align}
\label{eq:Lhatphi}
&{\mathcal L}^k\hat{\phi}_{\theta}^k(x)=\nonumber\\
&M_{k,n}^\theta - n\pare{u^{k-1}(x) - \croc{u^{k-2}\circ \psi^{k-1}_\theta}(x)}\nonumber\\
&\;\;\;\;\;+n\mathfrak{\Delta o}_\theta^k(x)\mathfrak{\Delta \tilde{u}}_\theta^{k-2}(x)+\mathfrak{\Delta o}_\theta^k(x)n\pare{\croc{u^{k-1}\circ \psi^k_\theta}(x) - \croc{u^{k-2}\circ \psi^{k-1}_\theta \circ \psi^k_\theta}(x)}\\
&\;\;\;\;\;+\mathfrak{\Delta b}_\theta^k(x)\croc{\pare{\partial_x u^{k-1}}\circ \psi^k_\theta}(x) + \mathfrak{\Delta c}_\theta^k(x)\croc{u^{k-1}\circ \psi^k_\theta}(x) + \mathfrak{\Delta f}_{\theta}^k(x) + \mathfrak{p''}_\theta^k(x).\nonumber
\end{align}

Observe also that for all $x\in (0,R)$,
\begin{align*}
\partial_x u^{k-1}(x) &= \partial_x u^{k-1}(x) - \partial_x u^{k-1}(R) = -\int_x^R  \partial^2_x u^{k-1}(z)dz\\
&= -\int_x^R \frac{dz}{a(t_{k-1}, z)}\left \{n\pare{u^{k-1} - u^{k-2}}(z)+ b(t_{k-1},z)\partial_x u^{k-1}(z)\right .\\
&\hspace{1,4 cm}\left .\;+ c(t_{k-1}, y)u^{k-1}(z) + f_i(t_{k-1},z)\right \}.
\end{align*}
Now using
\begin{align*}
\int_x^R \frac{b(t_{k-1},z)}{a(t_{k-1}, z)}\partial_x u^{k-1}(z) = \croc{\frac{b(t_{k-1},z)}{a(t_{k-1}, z)}u^{k-1}(z)}_{x}^R -\int_x^R \partial_z\croc{\frac{b(t_{k-1},z)}{a(t_{k-1}, z)}}u^{k-1}(z)dz,
\end{align*}
we see that
\begin{align*}
\sup_{x\in (0,R)}|\partial_x u^{k-1}(x)| \leq \frac{1}{\underline{a}}\pare{M_{k-1,n}^{\theta} + \frac{C_0}{\underline{a}^2}|b|_{W^{1,\infty}}|a|_{1,\infty} + C_0\pare{|c|_{\infty} + |f|_{\infty}} + n|{\mathfrak{\Delta \tilde{\tilde{u}}}}_\theta^{k-2}|_{\infty}}.
\end{align*}

Denote $\varepsilon_n(\theta) := \sup_{k\in [\![2,n]\!]}\pare{|\mathfrak{p''}_\theta^k|_\infty + n|\mathfrak{\Delta u}_\theta^{k-2}|_\infty + n|\mathfrak{\Delta \tilde{u}}_\theta^{k-2}|_\infty + n|\mathfrak{\Delta \tilde{\tilde{u}}}_\theta^{k-2}|_\infty}$.

From \eqref{eq:Lhatphi}, using all the previous estimates together with the induction hypothesis, we see that in order to guarantee ${\mathcal L}_k\hat{\phi}_{\theta}^k(x)\geq 0$ for all $x\in (0,R)$, it is sufficient that
\begin{equation}
\label{eq:Lphi-positif}
M_{k,n}^\theta \geq M_{k-1,n}^\theta\pare{1 + K\pare{\frac{1}{n} + \varepsilon_n(\theta)}}+ K\pare{\frac{1}{n} + \varepsilon_n(\theta)}
\end{equation} 
with $K=\frac{1}{\underline{a}^3 \vee 1}\pare{1 + \frac{|\alpha|}{\underline{\alpha}} + C_0 + C}\pare{|h|_{W^{1,\infty}} + |c|_{W^{1,\infty}} + |f|_{W^{1,\infty}}}$.
\medskip

Let us now turn to the boundary conditions needed to apply the comparison theorem. 

The continuity condition at the junction point $\{0\}$ is satisfied: indeed, it is satisfied for $(u_i^{k-1})_{i\in [\![1,I]\!]}$ and since for all $(i,j)\in [\![1,I]\!]^2$ $\hat{\psi}_{i,\theta}(0) = \hat{\psi}_{j,\theta}(0) = 0$, we have
\[\hat{\phi}^k_{i,\theta}(0) = u_i^{k-1}({\psi}_{i,\theta}(0)) + \frac{M_{k,n}}{n} = u^{k-1}(0) + \frac{M_{k,n}}{n} = \hat{\phi}^k_{j,\theta}(0) \;\;\;\forall (i,j)\in [\![1,I]\!]^2.\]
It is also clear that for all $x$ in the vicinity of $R$ we have  ${\psi}_{i,\theta}(x) = x$, so that
$\partial_{x}\hat{\phi}_i^k(R) = \partial_x u_i^{k-1}(R) = 0\geq 0$ ($i\in [\![1,I]\!]$) (satisfied when $k=1$ because of the compatibility condition b)(ii)). 

We now look at Kirchhoff's condition. Since the initial condition $(g_i)_{i\in [\![1,I]\!]}$ satisfies itself Kirchhoff's condition (compatibility assumption) we may treat the cases $k=1$ and $k\geq 2$ all together. Observe that for all $i\in [\![1,I]\!]$, $k\in [\![1,n]\!]$, $\hat{\phi}_{i,\theta}^k$ has been constructed such that \[\displaystyle \partial_x\hat{\phi}_{i,\theta}^k(0) = \frac{\alpha_i(t_{k-1})}{\alpha_i(t_k)}\partial_xu_i^k(0).\]

In particular, for fixed $i_0\in [\![1,n]\!]$ and $k\in [\![1,n]\!]$, using Kirchhoff's condition satisfied by $(u_i^k)_{i\in [\![1,I]\!]}$ we have, 
\begin{align*}
0\geq &-\lambda(t_k)\hat{\phi}_{i_0,\theta}^k(0) + \sum_{i=1}^I {\alpha}_i(t_k)\partial_x\hat{\phi}_{\theta}^k(0)-\gamma(t_k)\\
=&-\lambda(t_k)\pare{u_{i_0}^{k-1}(0) + \frac{M_{k,n}^\theta}{n}} + \sum_{i=1}^I {\alpha}_i(t_k)\partial_x\hat{\phi}_{i,\theta}^k(0)-\gamma(t_k)\\
=&-\lambda(t_k)\pare{u^{k-1}(0) + \frac{M_{k,n}^\theta}{n}} + \sum_{i=1}^I {\alpha}_i(t_{k-1})\partial_xu_{i}^{k-1}(0)-\gamma(t_k)\\
=&-\lambda(t_k)\frac{M_{k,n}^\theta}{n} - \pare{\lambda(t_k) - \lambda(t_{k-1})}u^{k-1}(0) - \pare{\gamma(t_{k}) - \gamma(t_{k-1})},
\end{align*}
which is guaranteed whenever
\begin{equation}
\label{eq:Kirchhoff-phi}
M_{k,n}^\theta \geq {\frac{C_0|\lambda|_{W^{1,\infty}} + |\gamma|_{W^{1,\infty}}}{\underline{\lambda}}}.
\end{equation}

In conclusion of our analysis, the family of functions $(\hat{\phi}_{i,\theta}^k)_{k\in [\![1,n]\!]}$ defined in \eqref{def:hat-phi} is assured to be a super solution if the sequence $\pare{M_{k,n}}_{k\in [\![0,n\!]}$ satisfies the initialization condition \eqref{eq:M-init-time-derivative} together with \eqref{eq:Lphi-positif} and \eqref{eq:Kirchhoff-phi}.

{\it Step 3. Synthesis} 

In regard of our previous analysis, we construct a purposely designed sequence $\pare{M^\theta_{k,n}}_{k\in [\![0,n]\!]}$ by setting
\begin{align*}
M_{0,n}^\theta &= {\frac{C_0|\lambda|_{W^{1,\infty}} + |\gamma|_{W^{1,\infty}}}{\underline{\lambda}}} \vee \pare{\varepsilon_n(\theta) + C(g)},\\
M_{k,n}^\theta &= M_{0,n}^\theta\pare{1 + K\pare{\frac{1}{n} + \varepsilon_n(\theta)}}^{k}  + \pare{1 + K\pare{\frac{1}{n} + \varepsilon_n(\theta)}}^{k}.
\end{align*}
Defined thus, the sequence $\pare{M^\theta_{k,n}}_{k\in [\![0,n\!]}$ is purposely constructed in order to satisfy \eqref{eq:M-init-time-derivative}, \eqref{eq:Lphi-positif} and \eqref{eq:Kirchhoff-phi}. Hence, we are in position to apply the comparison theorem to
the family of functions $(\hat{\phi}_{i,\theta}^k)_{i\in [\![1,I]\!]}$. We perform similar computations for the family $(\check{\phi}_{i,\theta}^k)_{i\in [\![1,I]\!]}$ defined in \eqref{def:hat-phi}.
As follows by the application of the comparison theorem, we ensure that for all $\theta \in (0,\theta_0]$, $i\in [\![1,I]\!]$ and $k\in [\![1,n]\!]$~:~
\[
n |u_i^k-\pare{u_i^{k-1}\circ \psi_{i,\theta}^k}|_{\infty} \leq M^\theta_{k,n}.
\]
Now using the explicit expression of $M^\theta_{k,n}$ and letting $\theta$ tend to $0$ in the previous inequality yields finally that
for all $i\in [\![1,I]\!]$ and $k\in [\![1,n]\!]$~:~
\[
n |u_i^k- u_i^{k-1}|_{\infty} \leq  \pare{1 + \frac{C_0|\lambda|_{W^{1,\infty}} + |\gamma|_{W^{1,\infty}}}{\underline{\lambda}}\vee C(g)}\exp\pare{K}
\]
for large enough $n\in \N^\ast$.
\end{proof}
Unfortunately, the previous inequality does not not give a sufficient bound at the junction point $\{0\}$ for our purposes. In order to ensure the convergence of the parabolic scheme involving the local time variable $l$ we need a more refined bound on the time derivative at the junction point $\{0\}$. This is the subject of the next subsection where we refine the previous analysis to get a better bound at the junction point $\{0\}$.

\subsection{Refined estimates for the approximated time derivative at the junction point.}

In this subsection, we give a proof of a specific estimation bound for the approximated time derivative at the junction point that is enough to ensure  the convergence of our forthcoming parabolic scheme in Section \ref{sec: preuve résultat principal}.
In order to derive this key estimate, we replace the construction of the sequence $(M_{k,n})$ in the previous subsection by the construction of a sequence of functions  $(v_{k,n})$ that are solutions of a well-designed system of iterated ODE (see \eqref{eq:ode-parab-estima}). Remarkably enough, our computations show that it is possible to decouple Kirchhoff's condition when passing to the first order time variable error and to prescribe separately the values of the solution and the values of its derivative at the junction.

Note that the system \eqref{eq:ode-parab-estima} possesses only constant coefficients and its solution may be viewed as a kind a supreme envelope of all possible first order time errors. Under this light and from a probabilistic perspective, the structure of \eqref{eq:ode-parab-estima} encompasses the different possible behaviors of the distinct speed measure of reflected stochastic diffusions with characteristics $(\sqrt{2a_i}, b_i)_{i\in [\![1,I]\!]}$. Moreover, it is notable that the same structure equation \eqref{eq:ode-parab-estima} will also be used as the key ingredient to show that the accumulated time spent by the spider motion at the junction point has Lebesgue measure zero (non-stickiness condition), which is a crucial step in order to prove an It\^o formula for the spider motion in presence of discontinuities of the driving coefficients at the junction point.

\medskip{}

Set 
\begin{equation}
\label{eq:def-Theta}
\Theta_n(\lambda, \gamma):={\frac{\sup_{k\in [\![1,n]\!]}|u_k(0)|\;|\lambda|_{\lfloor W^{1,\infty}([]0,T)\rfloor } + |\gamma|_{\lfloor W^{1,\infty}([0,T])\rfloor}}{\underline{\lambda}}}.
\end{equation}

\begin{Proposition}\label{pr : borne deriv temps en 0}
We have
$$\max_{k\in [\![1,n]\!]} \max_{i\in [\![1,I]\!]} n|u_{i}^k(0)-u_{i}^{k-1}(0)|\leq \Theta_n(\lambda, \gamma)\vee C(g)$$
where $C(g)$ is the constant defined in \eqref{eq:C(g)}.
\end{Proposition}
\begin{proof}
As already mentioned, the main idea is to perform the same computations carried over in the proof of Proposition \ref{Prop:estimation-time-derivative}, but replacing there the construction of the sequence $(M_{k,n})$ by the construction of a sequence of functions  $(v_{k,n})$ whose values at $x=0$ depend crudely on the parameters in Kirchhoff's condition. Such a sequence of functions will naturally tragically explode as $n$ tends to infinity on the interior of each branch -- except at the junction point $\{0\}$ -- which is just enough for our purposes.

Let us introduce the sequence $(v_{k,n})_{k\in [\![0,n]\!]}$ by setting
$$v^\theta_{0,n}\equiv~~\max_{i\in[\![1,I]\!]}\Big\{~~\sup_{x\in(0,R)}|-a_i(0,x)\partial_x^2g_i(x)+b_{i}(0,x)\partial_xg_i(x))|~~\Big\}$$
and by defining $v_{k,n}$ for $k \in [\![1,n]\!]$ inductively as the unique solution in $\mathcal{C}^{2}([0,R])$ of the following well-posed second order ordinary differential equation:
\begin{eqnarray}
\label{eq:ode-parab-estima}
\begin{cases}\displaystyle\partial_{x}^2v^\theta_{k,n}(x)+\frac{|b|_\infty}{\underline{a}}\partial_xv^\theta_{k,n}(x)=\kappa^\theta_{k,n}(x),~~ \text{ if }  x\in(0,R),\\
\partial_xv^\theta_{k,n}(0)=0,\\
v^\theta_{k,n}(0)=\Theta_n(\lambda, \gamma)\vee \pare{C(g) + n \max\limits_{i\in [\![1,I]\!]}|g\circ \psi_{i,\theta}^1 - g|_\infty},
\end{cases}
\end{eqnarray}
where the source term $\kappa^\theta_{k,n}$ is given by induction by setting
$$\kappa^\theta_{k,n}(x):=\frac{n}{\underline{a}}\Big(1+K(\frac{1}{n}+\varepsilon_n(\theta))\Big)v^\theta_{k-1,n}(x)+\frac{Kn}{\underline{a}}(\frac{1}{n}+\varepsilon_n(\theta)).$$
Here the constant $K$ and the error term $\varepsilon_n(\theta)$ are the same that appear in the proof of Proposition \ref{Prop:estimation-time-derivative}.
The ODE \eqref{eq:ode-parab-estima} is explicitly solvable for any $k\in [\![1,n]\!]$ with its solution given by
\begin{eqnarray*}
v^\theta_{k,n}(x)=v^\theta_{k,n}(0) + \int_0^x\exp\Big(-\frac{|b|_\infty}{\underline{a}}z\Big)\int_0^z\exp\pare{\frac{|b|_\infty}{\underline{a}}u}\kappa^\theta_{k,n}(u)du\,dz.
\end{eqnarray*}
Note that from the explicit form of $v^\theta_{k,n}$ it is easy to show by induction that 
\begin{eqnarray}\label{eq : conditions0 fonction test}
\forall x \in[0,R],~~0\leq v^\theta_{k,n}(x),~~0\leq \partial_xv^\theta_{k,n}(x),
\end{eqnarray}
In particular, note that $v^\theta_{k,n}$ is an increasing function.

Remember the definition of our family of approximations of the identity $\pare{\psi_{\theta}^k}$ introduced in \eqref{def:psi}.
Following the proof of the Proposition \ref{Prop:estimation-time-derivative}, we show by induction that the following maps
\begin{equation}
\label{def:hat-phi-0}
\hat{\phi}_{i,\theta}^k~:~x\mapsto \pare{u_i^{k-1}\circ \psi_{i,\theta}^k}(x)+\frac{v^\theta_{k,n}(x)}{n} \text{ and }  \check{\phi}_{i,\theta}^k~:~x\mapsto \pare{u_i^{k-1}\circ \psi_{i,\theta}^k}(x)-\frac{v^\theta_{k,n}(x)}{n},\;i\in [\![1,I]\!]
\end{equation}
are respectively super and sub solution of the corresponding elliptic problems. 

Initialization holds true due to the conditions imposed on $g$ at $x=0$ and $x=R$, the expression of the constant $v_{0,n}^{\theta}$ and the fact that $\psi_{\theta}^0\equiv id$.
Following the same computations as in {\it Step 2. Analysis} in the proof of Proposition \ref{Prop:estimation-time-derivative}, we are going to show that the conditions needed to ensure the comparison on the whole domain, namely
\begin{eqnarray*}
\begin{cases}
\hat{\phi}_{\theta}^k\in \mathcal{C}^2\pare{\mathcal{N}_R}\\
\displaystyle \partial_x \hat{\phi}_{i,\theta}^k(R) \ge 0,\\
\ds -\lambda(t_k)\hat{\phi}_{i,\theta}^k(0)+\sum_{i=1}^I\partial_x\hat{\phi}_{i,\theta}^k(0)-\gamma(t_k)\leq 0,\\
\displaystyle \mathcal{L}_i^k\hat{\phi}_{i,\theta}^k(x)=
\ds n(\pare{\hat{\phi}_{i,\theta}^k(x) - u_{i}^{k-1}(x)}-a_i(t_k,x)\partial_x^2\hat{\phi}_{i,\theta}^k(x)\\
\hspace{0,4 cm}\ds + b_{i}(t_k,x)\partial_x \hat{\phi}_{i,\theta}^k(x)+ c_i(t_k,x)\hat{\phi}_{i,\theta}^k(x) -f_i(t_k,x)\ge 0,
\end{cases}
\end{eqnarray*}
are satisfied.

Clearly, since $g$, $u^{k-1}$, $(\psi_{i,\theta})_{i\in [\![1,I]\!]}$ and the family $(v_{k,n}^\theta)_{i\in [\![1,I]\!]}$ all belong to $\mathcal{C}^2\pare{\mathcal{N}_R}$, we verify that $\hat{\phi}_{\theta}^k\in \mathcal{C}^2\pare{\mathcal{N}_R}$. Moreover, using the definition $\partial_xv^\theta_{k,n}(0)=0$, similar arguments as those used in the proof of Proposition \ref{Prop:estimation-time-derivative} ensure that the boundary inequality at junction point holds true. More precisely, using Kirchhoff's condition satisfied by $u^{k-1}$ and the initial values $v_{k,n}^\theta(0)\geq \Theta(\lambda, \gamma)$ and $\partial_x v_{k,n}^\theta(0)=0$ imply $-\lambda(t_k)\hat{\phi}_{i,\theta}^k(0)+\sum_{i=1}^I\partial_x\hat{\phi}_{i,\theta}^k(0)-\gamma(t_k)\leq 0$ for any $k\in [\![1,n]\!]$. 

At $x=R$, we have $$\partial_x\hat{\phi}_{i,\theta}^k(R) = \displaystyle \partial_x\Big(\pare{u_i^{k-1}\circ \psi_{i,\theta}^k}(x)+\frac{v^\theta_{k,n}(x)}{n}\Big)_{x=R}=\partial_x u_i^{k-1}(R)+\partial_xv^\theta_{k,n}(R)\ge 0,$$
because the positive derivative of $v^\theta_{k,n}$.

\medskip

We now focus on the remaining inequality involving the operator ${\mathcal L}_{i}^k$ on each edge.

For $k=1$, the initialization condition ${\mathcal L}_i^1\hat{\phi}^{1}_{i,\theta} \geq 0$ 
is ensured if
\begin{equation}
\label{eq:M-init-time-derivative-2}
 v_{1,n}^\theta \geq n |g(\psi_{i,\theta}^1(x)) - g(x)| + C(g),
\end{equation}
which 
is clearly satisfied because of our definition of  $v_{1,n}^\theta(0)$ and because $v_{1,n}^\theta$ is an increasing function.

Let us now turn to the case $k>1$ and fix $k\in [\![2,n]\!]$. Our induction hypothesis asserts that $\hat{\phi}_{i,\theta}^{k-2}$ is a super solution of $\mathcal{E}_{k-1}$. By comparison we have then that $u_{i}^{k-1} \leq \hat{\phi}_{i,\theta}^{k-1}$. 

In order to simplify the exposition, let us denote ${\mathcal H}_i^{k,n}$ the operator acting on $\phi \in C^2([0,R])$ defined by
\begin{align*}
{\mathcal H}_i^{k,n}\phi(x) &= -\frac{a_i(t_k,x)}{n}\partial_x^2 \phi(x) +\frac{\,h_{i}(t_k,x)}{n}\partial_x \phi(x) + \frac{c_i(t_k,x)}{n}\phi(x) + \frac{f_i(t_k,x)}{n}.
\end{align*}

Dropping any reference to the branch index $i$, using this notation together with the notations used in the proof of Proposition \ref{Prop:estimation-time-derivative}, we have
\begin{align}
\label{eq:Lhatphi-2}
&{\mathcal  L}^k\hat{\phi}_{\theta}^k(x)\nonumber\\
&=\croc{Id + {\mathcal H}^{k,n}}v^\theta_{k,n}(x) - v^{\theta}_{k-1,n}(x) - n\underbrace{{\pare{u^{k-1}(x) - \croc{u^{k-2}\circ \psi^{k-1}_\theta}(x) - \frac{v^\theta_{k-1,n}(x)}{n}}}}_{=u_{i}^{k-1} - \hat{\phi}_{i,\theta}^{k-1}\leq 0}\nonumber\\
&\;\;\;\;\; + n\mathfrak{\Delta u}_\theta^{k-2}(x) + n\mathfrak{\Delta o}_\theta^k(x)\mathfrak{\Delta \tilde{u}}_\theta^{k-2}(x)  +\mathfrak{\Delta o}_\theta^k(x)n\left \{\croc{u^{k-1}\circ \psi^{k}_\theta}(x) - \croc{u^{k-2}\circ \psi^{k-1}_\theta \circ \psi^{k}_\theta}(x)\right \} \nonumber\\
&\;\;\;\;\;+\mathfrak{\Delta h}_\theta^k(x)\croc{\pare{\partial_x u^{k-1}}\circ \psi^k_\theta}(x) + \mathfrak{\Delta c}_\theta^k(x)\croc{u^{k-1}\circ \psi^k_\theta}(x) + \mathfrak{\Delta f}_{\theta}^k(x) + \mathfrak{p''}_\theta^k(x).
\end{align}
Using the induction hypothesis and the regularity of the coefficients lead first to
\begin{eqnarray*}
\mathcal{L}^k\hat{\phi}_{\theta}^k(x) \geq \croc{Id + {\mathcal H}^{k,n}}v^\theta_{k,n}(x) -\Big(1 + K\Big(\frac{1}{n} + \varepsilon_n(\theta)\Big)v_{k-1,n}^\theta(x) - K\Big(\frac{1}{n} + \varepsilon_n(\theta)\Big)\Big).
\end{eqnarray*}
where the constants $K$ and the term of error $\varepsilon_n(\theta)$ have the same expressions given in the last proposition. 
Now using that $v^\theta_{k,n}$ is positive, as well as its first derivative, we see (since $n\ge \lceil |c|_\infty \rceil+1)$ that:
\begin{align*}
\displaystyle \croc{Id + {\mathcal H}_i^{k,n}}v^\theta_{k,n}(x) &= v^\theta_{k,n}(x)-\frac{a_i(t_k,x)}{n}\partial_x^2v^\theta_{k,n}(x) + \frac{b_{i}(t_k,x)}{n}\partial_x v^\theta_{k,n}(x)+ c_i(t_k,x)\frac{v^\theta_{k,n}(x)}{n}\\
&\geq \displaystyle v^\theta_{k,n}(x)\pare{1+\frac{c_i(t_k,x)}{n}}-\frac{a_i(t_k,x)}{n}\Big(\partial_x^2v^\theta_{k,n}(x)- \frac{b_{i}(t_k,x)}{a_i(t_k,x)}\partial_x v^\theta_{k,n}(x)\Big)\\
&\geq \ds -\frac{a_i(t_k,x)}{n}\Big(\partial_x^2v^\theta_{k,n}(x)+\frac{|b|_\infty}{\underline{a}}\partial_x v^\theta_{k,n}(x)\Big).
\end{align*}
From the previous computations and remembering the inductive definition of
$\kappa^\theta_{k,n}(x)$ gives finally
\begin{align*}
\ds \mathcal{L}^k\hat{\phi}_\theta^{k}(x)&\ge
\ds
-\Big(1+K(\frac{1}{n}+\varepsilon_n(\theta))\Big)v^\theta_{k-1,n}(x)-K\pare{\frac{1}{n}+\varepsilon_n(\theta)}\\
&\hspace{2,0 cm}-\frac{a_i(t_k,x)}{n}\Big(\partial_x^2v^\theta_{k,n}(x)+\frac{|b|_\infty}{\underline{a}}\partial_x v^\theta_{k,n}(x)\Big)\\
&\ge \ds \frac{a_i(t_k,x)}{n}\left \{-\partial_x^2v^\theta_{k,n}(x)-\frac{|b|_\infty}{\underline{a}}\partial_x v^\theta_{k,n}(x)+\kappa^\theta_{k,n}(x)\right \} = 0,
\end{align*}
where the equality to $0$ is ensured by the fact that $v^\theta_{k,n}$ satisfies the first line in
\eqref{eq:ode-parab-estima}.
Hence, we have proved that
$$ \mathcal{L}_{i}^k\hat{\phi}_{i,\theta}^{k}(x)\ge 0,$$
which holds for all $i\in [\![1,I]\!]$ and $x\in (0,R)$. 

In conclusion, we ensure that, for any $k\in [\![1,n]\!]$, the family $\pare{\hat{\phi}_{i,\theta}^{k}}_{i\in [\![1,I]\!]}$ is a super solution for $\mathcal{E}_k$. Applying the comparison principle ensures that for all $k\in [\![1,n]\!]$:
$$ \forall x\in[0,R],~~n\pare{u_i^{k}(x)-\pare{u_i^{k-1}\circ \psi_{\theta}^k}(x)}\leq v^\theta_{k,n}(x).$$
The same type of computation may be performed to prove that the family $\pare{\check{\phi}_{i,\theta}^{k}}_{i\in [\![1,I]\!]}$ is indeed a sub solution for $\mathcal{E}_k$ (for any $k\in [\![1,n]\!]$). Applying the comparison principle ensure then that
for all $k\in [\![1,n]\!]$:
$$ \forall x\in[0,R],~~n|u_i^{k}(x)-\pare{u_i^{k-1}\circ \psi_{\theta}^k}(x)|\leq v^\theta_{k,n}(x).$$

In particular since $\psi_{{\theta}}^k(0)=0$ and remembering our prescribed initial condition on $v^\theta_{k,n}(0)$, we conclude that
$$\max_{k\in [\![1,n]\!]} \max_{i\in [\![1,I]\!]} n|u_{i}^k(0)-u_{i}^{k-1}(0)|\leq \Theta_n(\lambda, \gamma)\vee \pare{C(g) + n \max_{i\in [\![1,I]\!]}|g\circ \psi_{i,\theta}^1 - g|_\infty}.$$
The result of the proposition follows then by letting $\theta$ tend to zero in the right hand side.
\end{proof}

%%%%%%%%%%%%%%%%%%%%%%%%%%%%%%%%%%%%%%%%%%%%%%%%%%%%%%%%%%%%
\subsection{Global gradient estimate}

\begin{Proposition}
\label{prop:non-Bernstein-estimate}
There exists a constant $C$ depending only on the data of the system, such that
\begin{equation}
    \label{eq:non-Bernstein-estimate}
\sup_{n} \max_{k\in [\![1,n]\!]} \max_{i\in [\![1,I]\!]} |\partial_x u_i^k|_\infty \leq C_2,
\end{equation}
with
\begin{equation}
\label{eq:def-C2-gradient}
C_2 := C'\exp\pare{\frac{R|b|_\infty}{\underline{a}}} + \frac{C_1 + |c|_\infty C_0 + |f|_\infty}{|b|_\infty}\pare{\exp\pare{\frac{R|b|_\infty}{\underline{a}}} - 1},
\end{equation}
where we have set
\begin{equation}
\label{eq:def-Cprime}
C':=\frac{1}{\underline{a}}\pare{RC_1 + C_0(2|b|_\infty + R(|c|_\infty + |f|_\infty)) + \frac{2|a|_{W^{1,\infty}}|b|_{W^{1,\infty}}}{\underline{a}}}.
\end{equation}
\end{Proposition}
\begin{proof}
Let $n\in \N^\ast$, $k\in [\![1,n]\!]$ and $i\in [\![1,I]\!]$.

Since
$$\partial_x^2u_i^k(x)=\frac{n(u_i^k-u_i^{k-1})(x)}{a_i(x)}+\frac{b_i(x)\partial_xu_i^k(x)}{a_i(x)}+\frac{c_i(x)u_i^k(x)}{a_i(x)}-\frac{f(x)}{a_i(x)},$$
by proceeding to integration between $0$ and $R$ and integrating by parts the gradient term using $\partial_xu^k_i(R) = 0$, we have
\begin{align*}
-\partial_xu_i^k(0)&=\int_0^R\frac{n(u_i^k-u_i^{k-1})(x)}{a_i(x)}dx-\int_0^R\partial_x\pare{\frac{b_i(x)}{a_i(x)}}u_i^k(x)dx+\croc{\frac{b_i(x)u_i^k(x)}{a_i(x)}}_0^R\\
&\;\;\;+\int_0^R\pare{\frac{c_i(x)u_i^k(x)}{a_i(x)}-\frac{f_i(x)}{a_i(x)}}dx.
\end{align*}
Using the results of Propositions \ref{prop:uniform-bound} and \ref{Prop:estimation-time-derivative}, the ellipticity of $a$ and the assumptions $(\mathcal{H}')$ on the coefficients $(a,b,c,f)$ yield an uniform bound 
\[
\sup_{n} \max_{k\in [\![1,n]\!]} \max_{i\in [\![1,I]\!]}|\partial_xu_i^k(0)| \leq C'\]
with 
$$
C':=\frac{1}{\underline{a}}\pare{RC_1 + C_0(2|b|_\infty + R(|c|_\infty + |f|_\infty)) + \frac{2|a|_{W^{1,\infty}}|b|_{W^{1,\infty}}}{\underline{a}}}.
$$
On another hand, from the results of Propositions \ref{prop:uniform-bound} and \ref{Prop:estimation-time-derivative}, we have see that:\\ $\forall x\in [0,R]$,
$$|\partial_x^2u_i^k(x)|\leq \frac{1}{\underline{a}}\big(C_1+|b|_\infty|\partial_xu_i^k(x)|+|f|_\infty+|c|_\infty C_0\big):=A+B|\partial_xu_i^k(x)|,$$
where $C_0$ and $C_1$ are given respectively in \eqref{eq:def-C0-schema-elliptique} and \eqref{eq:C1}.
Hence, we are in position to use Gr\"onwall's lemma, which gives
$$|\partial_xu_i^k(x)|\leq |\partial_xu_i^k(0)|\exp(BR)+\frac{A}{B}(\exp(BR)-1)$$
(where we use the convention $\frac{{\rm e}^{c0} - 1}{0} = c$).
Hence,
\[
\sup_{n} \max_{k\in [\![1,n]\!]} \max_{i\in [\![1,I]\!]} |\partial_x u_i^k|_\infty \leq C_2
\]
with
$$
C_2 = C'\exp\pare{\frac{R|b|_\infty}{\underline{a}}} + \frac{C_1 + |c|_\infty C_0 + |f|_\infty}{|b|_\infty}\pare{\exp\pare{\frac{R|b|_\infty}{\underline{a}}} - 1}.
$$
\end{proof}

%%%%%%%%%%%%%%%%%%%%%%%%%%%%%%%%%%%%%%%%%%%%%%%%%%%%%%%%%%%%%%%%%%%%%%%%%%%%%%%%%%%%%%%%%%%%%%%%%%%%%%%%%%%%%%%%%%%%%%%%%%%%%%%%%%%%%%%%%%%%%%%%%%%%%%%%%%%%%%%%%%%%%%%%%%
%%%%%%%%%%%%%%%%%%%%%%%%%%%%%%%%%%%%%%%%%%%%%%%%%%%%%%%%%%%%%%%%%%

\subsection{The main result of Von Below \cite{Von Below} revisited}

The results of the preceding subsection lead us to gather uniform estimates of the sequence $(u^k)_{k\in [\![0,n]\!]}$ and its partial derivatives. As shown below, similar arguments as those used for the proof of Theorem 2.2 in \cite{Ohavi PDE} give us assurance of the convergence of the elliptic scheme $\pare{{\mathcal E}_k}_{k\in [\![0,n]\!]}$. In turn, this allows us to state the following theorem --~which is somewhat a refined version in the case of a star shaped network~-- of the main result obtained by Von Below in \cite{Von Below}. We recall for the convenience of the reader, the functional spaces used for the class of solvability of the parabolic system \eqref{eq : pde para }.
\begin{align*}
&\mathcal{C}^{\frac{\alpha}{2},1+\alpha}\big([0,T]\times\mathcal{N}_R\big):=\Big\{f:=[\![1,I]\!]\times[0,T]\times [0,R]\to \R,~~(i,t,x)\mapsto f_i(t,x),~~\Big|\\
&\hspace{6.5 cm}\forall i\in [\![1,I]\!],~~f_i\in \mathcal{C}^{\frac{\alpha}{2},1+\alpha}\big([0,T]\times[0,R]\big)\Big\},\\
&\mathcal{C}^{1+\alpha,2+\alpha}\big((0,T)\times \overset{\circ}{\mathcal{N}_R^*}\big):=\Big\{f:=[\![1,I]\!]\times[0,T]\times [0,R]\to \R,~~(i,t,x)\mapsto f_i(t,x),~~\Big|\\
&\hspace{6.5 cm} \forall i\in [\![1,I]\!],~~f_i\in \mathcal{C}^{1+\alpha,2+\alpha}\big((0,T)\times(0,R)\big)\Big\}.
\end{align*}
\begin{Theorem}\label{th: ex sys para}
Assume that the data $\mathfrak{D}^{'}$ satisfy assumptions $(\mathcal{H}^{'})$. Then the parabolic system \eqref{eq : pde para }
is uniquely solvable in the class $\mathcal{C}^{\frac{\alpha}{2},1+\alpha}\big([0,T]\times\mathcal{N}_R\big)\cap\mathcal{C}^{1+\alpha,2+\alpha}\big((0,T)\times \overset{\circ}{\mathcal{N}_R^*}\big)$.
Moreover, there exist constants $(C_0,C_1,C_2,C_3)$, depending only on $R$, $T$, and the data $\mathfrak{D}^{'}$, such that
\begin{eqnarray*}
||u||_{\infty}\leq C_0,~~||\partial_tu||_{\infty}\leq C_1,~~||\partial_xu||_{\infty}\leq C_2,~~||\partial^2_{x}u||_{\infty}\leq C_3,
\end{eqnarray*}
with the expressions of $C_0, C_1, C_2, C_3$ given respectively by
$$
C_0 := \pare{\frac{|\gamma|_\infty}{\underline{\lambda}}\vee |g|_\infty + |f|_\infty}{\rm e}^{|c|_\infty + 1};\;\;\;\;\;C_1:=\pare{1 + \frac{C_0|\lambda|_{W^{1,\infty}} + |\gamma|_{W^{1,\infty}}}{\underline{\lambda}}\vee C(g)}\exp\pare{K};
$$
$$
C_2:= C'\exp\pare{\frac{R|b|_\infty}{\underline{a}}} + \frac{C_1 + |c|_\infty C_0 + |f|_\infty}{|b|_\infty}\pare{\exp\pare{\frac{R|b|_\infty}{\underline{a}}} - 1};$$
$$C_3:=\frac{1}{\underline{a}}\pare{C_1 + |b|_\infty C_2 + |c|_\infty C_0 + |f|_\infty};
$$
with
$$C(g):=|a|_\infty|\partial_x^2 g|_\infty + |b|_\infty|\partial_x g|_\infty + |c|_\infty|g|_\infty + |f|_\infty;$$
$$\ds K:=\frac{1}{\underline{a}^3 \vee 1}\pare{1 + \frac{|\alpha|}{\underline{\alpha}} + C_0 + C}\pare{|b|_{W^{1,\infty}} + |c|_{W^{1,\infty}} + |f|_{W^{1,\infty}}},$$\;\;\;\;
where $C$ stands for the universal constant of Proposition \ref{Prop:estimation-time-derivative};
$$C':=\exp\pare{\frac{R|b|_\infty}{\underline{a}}} + \frac{C_1 + |c|_\infty C_0 + |f|_\infty}{|b|_\infty}\pare{\exp\pare{\frac{R|b|_\infty}{\underline{a}}} - 1}$$
(with the convention $\frac{{\rm e}^{0c} - 1}{0} = c$).

Moreover, set $\;\ds \Theta(\lambda, \gamma):={\frac{||u||_{\infty}\;|\lambda|_{\lfloor W^{1,\infty}([0,T])\rfloor} + |\gamma|_{\lfloor W^{1,\infty}([0,T])\rfloor}}{\underline{\lambda}}}$. Then,
\begin{equation}
\label{eq:estim-deriv-zero}
\max_{i\in [\![1,I]\!]}|\partial_tu_i(.,0)|_\infty\leq \Theta(\lambda, \gamma)\vee C(g).
\end{equation}
\end{Theorem}

\begin{proof}
The proof uses exactly the same arguments of the proof of Theorem 2.2 in \cite{Ohavi PDE} that is given in the quasi-linear parabolic context with a fully non-linear Kirchhoff's boundary condition at the junction point $\{0\}$. For the convenience of the reader we shall give the main issues of the proof, avoiding to linger too much on the details.

Uniqueness (point-wise) is a straight forward consequence of the comparison Theorem 2.4 of \cite{Ohavi PDE} that remains applicable in our linear framework.
%Before getting started with the proof of the solvability, let us precise the main outline of the proof. The solution is constructed with our convergence scheme by letting $n$ tend to infinity. The convergence of partial time derivatives and the second order partial space derivatives is first shown in $L^{2}$, so there is {\it a priori} no limit in $L^{\infty}$. But we have proved that $n(u^k - u^{k-1})$ is uniformly bounded and the elliptic equation $\mathcal{E}_k$ shows that this holds also for $\partial^2_x u^k$. Consequently, the previous $L^2$-limit possesses weak time and second order space partial derivatives that are bounded (by the same constants gathered in the previous subsection): hence, the limiting function is Lipschitz w.r.t the time variable (uniformly w.r.t the space variable) and belongs to $\mathcal{C}^{0,1+\alpha}([0,T]\times [0,R])$; moreover, on each branch the limit is also a weak solution of a parabolic equation, for which standard results apply. We now turn to the formal proof.

Let $n\ge (\lfloor |c|_{\infty} \rfloor +1)\wedge |c|_{\infty}^2$. Consider the subdivision $(t_k^n=\frac{kT}{n})_{0\leq k\leq n}$ of $[0,T]$, and $(u_{k})_{0\leq k\leq n}$ the solution of the elliptic scheme $\mathcal{E}_k$ defined in \eqref{eq: schema ell}. From estimates obtained in Propositions \ref{prop:uniform-bound}-  \ref{Prop:estimation-time-derivative} and \ref{prop:non-Bernstein-estimate}, we obtain that there exists a constant $M>0$ independent of $n$, such that:  
\begin{eqnarray}\label{eq : bornes glob}
\sup_{n\ge 0}~~\max_{k\in[\![1,n]\!]}~~\Big\{|u_{k}|_{\infty}+|n(u_{k}-u_{k-1})|_{\infty}+
|\partial_xu_{k}|_{(0,R)}+|\partial_x^2u_{k}|_{(0,R)}~~\Big\}\leq M.
\end{eqnarray}
Define the following sequence $(v_{n})_{n\ge0}$ in $\mathcal{C}^{0,2}\big([0,T]\times \mathcal{N}_R\big)$, which is piecewise differentiable with respect to the time variable: 
\begin{eqnarray*}
&\forall i\in [\![1,I]\!], ~~v_{i}(0,x)=g_i(x)~~ \text{ if } x\in[0,R],\\
&v_i^n(t,x)~~=~~u_{i,k}(x)+n(t-t_k^n)(u_{i,k+1}(x)-u_{i,k-1}(x))~~\text{ if } (t,x)\in [t_k^n,t_{k+1}^n)\times[0,R].
\end{eqnarray*}
Hence, the uniform upper bounds in \eqref{eq : bornes glob} yield that there exists a constant $M_1$ independent of $n$, depending only on the data of the system, such that for all $i\in[\![1,I]\!]$:
\begin{eqnarray*}
|v_i^n|_{[0,T]\times[0,R]}^{\alpha}~~+~~|\partial_xv_i^n|_{x,[0,T]\times[0,R]}^{\alpha}~~\leq~~ M_1.
\end{eqnarray*}
Using Lemma \ref{lm : cont deriv temps grad}, we deduce that there exists a constant $M_2(\alpha)>0$, independent of $n$, such that for all $i\in[\![1,I]\!]$, we have the following global H\"{o}lder condition:
\begin{eqnarray*}
|\partial_xv_i^n|_{t,[0,T]\times[0,R]}^{\frac{\alpha}{2}}~~+~~|\partial_xv_i^n|_{x,[0,T]\times[0,R]}^{\alpha}~~\leq ~~M_2(\alpha).
\end{eqnarray*}
We deduce then from Ascoli's Theorem that up to a sub sequence denoted in the same way by $n$, $(v_i^n)_{n\ge 0}$ converges in $\mathcal{C}^{0,1}([0,T]\times[0,R])$ to $v_i$, and then $v_i\in \mathcal{C}^{\frac{\alpha}{2},1+\alpha}([0,T]\times[0,R])$.
Since $v_{n}$ satisfies the following continuity condition at the junction point:
\begin{eqnarray*}
\forall (i,j)\in [\![1,I]\!]^2,~~\forall n\ge 0,~~\forall t\in[0,T],~~v_i^n(t,0)=v_j^n(t,0),
\end{eqnarray*}
we deduce then $v\in \mathcal{C}^{\frac{\alpha}{2},1+\alpha}\big([0,T]\times \mathcal{N}_R\big)$. 
We now focus on the regularity of $v$ at the interior of each ray $\mathcal{R}_i$. We prove that $v\in \mathcal{C}^{1+\frac{\alpha}{2},2+\alpha}\big((0,T)\times\overset{\circ}{\mathcal{N}^*_R})$ and satisfies on each edge: 

$\forall\,(t,x) \in (0,T)\times (0,R)$,
\begin{eqnarray*}
\partial_tv_i(t,x)-a_i(t,x)\partial_x^2v_i(t,x)+b_i(t,x)\partial_xv_i(t,x)+v_i(t,x)c_i(t,x)=f_i(t,x).
\end{eqnarray*}
Using once again \eqref{eq : bornes glob}, there exists a constant $M_3$ (independent of $n$) such that for each $i\in[\![1,I]\!]$:
\begin{eqnarray*}
\|\partial_{t}v_i^n\|_{L^{2}\big((0,T)\times (0,R)\big)}~~\leq~~ M_3,~~\|\partial_x^2v_i^n\|_{L^{2}\big((0,T)\times (0,R)\big)}~~\leq~~ M_3.
\end{eqnarray*}
Hence, we get up to a sub sequence denoted abusively using the same subscript $n$:
\begin{eqnarray*}
\partial_{t}v_i^n~~{\rightharpoonup}~~\partial_{t}v_i,~~\partial_x^2v_i^n~~{\rightharpoonup}~~\partial_x^2v_i,
\end{eqnarray*}
weakly in $L^2\big((0,T)\times (0,R)\big)$. Denote by $\mathcal{C}_c^\infty\big((0,T)\times (0,R)\big)$, the set of infinite continuous differentiable functions on $(0,T)\times (0,R)$ with compact support. We obtain therefore that, $\forall \psi \in \mathcal{C}_c^\infty\big((0,T)\times (0,R)\big)$:
\begin{eqnarray*}
& \displaystyle\int_0^T\!\!\!\int_0^{R}\croc{\Big(\partial_tv_i^n-a_i\partial_x^2v_i^n+b_i\partial_xv_i^n+v_i^nc_i-f_i\Big)\psi}(t,x)dxdt\\
&\xrightarrow[]{n\to +\infty}~~
\\&\displaystyle\int_0^T\!\!\!\int_0^{R}\croc{\Big(\partial_tv_i-a_i\partial_x^2v_i+b_i\partial_xv_i+v_ic_i-f_i\Big)\psi}(t,x)dxdt.
\end{eqnarray*}
We now prove that for any $\psi\in\mathcal{C}_c^\infty\big((0,T)\times (0,R)\big)$:
\begin{eqnarray*}
&\displaystyle\int_0^T\!\!\!\int_0^{R}\croc{\Big(\partial_tv_i^n-a_i\partial_x^2v_i^n+b_i\partial_xv_i^n+c_iv_i^n-f_i\Big)\psi}(t,x)dxdt\\
&~~\xrightarrow[]{n\to +\infty}~~0.
\end{eqnarray*}
Using that $(u^{k})_{k\in [\![1,n]\!]}$ is the solution of \eqref{eq: schema ell} and satisfies on each ray $\mathcal{R}_i$
$$n(u_i^k(x)-u_{i}^{k-1}(x))-a_i(t_k,x)\partial_x^2u_i^k(x)+
b_{i}(t_k,x)\partial_xu_i^k(x)+c_{i}(t_k,x)u_i^k(x)-f_i(t_k,x)=0,$$
we obtain:
\begin{eqnarray*} 
&\Big|\displaystyle\int_0^T\!\!\!\int_0^{R}\croc{\Big(\partial_tv_i^n-a_i\partial_x^2v_i^n+b_i\partial_xv_i^n+c_iv_i^n-f_i\Big)\psi}(t,x)dxdt\Big|\\
&\;=\Big|\ds \sum_{k=0}^{n-1}\displaystyle\int_{t_k^n}^{t_{k+1}^n}\!\!\!\int_0^{R}\psi(t,x)\Big(\croc{-a_i\partial_x^2v_i^n+b_i\partial_xv_i^n+c_iv_i^n-f_i}(t,x)\\&
-\pare{-a_i(t_{k+1},x)\partial_x^2u_i^{k+1}(x)+
b_{i}(t_{k+1},x)\partial_xu_i^{k+1}(x)+c_{i}(t_{k+1},x)u_i^{k}(x)-f_i(t_{k+1},x)}\Big)dxdt\Big|.
\end{eqnarray*}
Using assumption $(\mathcal{H}^{'})$, the H\"{o}lder equicontinuity in time of $(v_i^n,\partial_xv_i^n)$, we obtain that there exists a constant $M_4(\alpha)$ independent of $n$ such that: 

$\forall\,i\in [\![1,I]\!]$, $\forall\,(t,x)\in [t_k^n,t_{k+1}^n]\times[0,R]$,
\begin{eqnarray*}
&\Big|\croc{b_i\partial_xv_i^n+c_iv_i^n-f_i}(t,x)-\big(
b_{i}(t_{k+1},x)\partial_xu_i^{k+1}(x)+c_{i}(t_{k+1},x)u_i^{k}(x)-f_i(t_{k+1},x)\big)\Big|\\
&\leq M_4(\alpha)(t-t_k^n)^{\frac{\alpha}{2}}\leq M_4(\alpha)/{n^\frac{\alpha}{2}}.
\end{eqnarray*}
For the Laplacian term, we write, for all $i\in [\![1,I]\!]$, for each $(t,x)\in (t_k^n,t_{k+1}^n)\times(0,R)$:
\begin{eqnarray*}
&a_i(t_{k+1},x)\partial_xu_{i,k+1}(x)-a_i(t,x)\partial_x^2v_i^n(t,x)~~=\\&
\big(a_i(t_{k+1},x)-a_i(t,x)\big)\partial_x^2u_{i}^{k+1}(x)+
a_i(t,x)\big(\partial_x^2v_i^n(t_{k+1},x)-\partial_x^2v_i^n(t,x)\big).
\end{eqnarray*}
Using again the H\"{o}lder equicontinuity in time of $(v_i^n,\partial_xv_i^n)$, the uniform bound on $|\partial_x^2u_{i,k}|_{[0,R]}$ and that the coefficients $a_i$ are almost everywhere differentiable with respect to the variable $x$, we obtain with an integration by parts:

for any $\psi\in\mathcal{C}_c^\infty\big((0,T)\times (0,R)\big)$,
\begin{eqnarray*}
&\Big|\displaystyle\sum_{k=0}^{n-1}~~\displaystyle\int_{t_k^n}^{t_{k+1}^n}\!\!\!\int_0^{R}\Big(a_i(t_{k+1},x)\partial_xu_{i,k+1}(x)-a_i(t,x)\partial_x^2v_i^n(t,x)\Big)\psi(t,x)dxdt\Big|\\
&\leq  \Big|\displaystyle\sum_{k=0}^{n-1}~~\displaystyle\int_{t_k^n}^{t_{k+1}^n}\!\!\!\int_0^{R}\big(a_i(t_{k+1},x)-a_i(t,x)\big)\partial_x^2u_{i}^{k+1}(x)\psi(t,x)dxdt\Big|+\\
&\Big|\displaystyle\sum_{k=0}^{n-1}~~\displaystyle\int_{t_k^n}^{t_{k+1}^n}\!\!\!\int_0^{R}\partial_xa_i(t,x)\big(\partial_xv_i^n(t_{k+1},x)-\partial_xv_i^n(t,x)\big)\partial_x\psi(t,x)dxdt\Big|~~\xrightarrow[]{n\to +\infty}~~0.
\end{eqnarray*}
We conclude that for any $\psi\in\mathcal{C}_c^\infty\big((0,T)\times (0,R)\big)$,
\begin{eqnarray*}
\displaystyle\int_0^T\!\!\!\int_0^{R}\croc{\Big(\partial_tv_i-a_i\partial_x^2v_i+b_i\partial_xv_i+c_iv_i-f_i\Big)\psi}(t,x)dxdt=0.
\end{eqnarray*}
Using Theorem III.12.2 of \cite{pde para}, we get finally that for all $i\in[\![1,I]\!]$, $v_i\in \mathcal{C}^{1+\frac{\alpha}{2},2+\alpha}\big((0,T)\times (0,R)\big)$, which means that $v\in \mathcal{C}^{1+\frac{\alpha}{2},2+\alpha}\big((0,T)\times \overset{\circ}{\mathcal{N}_R^*}\big) $, and we deduce that $v_i$ satisfies on each edge:
$\forall (t,x) \in (0,T)\times (0,R),$
\begin{eqnarray*}
\partial_tv_i(t,x)-a_i(t,x)\partial_x^2v_i(t,x)+b_i(t,x)\partial_xv_i(t,x)+c_i(t,x)v_i(t,x)=f_i(t,x).
\end{eqnarray*}
Remark now, from the estimates \eqref{eq : bornes glob}, that $\partial_tv_i^n$ and $\partial_x^2v_i^n$ are uniformly bounded in $n$. Since $t\mapsto\partial_tv_i(t,x)\in \mathcal{C}\big((0,T)\big)$ and $t\mapsto v_i(t,x)$ is Lipschitz continuous on $[0,T]$ uniformly w.r.t. $x\in [0,R]$ (this can be seen because $t\mapsto v_i^n$ is equi-Lipschitz continuous and there is uniform $\mathcal{C}^{0,1}$ convergence of $v_i^n$ to $v_i$), we obtain that  $t\mapsto \partial_tv_i(t,x)$ is bounded on $(0,T)$ uniformly w.r.t $x\in [0,R]$. Therefore, $\partial_tv_i\in L_\infty\big((0,T)\times(0,R)\big)$. The same argument may be used to obtain $\partial_x^2v_i\in L_\infty\big((0,T)\times(0,R)\big)$. We conclude finally that $v\in \mathcal{C}^{1+\frac{\alpha}{2},2+\alpha}\big((0,T)\times \overset{\circ}{\mathcal{N}_R^*}\big)$ with bounded derivatives $\partial_tv_i$ and $\partial_x^2v_i$ in $(0,T)\times (0,R)$ ($i\in [\![1,I]\!]$).

Close arguments would lead us to show that $v$ satisfies the linear Kirchhoff's boundary condition at the junction point $\{0\}$:
$$-\lambda(t)v(t,0)+\sum_{i=1}^I\alpha_i(t)\partial_xv_i(t,0)=\gamma(t),~~t\in(0,T).$$

Finally, the expression of the upper bounds of the partial derivatives of $v$ are direct consequences of Propositions \ref{prop:uniform-bound}, \ref{Prop:estimation-time-derivative}, \ref{pr : borne deriv temps en 0}, and \ref{prop:non-Bernstein-estimate}.
\end{proof}

%%%%%%%%%%%%%%%%%%%%%%%%%%%%%%%%%%%%%%%%%%%%%%%%%%%%%%%%%%%%%%%%%%%%%%%%%%%%%%%%%%%%%%%%%%%%%%%%%%%%%%%%%%%%%%%%%%%%%%%%%%%%%

\section{Proof of the main result}\label{sec: preuve résultat principal}
In this entire section, we work under the assumption $(\mathcal{H})$ for the data $\mathfrak{D}$.

Let $n\in \mathbb{N}^*$. 

We introduce the following grid of $[0,K]$~:~ ${\mathcal{G}^n_K}:=\{l_p:=\frac{Kp}{n}\,|\,p\in [\![0,n]\!]\}$. 

We consider the following sequence $(u^p)_{p\in [\![0,n]\!]}$ built by induction, constructed so that at each step $p\in [\![0,n-1]\!]$, $u_p$ solves the following backward parabolic scheme (in the variable $l$) on the star-shaped network $\mathcal{N}_R$:      \begin{eqnarray}\label{eq: schema para}
{\mathcal P}_p:~
\begin{cases}
\forall i\in[\![1,I]\!],~~\partial_tu_i^p(t,x)-a_i(t,x,l_p)\partial_x^2u_i^p(t,x)+b_{i}(t,x,l_p)\partial_xu_i^p(x) +\\ c_i(t,x,l_p)u_i^p(t,x)=f_i(t,x,l_p)~~ \text{ if } (t,x) \in (0,T)\times (0,R),\\ 
\ds n(u^{p+1}(t,0)-u^{p}(t,0))+\sum_{i=1}^I {\alpha}_i(t,l_p)\partial_x u_i^p(t,0)-r(t,l_p)u^{p}(t,0)=\phi(t,l_p)+\beta_p^n,\\
\forall i\in[\![1,I]\!],~~\partial_xu_i^p(t,R)=0,~~t \in(0,T),\\
\forall (i,j)\in[\![1,I]\!]^2,~~u_i^p(t,0)=u_{j}^p(t,0):=u^p(t,0),~~t\in(0,T),\\
\forall i\in[\![1,I]\!],~~u_i^p(0,x)=g_i(x,l_p),~~ x\in [0,R].
\end{cases}
\end{eqnarray}
The sequence $(u^p)_{p\in [\![0,n]\!]}$ is initialized with the initial backward condition
$$u^n=\psi.$$
The family of constants $(\beta_p^n)$ is fixed by:
\begin{eqnarray}\label{eq const beta compatibil}
\forall p\in[\![0,n-1]\!],~~\beta_p^n=n(g(0,l_{p+1})-g(0,l_p))-\partial_lg(0,l_p)
\end{eqnarray}
in order to obtain the compatibility condition of Theorem \ref{th: ex sys para} at the junction point $x=\{0\}$, assumption $(\mathcal{H}^{'})$ b)-(i); recall also that from assumption  $(\mathcal{H})$ b)-i) the following compatibility condition holds:
$$ \partial_lg(0,l)+\sum_{i=1}^I\alpha_i(0,l)\partial_xg_i(0,l)-r(0,l)g(0,l)=\phi(0,l),~~l\in[0,K).$$
For any $p\in [\![0,n-1]\!]$ let us define for a while
\begin{eqnarray}
\label{eq:gamma_p}
&\gamma_p: t\mapsto \phi(t,l_p)+\beta_p^n-nu^{p+1}(t,0),\\
\label{eq:l_p}
&\lambda_p:t\mapsto n+r(t,l_p),
\end{eqnarray}
that are both in the class $W^{1,\infty}\big([0,T]\big)$.

Under assumption $(\mathcal H)$, we are in position to apply the result from Theorem \ref{th: ex sys para} iteratively at each step $p$ varying from $n-1$ to $0$ and show that the parabolic system ${\mathcal P}_p$ admits a unique solution $u^p$ in the class $ \mathcal{C}^{\frac{\alpha}{2},1+\alpha}\big([0,T]\times\mathcal{N}_R\big)\cap\mathcal{C}^{1+\alpha,2+\alpha}\big((0,T)\times \overset{\circ}{\mathcal{N}_R^*}\big)$.
\medskip

We start first by getting uniform bounds for the derivatives of $u^p$ and also for the term $n|u^p-u^{p-1}|$.
%%%%%%%%%%%%%%%%%%%%%%%%%%%%%%%%%%%%%%%%%%%%%

\textbf{In whole remaining of this section, we fix $L:= \pare{\lfloor |r|_{\infty}\rfloor + 1} \vee |r|^2_\infty$).}
\subsection{Uniform bound for $|u^p|_\infty$}
\begin{Proposition}
\label{Prop : borne-unform-sec4}
We have:
\begin{equation}
\label{eq:def-M0}
\sup_{n\geq L}\max_{i\in[\![1,I]\!]}  
\max\limits_{p\in[\![0,n]\!]} |u_i^p|_{\infty} \leq M_0
\end{equation}
with
\begin{equation}
\label{def-M0}
M_0:=\pare{|g|_\infty + |\psi|_\infty+|f|_\infty + |\phi|_\infty + 2|\partial_lg(0,\cdot)|_\infty}{\rm e}^{|r|_\infty + 1}{\rm e}^{T(|c|_\infty + 1)}.
\end{equation}
\end{Proposition}

\begin{proof}
We will show by induction that, for a well chosen constant $B_p$ defined by induction, the following continuous map:
$$\kappa_p^n~:~(t,(x,i))\mapsto B_p{\rm e}^{t(|c|_\infty + 1)},~~p\in [\![0,n-1]\!],$$
is a super solution (in the sense given in \cite{Ohavi PDE}) of the parabolic system $\mathcal{P}_p$ posed on the junction network.

Let us first choose $B_p\geq \max_{i\in [\![1,I]\!]}\Big\{|g_i|_\infty+|\psi_i|_\infty \Big\}$ in order to guarantee $\kappa_p^n(0,\cdot)\ge u^p(0,\cdot)$  for any $p\in [\![0,n]\!]$ together with $\kappa_n^n\ge u^n$.

At $x=R$, the Neumann condition is trivially satisfied for $\kappa_p^n$.

At the junction point, we proceed by induction to construct the family $(B_p{\rm e}^{t(|c|_\infty + 1)})_{p\in [\![0,n-1]\!]}$. The condition at the junction point is satisfied whenever
$$\forall t\in[0,T],~~p\in [\![0,n-1]\!]~~B_p[1-\frac{|r|_\infty}{n}]\ge B_{p+1} + |\phi|_\infty{\rm e}^{-t(|c|_\infty + 1)}  + |\beta_p^n|.$$
(Note the crucial importance of the sign in front of $n(u^{p-1} - u^{p})$ in \eqref{eq: schema para} at this step of the reasoning).
Hence, using the expression of the constant $\beta_p^n$, we choose $B_p$ satisfying also
$$B_p\ge \pare{|g|_\infty + |\phi|_ \infty + 2|\partial_lg(0,\cdot)|_\infty}{\rm e}^{|r|_\infty + 1},$$
(note that the right-hand side of the inequality is finite in regard to our assumptions.)

On each ray $\mathcal{R}_i$ and for all $(t,x) \in (0,T)\times (0,R)$ :
\begin{eqnarray*}
&\partial_t\kappa_p^n(t,x)-a_i(t,x,l_p)\partial_x^2\kappa_p^n(t,x)+b_{i}(t,x,l_p)\partial_x\kappa_p^n(x) +c_i(t,x,l_p)\kappa_p^n(t,x)-f_i(t,x,l_p)\\
&=\ds (|c|_\infty + 1)\kappa_p^n(t,x) +c_i(t,x,l_p)\kappa_p^n(t,x)-f_i(t,x,l_p)\\
&\ge (|c|_\infty + 1)\kappa_p^n(t,x) - |c|_{\infty}\kappa_i(t,x) - |f|_{\infty}\\
&\geq B_p - |f|_{\infty}
\end{eqnarray*}
which remains positive as long as $B_p \geq |f|_{\infty}$.

In regard of all the previous conditions, we may then set the following constant 
$$B_p=M_0:=\pare{|g|_\infty + |f|_\infty + |\phi|_ \infty + 2|\partial_lg(0,\cdot)|_\infty}{\rm e}^{|r|_\infty + 1}{\rm e}^{T(|c|_\infty + 1)},$$
independent of $p$ and $n$ in the expression of the function of $\kappa_p^n$.
Gathering all the previous facts ensures that $\kappa_p^n$ is a super solution. Similar arguments hold true for a construction of a sub solution of the form $(t,(x,i))\mapsto -M_0{\rm e}^{t(|c|_\infty + 1)}$ with the same constant $M_0$, which proves the result by application of the parabolic comparison theorem adapted to junction networks (see Theorem 2.4 in \cite{Ohavi PDE}).
\end{proof}
\subsection{Uniform bound for the Lipschitz constant $|u^p(\cdot,0)|_{\lfloor 
W^{1,\infty}([0,T])\rfloor}$ at the junction point.}

\begin{Proposition}
\label{Prop : born deriv en temps 0-sec4}
\begin{eqnarray}\label{eq : born deriv en temps 0}
&\sup\limits_{n\geq L}\max\limits_{p\in [\![0,n]\!]}|u^p(\cdot,0)|_{\lfloor W^{1,\infty}([0,T])\rfloor}\leq M_1
\end{eqnarray}
with
\begin{eqnarray*}
M_1:=&\Big[M_{0}|r|_{\lfloor 
W^{1,\infty}([0,T])\rfloor}+ |\phi|_{\lfloor 
W^{1,\infty}([0,T])\rfloor}\Big]\vee M(g)+|\psi(\cdot,0)|_{\lfloor 
W^{1,\infty}([0,T])\rfloor}\vee M(g),
\end{eqnarray*}
where
$$M(g):=|a|_\infty|\partial_x^2g|_\infty + |b|_\infty|\partial_x g|_\infty + |c|_\infty|g|_\infty + |f|_\infty.$$
\end{Proposition}
\begin{proof}
Recall that from Theorem \ref{th: ex sys para}, we have that $t\mapsto u^p(t,0)\in W^{1,\infty}\big([0,T]\big)$. 
Note also that the constant $M(g)$ of the statement corresponds to the constant \eqref{eq:C(g)} of Proposition \ref{pr : borne deriv temps en 0}, but taking now into account the parameter $l$.

For $p\in [\![0,n-
1]\!]$ we make use of the estimate \eqref{eq:estim-deriv-zero} in Theorem \ref{th: ex sys para} using the definitions \eqref{eq:gamma_p}-\eqref{eq:l_p} and $\underline{\lambda} = n$ coming from the problem $\mathcal{P}_p$. We have the corresponding 
$$
\Theta_0(\lambda_p, \gamma_p)\leq \Big[\frac{M_{0}|r|_{\lfloor 
W^{1,\infty}([0,T])\rfloor}+ |\phi|_{\lfloor 
W^{1,\infty}([0,T])\rfloor}}{n}+|u^{p+1}(\cdot,0)|_{\lfloor 
W^{1,\infty}([0,T])\rfloor}\Big]
$$
so that
\begin{align*}
&|u^p(\cdot,0)|_{\lfloor 
W^{1,\infty}([0,T])\rfloor}\\
&\leq \ds \Big[\frac{M_{0}|r|_{\lfloor 
W^{1,\infty}([0,T])\rfloor}+ |\phi|_{\lfloor 
W^{1,\infty}([0,T])\rfloor}}{n}+|u^{p+1}(\cdot,0)|_{\lfloor 
W^{1,\infty}([0,T])\rfloor}\Big]\vee M(g)\\
&\leq \ds \Big[\frac{M_{0}|r|_{\lfloor 
W^{1,\infty}([0,T])\rfloor}+ |\phi|_{\lfloor 
W^{1,\infty}([0,T])\rfloor}}{n}\Big]\vee M(g)+|u^{p+1}(\cdot,0)|_{\lfloor 
W^{1,\infty}([0,T])\rfloor}\vee M(g),
\end{align*}
where we made use of the inequality
$$\forall (x,y,z)\in \R_+^3,~~(x+y)\vee z\leq x\vee z+y\vee z.$$
Clearly, by induction we get: for all $p \in [\![0,n-1]\!]$,
$$ |u^p(\cdot,0)|_{\lfloor 
W^{1,\infty}([0,T])\rfloor}\leq \Big[M_{0}|r|_{\lfloor 
W^{1,\infty}([0,T])\rfloor}+ |\phi|_{\lfloor 
W^{1,\infty}([0,T])\rfloor}\Big]\vee M(g)+|u^{n}(\cdot,0)|_{\lfloor 
W^{1,\infty}([0,T])\rfloor}\vee M(g),$$
which proves the result.
\end{proof}

\subsection{Uniform bound for the gradient $|\partial_x u^p|_\infty$}
\begin{Proposition}
\label{Prop : borne-uniform-gradient-sec4}
There exist a constant $M_2$ depending only on $R$, $T$ and the data $\mathcal{D}$ such that
\begin{eqnarray}\label{eq : born grad globale}
\sup_{n\geq L}\max_{p\in [\![0,n]\!]}|\partial_x u^p|_\infty \leq M_2\vee |\partial_x\psi|_\infty.
\end{eqnarray}
\end{Proposition}
\begin{proof}
From the result of Theorem \ref{th: ex sys para}, there are constants $L_p$ such that
$$\forall p\in [\![0,n-1]\!],~~|\partial_x u^p|_\infty \leq L_p.$$
Since the coefficients and their weak derivatives are uniformly bounded by $l$, and $|u^p|$ is uniformly bounded by $M_0$, we see from the expression of $C_2$ given in Theorem \ref{th: ex sys para} that we can choose $L_p$ not depending on $p$: the announced result follows with help of the previous Proposition \ref{Prop : born deriv en temps 0-sec4} and the bound \eqref{eq : born deriv en temps 0}.
\end{proof}
\subsection{Uniform bound for the time derivative $|\partial_tu^p|_\infty$.} 
Finding an uniform bound of the time derivative may be done by  adapting standard arguments, but it would be long and tedious to write a proof in full detail. For the sake of conciseness, we will only outline the proof.

The idea is to follow the arguments exposed in Theorem 2.2 VI in the monograph \cite{pde para}. More precisely, Theorem 2.2 VI in \cite{pde para} states that if $v\in \mathcal{C}^{1,2}([0,T]\times[0,R]\big)$ is some (strong) solution of 
$$\partial_tv(t,x)-a(t,x)\partial_x^2v(t,x)+b(t,x)\partial_x v(t,x) +\\ c(t,x)v(t,x)=f(t,x),~~ (t,x)\in (0,T)\times (0,R),$$
in the context where the coefficients $(a,b,c,f)$ are continuously differentiable and $a\ge \underline{a}>0$ is elliptic, then
$$\sup \Big\{|\partial_tv(t,x)|,~(t,x)\in [0,T]\times [0,R]\Big\}$$
can be estimated in terms of the quantities $\sup \Big\{ |v(t,x)|, (t,x)\in [0,T]\times [0,R]\Big\}$,\\
$\sup \Big\{ |\partial_xv(t,x)|,(t,x)\in [0,T]\times [0,R]\Big\},$ 
the ellipticity constant $\underline{a}$, the upper bounds of the coefficients $(a,b,c,f)$ and their derivatives, and the supreme of $|\partial_t v|$ at the boundary, namely $\sup \Big\{ |\partial_t v(t,0)|, t\in [0,T] \Big\}$ and $\sup \Big\{ |\partial_t v(t,R)|,t\in [0,T]\Big\}$. 

For the solution $u^p \in \mathcal{C}^{\frac{\alpha}{2},1+\alpha}\big([0,T]\times\mathcal{N}_R\big)\cap\mathcal{C}^{1+\frac{\alpha}{2},2+\alpha}\big((0,T)\times \overset{\circ}{\mathcal{N}_R^*}\big)$ of our parabolic scheme $(\mathcal{P}_p)$ (in \eqref{eq: schema para}), we cannot apply directly the result of Theorem 2.2 VI in \cite{pde para} on each branch because the coefficients involved in $(\mathcal{P}_p)$ and the values of $u^p$ at the boundary $x=0,~x=R$ possess only a Lipschitz continuous regularity w.r.t. the time variable $t$. However, we may smooth by convolution the terms $t\mapsto u^p(t,0),~~t\mapsto u^p(t,R)$ together with the coefficients $(a,b,c,f)$. Then -- using standard notations for the convolutions with $\varepsilon$ as upper index -- we may consider a solution $w^{p}_{\varepsilon}$ on each ray $\mathcal{R}_i$ of 
\begin{eqnarray*}
&\partial_tw^{p}_{\varepsilon, i}(t,x)-a_i^\varepsilon(t,x,l_p)\partial_x^2w^{p}_{\varepsilon,i}(t,x)+b_i^\varepsilon(t,x,l_p),\partial_x w^{p}_{\varepsilon,i}(t,x) +\\ &c_i^\varepsilon(t,x,l_p)w^{p}_{\varepsilon,i}(t,x)=f_i^\varepsilon(t,x,l_p), (t,x)\in (0,T)\times (0,R),
\end{eqnarray*}
with smooth Dirichlet boundary conditions $u^{p,\varepsilon}(\cdot,0),u^{p,\varepsilon}(\cdot,R)$ on the time-edge of the parabolic cylinder. Well-known results (see for example Theorem 3.4' in \cite{pde para}) ensure that the solution $w^{p}_{\varepsilon}$ satisfies the conditions of Theorem 2.2 VI in \cite{pde para}. 

Now note that the smoothed data is uniformly bounded w.r.t. $\varepsilon$ in $\mathcal{C}^1$ norm ; namely using transparent notations,
$$\forall \varepsilon >0,~~|a^\varepsilon,b^\varepsilon,c^\varepsilon,f^\varepsilon,u^{p,\varepsilon}(\cdot,0),u^{p,\varepsilon}(\cdot,R)|_{\mathcal{C}^1}\leq |a,b,c,f,u^{p}(\cdot,0),u^p(\cdot,R)|_{W^{1,\infty}}.$$
Also, it is easy to check after some lines of calculation - for e.g. using arguments similar to those in the proof of Theorem 2.2 VI of \cite{pde para} but in our much simplest case - that  $w^{p}_{\varepsilon,i}$ converges to $u_i^p$ in $\mathcal{C}^{0,1}\big([0,T]\times [0,R]\big)\cap \mathcal{C}^{1,2}\big((0,T)\times (0,R)\big)$. Therefore, - by using first the convergence in $\mathcal{C}^{1,2}\big((0,T)\times (0,R)\big)$ - we have that for any compact $[\eta, \tau]\times \mathcal{K}\subset (0,T)\times (0,R)$ $|\partial_tu_i^p|_{L^{\infty}([\eta, \tau]\times \mathcal{K})}$ is also estimated in terms of the quantities $|a,b,c,f,u^{p}(\cdot,0),u^p(\cdot,R)|_{W^{1,\infty}}$, $|\partial_xu^p|_{\infty}$, and $|u^p|_{\infty}$ ; the same holds true for $|u_i^p(.,R)|_{\lfloor W^{1,\infty}([\eta, \tau])\rfloor}$. We refer to equation 2.6 in the proof of Theorem 2.2 VI in \cite{pde para} for the exact expression of the upper bound that is uniform w.r.t $[\eta, \tau]\times \mathcal{K}$: recall that $|\partial_xu^p|_{\infty}$ and $|u^p|_{\infty}$ are uniformly bounded by $p$ and that we have obtained an uniform bound for $|u^p(\cdot,0)|_{\lfloor W^{1,\infty}([0,T])\rfloor}$ in Proposition \ref{Prop : born deriv en temps 0-sec4}. 
By the same arguments as those given in Proposition \ref{pr : borne deriv temps en 0} our line of reasoning takes care of the Neumann boundary condition.

As a first conclusion, we claim that there exists a constant $M_3$ depending on $M_0,M_1,M_2,$ and $|a,b,c,f|_{W^{1,\infty}}$ such that for every compact $[\eta, \tau]\times \mathcal{K}\subset (0,T)\times (0,R)$:
\begin{eqnarray}\label{eq : born deriv temps globale}
\max_{i\in [\![1,I]\!]}\max_{p\in [\![0,n]\!]}\pare{|\partial_t u_i^p|_{L^{\infty}([\eta, \tau]\times \mathcal{K})} + |u_i^p(.,R)|_{\lfloor W^{1,\infty}([\eta, \tau])\rfloor}}\leq M_3\vee |\partial_t \psi|_\infty,
\end{eqnarray}
which by the uniformity of $M_3$, implies
$$\max_{i\in [\![1,I]\!]}\max_{p\in [\![0,n]\!]}\sup \Big\{ |\partial_t u^p_i(t,x)|,\;\;~(t,x)\in (0,T)\times (0,R)\Big\}\leq M_3\vee |\partial_t \psi|_\infty.$$
Similarly, we can check that there exists a constant $M$ independent of $\varepsilon$, depending
only on the data of the system, such that for all $i \in [\![1,n]\!]$,
$$|w_{\varepsilon,i}^p |^{\alpha}_{[0,T]\times [0,R]} + |\partial_x w_{\varepsilon,i}^p |^\alpha_{x,[0,T ]\times [0,R]} \leq M.$$
Now making use of the convergence in $\mathcal{C}^{0,1}\big([0,T]\times [0,R]\big)$ of $w^{p}_{\varepsilon,i}$ towards $u_i^p$ in $\mathcal{C}^{0,1}\big([0,T]\times [0,R]\big)$, we have the bound in H\"older norms
$$|u_{i}^p |^{\alpha}_{[0,T]\times [0,R]} + |\partial_x u_{i}^p |^\alpha_{x,[0,T]\times [0,R]} \leq M.$$
From the result of Propositions \ref{pr : borne deriv temps en 0} and \ref{Prop:estimation-time-derivative} or using a standard interpolation lemma,
we get the following: 
\begin{Proposition}
\label{Prop : borne-uniform-lipschitz-constant-sec4}
There exists a constant $M_3=M\Big(M_0,M_1,M_2,|a,b,c,f|_{W^{1,\infty}}\Big)$ such that
$$
\sup_{n\geq L}\max_{i\in [\![1,I]\!]}\max_{p\in [\![0,n]\!]}\sup_{x\in [0,R]}|u_i^p(.,x)|_{\lfloor W^{1,\infty}([0,T])\rfloor}\leq M_3\vee |\partial_t\psi|_\infty.
$$
\end{Proposition}

\subsection{Uniform bound for the term $|\partial_x^2u^p|_\infty$.} 
\medskip

For any $i\in [\![1,I]\!]$ and $(t,x) \in (0,T)\times (0,R)$, we have that 
$$
\partial_x^2u_i^p(t,x) = \frac{1}{a_i(t,x,l_p)}\pare{\partial_tu_i^p(t,x)+b_{i}(t,x,l_p)\partial_xu_i^p(x) +\\ c_i(t,x,l_p)u_i^p(t,x) - f_i(t,x,l_p)}. 
$$
So that by direct application of the results of Propositions 4.1 to 4.4, we may state the following result:
\begin{Proposition}
\label{Prop : borne-uniform-derivee-spatiale-seconde-sec4}
$$
\sup_{n\geq L}\max_{p\in [\![0,n]\!]}|\partial_x^2 u^p|_\infty\leq \frac{1}{\underline{a}}\pare{M_3\vee |\partial_t\psi|_\infty + M_2 |b|_\infty + |c|_\infty M_0 + |f|_\infty}.
$$
\end{Proposition}
\subsection{Uniform bound for the term $n|u^{p+1}-u^p|_\infty$.}
Our concern is to obtain an uniform bound for the term:
$$n|u^{p+1}-u^p|_\infty.$$
Importantly, note that we obtain an uniform bound for $n|u^{p+1}-u^p|_\infty$ only for all $p\in [\![0,n-2]\!]$ and not for $p=n-1$ (contrary to the bounds gathered for $|u^p|,|\partial_t u^p|,|\partial_x u^p|$ and $|\partial_x^2 u^p|$ that hold for all ${p\in [\![0,n]\!]}$). Because of the lack of first order compatibility conditions w.r.t $l$ at the boundary, it does not seem reasonable to expect that the bound below should be satisfied for $p=n-1$.

\begin{Proposition}
\label{Prop : borne-uniform-derivee-l}
We have
$$\sup_{n\geq L}\max_{p\in  [\![0,n-2]\!]} n|u^{p+1}-u^p|_{\infty}\leq M_4,$$
with
\begin{align*}M_4:=&\frac{2}{\underline{a}}\Big[M_2|a|_{W^{1,\infty}}+M_1|a|_{W^{1,\infty}}|b|_{W^{1,\infty}}+M_0|a|_{W^{1,\infty}}|c|_{W^{1,\infty}}+|a|_{W^{1,\infty}}|f|_{W^{1,\infty}}\Big]\\
&\hspace{3,4 cm}\vee \Big[ I|\alpha|_\infty M_1+|\psi|_\infty\Big].
\end{align*}
\end{Proposition}
\begin{proof}
We will show by induction that, for a well chosen constant $B_p$ to be produced later and that can be chosen independently of $p\in [\![0,n-2]\!]$ and $n$, the following map
$$ \kappa
_p^n:=(t,(x,i))\mapsto u_i^{p+1}(t,x)+\frac{(t+1)B_p}{n}$$
is a super solution.
At $x=R$, the condition is trivial, whereas at the junction point we remark that it is sufficient to satisfy:
$$\forall t\in[0,T],~~p\in [\![0,n-2]\!];~~B_p(t+1)[1+\frac{|r|_\infty}{n}]\ge \sum_{i=1}^I \alpha_i(t,l_p)\partial_xu^{p+1}_i(t,0)-\phi(t,l_p).$$

Let $p\in [\![0,n-2]\!]$ fixed and first choose $B_p$ satisfying:
$$B_p\ge I|\alpha|_\infty\max_{j\in [\![0,n-1]\!]}|\partial_xu^{j+1}_i(\cdot,0)|_{\infty}+|\phi|_\infty,$$
that is finite in view of our previous estimates.

Making use of the uniform upper bounds obtained for $u^p$ and its derivatives, we are going to see that it is also possible to produce the constant $B_p$ so that the condition on each ray $\mathcal{R}_i$ to be a super solution holds true (condition away from the junction point $\{0\}$). For this purpose, the main ingredient is the ellipticity condition on the coefficients $a_i$ combined with the Lipschitz regularity of the coefficients w.r.t the variable $l$. 
On each ray $\mathcal{R}_i$ and for all $(t,x) \in (0,T)\times (0,R)$, by crucially making use of the fact that $p\in [\![0,n-2]\!]$ for the substitution of $\partial_x^2u_i^{p+1}(t,x)$, we have:
\begin{align*}
&\partial_t\kappa_p^n(t,x)-a_i(t,x,l_p)\partial_x^2\kappa_p^n(t,x)+b_{i}(t,x,l_p)\partial_x\kappa_p^n(x) +c_i(t,x,l_p)\kappa_p^n(t,x)-f_i(t,x,l_p)\\
&=\ds \frac{B_p}{n}+c_i(t,x,l_p)\frac{B_p(t+1)}{n}+ \partial_tu_i^{p+1}(t,x)-a_i(t,x,l_p)\partial_x^2u_i^{p+1}(t,x)+\\
&\hspace{0,4 cm}b_{i}(t,x,l_p)\partial_xu_i^{p+1}(t,x) +c_i(t,x,l_p)u_i^{p+1}(t,x)-f_i(t,x,l_p)\\
&=\ds \frac{B_p}{n}+c_i(t,x,l_p)\frac{B_p(t+1)}{n}+ \partial_tu_i^{p+1}(t,x)\\
&\;\;\ds-\frac{a_i(t,x,l_p)}{a_i(t,x,l_{p+1})}\Big[\partial_tu_i^{p+1}(t,x)+b_{i}(t,x,l_{p+1})\partial_xu_i^{p+1}(t,x) +c_i(t,x,l_{p+1})u_i^{p+1}(t,x)-f_i(t,x,l_{p+1})\Big]\\
&\;\;+b_{i}(t,x,l_p)\partial_xu_i^{p+1}(t,x) +c_i(t,x,l_p)u_i^{p+1}(t,x)-f_i(t,x,l_p).
\end{align*}
Because of the Lipschitz regularity of the coefficients with respect to the variable $l$ and the upper bounds obtained for the derivatives of $u^p$, we have that for all $(t,x) \in (0,T)\times (0,R)$,
\begin{align*}
&\Big|~\partial_tu_i^{p+1}(t,x)
\ds-\frac{a_i(t,x,l_p)}{a_i(t,x,l_{p+1})}\Big[\partial_tu_i^{p+1}(t,x)+b_{i}(t,x,l_{p+1})\partial_xu_i^{p+1}(t,x) +c_i(t,x,l_{p+1})u_i^{p+1}(t,x)\\
&-f_i(t,x,l_{p+1})\Big]+b_{i}(t,x,l_p)\partial_xu_i^{p+1}(t,x) +c_i(t,x,l_p)u_i^{p+1}(t,x)-f_i(t,x,l_p)~\Big|\\
& \ds\leq \frac{1}{n\underline{a}}\Big[|\partial_tu^{p+1}|_{\infty}|a|_{W^{1,\infty}}+2|\partial_xu^{p+1}|_{\infty}|a|_{W^{1,\infty}}|b|_{W^{1,\infty}}+2|u^{p+1}|_{\infty}|a|_{W^{1,\infty}}|c|_{W^{1,\infty}}\Big.\\
&\hspace{3,4 cm}\Big.+2|a|_{W^{1,\infty}}|f|_{W^{1,\infty}}\Big].
\end{align*}
Therefore, by choosing
$$ B_p\ge \frac{1}{\underline{a}}\Big[|\partial_tu^p|_{\infty}|a|_{W^{1,\infty}}+2|\partial_xu^p|_{\infty}|a|_{W^{1,\infty}}|b|_{W^{1,\infty}}+2|u^p|_{\infty}|a|_{W^{1,\infty}}|c|_{W^{1,\infty}}+2|a|_{W^{1,\infty}}|f|_{W^{1,\infty}}\Big],$$
we obtain that $\kappa_p^n$ is a super solution. The same arguments may be applied for a construction of a sub solution of the form $\ds -u_i^{p+1}-\frac{(t+1)B_p}{n}$ with the same constant $B_p$ and we see that this constant can be chosen independent of $p\in [\![0,n-2]\!]$ and $n-1$.

Gathering both facts together yields the announced result. 
\end{proof}
\textbf{Proof of Theorem \ref{th : exis para with l}}.
\begin{proof}
Uniqueness is a direct consequence of the comparison Theorem \ref{th : para comparison th with l}.
%%%%%%%%%%%%%%%%%%%%%%%%%%%%%%%%%%%%%%%%
%%%%%%%%%%%%%%%%%%%%%%%%%%%%%%%%%%%%%%
%%%%%%%%%%%%%%%%%%%%%%%%%%%%%%%%%%%%%%

\textbf{Fix $\overline{K} \in (0,K)$ and $n\geq N_0:=\big(\lfloor |r|_{\infty}\rfloor + 1\big)\vee |r|^2_\infty\vee \displaystyle \pare{\lfloor \frac{K}{K-\overline{K}}\rfloor + 1}$.}

The estimates obtained in Propositions \ref{Prop : borne-unform-sec4} \ref{Prop : born deriv en temps 0-sec4} \ref{Prop : borne-uniform-gradient-sec4} \ref{Prop : borne-uniform-lipschitz-constant-sec4}, \ref{Prop : borne-uniform-derivee-spatiale-seconde-sec4}, \ref{Prop : borne-uniform-derivee-l} for the unique solution $(u^p)_{p\in [\![0,n]\!]}$ of the parabolic scheme $(\mathcal{P}_p)_{p\in [\![0,n]\!]}$ ensure that there exists a constant $M>0$ independent of $n$, such that:  
\begin{eqnarray}\label{eq : bornes glob-2}
&\sup_{n \ge N_0}\left \{\sup_{p\in[\![0,n]\!]}\pare{|u^p|_{\infty},\sup_{x\in \mathcal{N}_R}|u^{p}(.,x)|_{\lfloor W^{1,\infty}\rfloor},
|\partial_xu^{p}|_\infty,|\partial_x^2u^{p}|_\infty}\right .\nonumber\\
&\hspace{7,3 cm}\vee \left .\sup_{p\in[\![1,n-1]\!]}\pare{|n(u^p-u^{p-1})|_{\infty}}\right \}\leq M.~~
\end{eqnarray}
Define the following sequence in $\mathcal{C}^{0,1,0}\big([0,T]\times \mathcal{N}_R \times [0,K]\big)$ of linear interpolating functions  with respect to the variable $l$ : 
\begin{eqnarray*}
&\forall i\in [\![1,I]\!], ~~v_{i}^n(t,x,K)=\psi_i(t,x)~~ \text{ if } (t,x)\in[0,T]\times[0,R],~~\text{and for all}~~p\in [\![0,n-1]\!]:\\
&v_i^n(t,x,l)=u_i^{p}(t,x)+n(l-l_{p}^n)(u_i^{p+1}(t,x)-u_i^{p}(t,x))~~\text{ if } (t,x,l)\in[0,T]\times [0,R]\times[l_{p}^n,l_{p+1}^n).
\end{eqnarray*}

Now, the uniform upper bounds obtained in \eqref{eq : bornes glob-2} together with the upper bound obtained for $|n(u^{p+1}-u^{p})|_{\infty}$ (for all $p\in[\![0,n-2]\!]$) and our choice of $\overline{K}$ are enough to ensure that there exists a constant $B_1$ depending only on the data of the system but independent of $n$, such that for all $i\in[\![1,I]\!]$:
\begin{eqnarray*}
|v_i^n|_{[0,T]\times[0,R]\times [0,\overline{K}]}^{\alpha}~~+~~|\partial_xv_i^n|_{x,[0,T]\times[0,R]\times [0,\overline{K}]}^{\alpha}~~\leq~~ B_1.
\end{eqnarray*}
Using Lemma \ref{lm : cont deriv temps grad}, we deduce that there exists a constant $M_2(\alpha)>0$ that does not depend on $n$, such that for all $i\in[\![1,I]\!]$, we have the following global H\"{o}lder bound :
\begin{eqnarray}\label{eq equi conti suite}
|\partial_xv_i^n|_{t,[0,T]\times[0,R]\times [0,\overline{K}]}^{\frac{\alpha}{2}}~~+~~|\partial_xv_i^n|_{x,[0,T]\times[0,R]\times [0,\overline{K}]}^{\alpha}+|\partial_xv_i^n|_{l,[0,T]\times[0,R]\times [0,\overline{K}]}^{\frac{\alpha}{2}}~~\leq ~~B_2(\alpha).
\end{eqnarray}
Applying Ascoli's Theorem, we have that - up to a sub sequence denoted still abusively by the subscript index $n$ -  $(v_i^n)_{n\ge 0}$ converges in $\mathcal{C}^{0,1,0}([0,T]\times[0,R]\times [0,\overline{K}])$ to $v_i$ and $v_i\in \mathcal{C}^{\frac{\alpha}{2},1+\alpha,\frac{\alpha}{2}}([0,T]\times[0,R]\times [0,\overline{K}])$.
Since $v_{n}$ satisfies the following continuity condition at the junction point:
\begin{eqnarray*}
\forall (i,j)\in [\![1,I]\!]^2,~~\forall n\ge 0,~~\forall (t,l)\in[0,T]\times [0,\overline{K}],~~v_i^n(t,0,l)=v_j^n(t,0,l),
\end{eqnarray*}
we deduce then $v\in \mathcal{C}^{\frac{\alpha}{2},1+\alpha,\frac{\alpha}{2}}\big([0,T]\times \mathcal{N}_R\times [0,\overline{K}]\big)$. Using the arbitrary choice of $\overline{K}\in (0,K)$, we conclude that $v\in \mathcal{C}^{\frac{\alpha}{2},1+\alpha,\frac{\alpha}{2}}\big([0,T]\times \mathcal{N}_R\times [0,K)\big)$. 

We now focus on the regularity of $v$ at the interior of each ray $\mathcal{R}_i$. We aim at proving that $v\in \mathcal{C}^{1+\frac{\alpha}{2},2+\alpha,\frac{\alpha}{2}}\big((0,T)\times\overset{\circ}{\mathcal{N}^*_R}\times (0,K)\big)$ and satisfies:
\\
for all $(t,x,l) \in (0,T)\times (0,R)\times (0,K)$:
\begin{eqnarray*}
\partial_tv_i(t,x,l)-a_i(t,x,l)\partial_x^2v_i(t,x,l)+b_i(t,x,l)\partial_xv_i(t,x,l)+v_i(t,x,l)c_i(t,x,l)- f_i(t,x,l) = 0
\end{eqnarray*}
on the interior of each ray $\mathcal{R}_i$.

Using once again \eqref{eq : bornes glob-2}, there exists a constant $B_3$ independent of $n$, such that for each $i\in[\![1,I]\!]$:
\begin{eqnarray*}
\|\partial_{t}v_i^n\|_{L^{2}((0,T)\times (0,R)\times (0,K))}~~\leq~~ B_3,~~\|\partial_x^2v_i^n\|_{L^{2}((0,T)\times (0,R)\times (0,K))}~~\leq~~ B_3.
\end{eqnarray*}
Hence, we get -- up to a sub sequence denoted abusively in the same way by $n$:
\begin{eqnarray*}
\partial_{t}v_i^n~~{\rightharpoonup}~~\partial_{t}v_i,~~\partial_x^2v_i^n~~{\rightharpoonup}~~\partial_x^2v_i
\end{eqnarray*}
weakly in $L^2\big((0,T)\times (0,R)\times (0,K)\big)$. Denote by $\mathcal{C}_c^\infty\big((0,T)\times (0,R)\times (0,K)\big)$ the set of infinite differentiable functions on $(0,T)\times (0,R)\times (0,K)$ with compact support. We obtain therefore that, $\forall \psi \in \mathcal{C}_c^\infty\big((0,T)\times (0,R)\times (0,K)\big)$:
\begin{eqnarray*}
& \displaystyle\int_0^T\!\!\!\int_0^{R}\!\!\!\int_0^K\Big(\partial_tv_i^n(t,x)-a_i(t,x,l)\partial_x^2v_i^n(t,x)+b_i(t,x,l)\partial_xv_i^n(t,x,l)\\
&+c_i(t,x,l)v_i^n(t,x,l)-f_i(t,x,l)\Big)\psi(t,x,l)dldxdt\\
&\xrightarrow[]{n\to +\infty}\\
&\displaystyle\int_0^T\!\!\!\int_0^{R}\!\!\!\int_0^K\Big(\partial_tv_i(t,x,l)-a_i(t,x,l)\partial_x^2v_i(t,x,l)+b_i(t,x,l)\partial_xv_i(t,x,l)\\
&+c_i(t,x,l)v_i(t,x,l)-f_i(t,x,l)\Big)\psi(t,x,l)dldxdt.
\end{eqnarray*}
We now prove that for any $\psi\in\mathcal{C}_c^\infty\big((0,T)\times (0,R)\times (0,K)\big)$:
\begin{eqnarray*}
&\displaystyle\int_0^T\!\!\!\int_0^{R}\!\!\!\int_0^K\!\Big(\partial_tv_i^n(t,x,l)-a_i(t,x,l)\partial_x^2v_i^n(t,x,l)+b_i(t,x,l)\partial_xv_i^n(t,x,l)\\&+c_i(t,x,l)v_i^n(t,x,l)-f_i(t,x,l)\Big)\psi(t,x,l)dldxdt
~~\xrightarrow[]{n\to +\infty}~~0.
\end{eqnarray*}
Using that $(u^p)_{p \in[\![0,n]\!]}$ is the solution of \eqref{eq: schema para} and satisfies on each ray $\mathcal{R}_i$:
\\
$\forall p \in[\![0,n-1]\!]$ and $\forall (t,x)\in(0,T)\times(0,R),$
\begin{eqnarray*}\partial_tu_i^p(t,x)-a_i(t,x,l_p)\partial_x^2u_i^p(t,x)+
b_{i}(t,x,l_p)\partial_xu_i^p(t,x)+c_{i}(t,x,l_p)u_i^p(t,x)-f_i(t,x,l_p)=0,\end{eqnarray*}
combined with the Lipschitz regularity of the coefficients and free terms $(a,b,c,f)$ w.r.t the variable $l$ (Assumption $\mathcal{H}$) and the uniform upper bound obtained in \eqref{eq : bornes glob-2}, we obtain that there is a constant $B_4$, independent of $n$, such that:
$$\Big|\partial_tv_i^n-a_i\partial_x^2v_i^n+b_i\partial_xv_i^n+v_i^nc_i-f_i\Big|_{\infty}\leq B_4/n.$$
In turn, this leads to the expected result, namely:\\
$\forall \psi\in\mathcal{C}_c^\infty\big((0,T)\times (0,R)\times (0,K)\big)$
\begin{eqnarray*}
&\displaystyle\int_0^T\!\!\!\int_0^{R}\!\!\!\int_0^K\Big(\partial_tv_i(t,x,l)-a_i(t,x,l)\partial_x^2v_i(t,x,l)\\
&+b_i(t,x,l)\partial_xv_i(t,x,l)+v_i(t,x,l)c_i(t,x,l)-f_i(t,x,l)\Big)\psi(t,x,l)dldxdt=0.
\end{eqnarray*}
%\textcolor{red}{il faudrait sans doute détailler un tout petit peu plus}
 Now using the key Lemma \ref{lem : regu Holder interieure}, which gives a result on the interior regularity for weak parabolic solutions that depend on the parameter $l$,
 %using the H\"{o}lder regularity of gradient $\partial_xv$, 
 we conclude that on each ray $\mathcal{R}_i$, $v_i$ belongs to $\mathcal{C}^{1+\frac{\alpha}{2},2+\alpha,\frac{\alpha}{2}}\big((0,T)\times (0,R)\times (0,K)\big)$, ($i\in [\![1,I]\!]$). 
 
Moreover, from the estimates \eqref{eq : bornes glob-2}, recall that $\partial_tv_i^n$ and $\partial_x^2v_i^n$ are uniformly bounded by $n$. Since $t\mapsto\partial_tv_i(t,x,l)\in \mathcal{C}\big((0,T)\big)$ and $t\mapsto v_i(t,x,l)$ is Lipschitz continuous on $[0,T]$ uniformly w.r.t. $(x,l)\in [0,R]\times[0,\overline{K}]$ (this can be seen because $t\mapsto v_i^n$ is equi-Lipschitz continuous and there is uniform $\mathcal{C}^{0,1,0}$ convergence of $v_i^n$ to $v_i$), we obtain that  $t\mapsto \partial_tv_i(t,x,l)$ is bounded on $(0,T)$ uniformly w.r.t variables $x$ and $l$ and independently of $\overline{K}$. Therefore, $\partial_tv_i\in L_\infty\big((0,T)\times(0,R)\times(0,K)\big)$ (using the arbitrary choice of $\overline{K}\in (0,K)$). The same argument may be used to obtain $\partial_x^2v_i\in L_\infty\big((0,T)\times(0,R)\times (0,K)\big)$. We conclude finally that $v\in \mathcal{C}^{1+\frac{\alpha}{2},2+\alpha,\frac{\alpha}{2}}\big((0,T)\times \overset{\circ}{\mathcal{N}_R^*}\times(0,K)\big)$, with bounded derivatives $\partial_tv_i$ and $\partial_x^2v_i$ in $(0,T)\times (0,R)\times(0,K)$ ($i\in [\![1,I]\!]$).

Concerning the derivative of the limit solution $v$ w.r.t the variable $l$, observe once again using \eqref{eq : bornes glob-2} that there exists a constant $B_5$ independent of $n$ such that on each ray $\mathcal{R}_i$:
$$ |\partial_lv_i^n|_{(0,T)\times (0,R)\times (0,K)}\leq B_5.$$
%\textcolor{red}{L\`a encore il faudrait sans doute détailler un tout petit peu plus}

Therefore, for any fixed $q \in (1,+\infty)$ and by reflexivity of $L^q\big((0,T)\times (0,R)\times (0,K)\big)$, we get there exists an extraction sequence $(n^q)$ such that:
$$\partial_lv_i^{n^q}~~{\rightharpoonup}~~\xi_i$$
in $L^q\big((0,T)\times (0,R)\times (0,K)\big)$ where $\xi_i$ denotes an element of $L^q\big((0,T)\times (0,R)\times (0,K)\big)$. Because of the strong convergence of $(v_i^{n^q})$ to $v_i$ and the almost-everywhere uniqueness of weak derivatives, we identify $\xi_i = \partial_lv_i$ $a.e$ in $(0,T)\times (0,R)\times (0,K)$ and $\partial_lv_i\in L^q\big((0,T)\times (0,R)\times (0,K)\big)$. 

This shows that the weak limit $\partial_lv_i$ belongs to $\ds \bigcap_{q\in (1,\infty)}L^q\big((0,T)\times (0,R)\times (0,K)\big)$. Hence, the solution $v$ admits on each ray $\mathcal{R}_i$ ($i\in [\![1,I]\!]$) a generalized derivative with respect to the variable $l$ that belongs to the class $\ds \bigcap_{q\in (1,\infty)}L^q\big((0,T)\times (0,R)\times (0,K)\big)$. 

 Consequently, $v \in  \mathcal{C}^{\frac{\alpha}{2},\alpha,\frac{\alpha}{2}}\big([0,T]\times [0,R]\times [0,K]\big)$.

Let us turn now to the {\it local time Kirchhoff'scondition} at the junction point $\{0\}$ for the limit $v$.

Recall first that for all $p \in [\![0,n-1]\!]$:
$$n(u^{p+1}(t,0)-u^{p}(t,0))+\sum_{i=1}^I {\alpha}_i(t,l^n_p)\partial_x u_i^p(t,0)-r(t,l^n_p)u^{p}(t,0)=\phi(t,l^n_p)+\beta_p^n,~~\forall t\in(0,T).$$
Remark successively:\\
a) $n(u^{p+1}(t,0)-u^{p}(t,0))=\partial_lv^n(t,0,l),~~\forall (t,l)\in [0,T]\times[l_{p}^n,l_{p+1}^n),~~\forall p\in [\![0,n-1]\!]$;\\
b) from the expression of the constant $\beta_n^p$ given in \eqref{eq const beta compatibil},
the Lipschitz continuity of $l\mapsto \partial_lg(l,0)$, we obtain that:
$$\forall (p,q) \in [\![0,n-1]\!]^2,~~|\beta_p^n-\beta_q^n|\leq \frac{C\,K^2}{n},$$
for a constant $C>0$ independent of $(p,q,n)$;\\
c) from $u_i^p(t,0)  = v_i^n(t,0,l_p^n)$, we have: 
$$\displaystyle \sum_{i=1}^I {\alpha}_i(t,l^n_p)\partial_x u_i^p(t,0)-r(t,l^n_p)u^{p}(t,0)=\sum_{i=1}^I {\alpha}_i(t,l^n_p)\partial_x v_i^n(t,0,l_p^n)-r(t,l_p)v^{n}(t,0,l_p^n).$$ 

Therefore, we can conclude together with the equicontinuity of the sequence $(\partial_xv_i^n)$ obtained in \eqref{eq equi conti suite}, the Lipschitz regularity of the coefficients $(r,\phi)$, and the last points a) and b), that $\partial_lv^n(\cdot,0,\cdot)$ satisfies:

$\forall \overline{K}\in (0,K),\forall n \ge N_0,\forall (t,s,l,z)\in [0,T]^2\times ([0,\overline{K}]\setminus {\mathcal{G}^n_K})^2$,
\begin{eqnarray*}
&|\partial_lv^n(t,0,l)-\partial_lv^n(s,0,l)|+|\partial_lv^n(t,0,l)-\partial_lv^n(t,0,z)|\\
&\leq B_6\Big(|t-s|^{\alpha/2}+(|l-z|+2/n)^{\alpha/2}\Big)
,\end{eqnarray*}
where once again the involved constant $B_6$ is independent of $n$.

Hence, applying a generalization of Ascoli's theorem for piecewise continuous functions (we refer for e.g. to Theorem 6.2 in \cite{Droniou-Eymard}), we get that -  up to a subsequence still denoted abusively with the superscript $n$ - $(\partial_lv^n(\cdot,0,\cdot))_n$ converges uniformly to $\partial_lv(\cdot,0,\cdot)$ in $\mathcal{C}^{\frac{\alpha}{2},\frac{\alpha}{2}}\big([0,T]\times [0,\overline{K}]\big)$. This convergence holds for all $\overline{K}\in (0,K)$, yielding that $v(\cdot,0,\cdot)\in \mathcal{C}^{\frac{\alpha}{2},1+\frac{\alpha}{2}}\big([0,T]\times [0,K)\big)$.

Finally, using once again the Kirchhoff's conditions satisfied by the solution $(u_p)_{p\in [\![0,n]\!]}$,
the Lipschitz regularity of the coefficients and free terms $(\phi,r)$ w.r.t the variable $l$ (Assumption $\mathcal{H}$), and the uniform upper bound obtained in \eqref{eq : bornes glob-2}, with the same arguments as those used above, we can show that the limit solution satisfies the required {\it local-time Kirchhoff's condition} at $\{0\}$: 
$\forall (t,l)\in[0,T]\times [0,K)$:
$$\partial_lv(t,0,l)+\displaystyle \sum_{i=1}^I \alpha_i(t,l)\partial_xv_i(t,0,l)-r(t,l)v(t,0,l)=\phi(t,l).$$
We conclude then that the limit $v$ is in the class $\mathfrak{C}^{1+\frac{\alpha}{2},2+\alpha,\frac{\alpha}{2}}_{\{0\}} \big(\Omega_T\big)$ and solves system \eqref{eq : pde with l} of Theorem \ref{th : para comparison th with l}.  
\end{proof}  
\appendix 
\section{Function spaces}\label{sec : functionnal spaces}
Fix $(T,R,K)\in(0,+\infty)^3$.

Given $(p,q,r)\in \mathbf{N}^3$, $\mathcal{C}^{p,q,r}\big([0,T]\times[0,R]\times [0,K]\big)$ is the Banach space whose elements are continuous functions $(t,x,l)\mapsto u(t,x,l)$ in $[0,T]\times[0,R]\times [0,K]$ valued in $\R$, together with all its derivatives of the form $\partial^{i,j,k}_{t,x,l}u$, with $(i,j,k)\in[\![0,p]\!]\times [\![0,q]\!]\times [\![0,r]\!]$.
The norm $\|\cdot\|_{\mathcal{C}^{p,q,r}([0,T]\times[0,R]\times [0,K])}$ is defined for all $u\in \mathcal{C}^{p,q,r}\big([0,T]\times[0,R]\times [0,K]\big)$ by:
\begin{align*}
&\|u\|_{\mathcal{C}^{p,q,r}([0,T]\times[0,R]\times [0,K])}~~=\\
&\displaystyle\sum_{\underset{[\![0,q]\!]\times [\![0,r]\!]}{(i,j,k)\in[\![0,p]\!]\times} }\sup \Big\{|\partial^{i,j,k}_{t,x,l}u(t,x,l))|,~~(t,x,l)\in [0,T]\times[0,R]\times [0,K]\Big\}.
\end{align*}
Let $(\alpha,\beta,\gamma)\in(0,1)^3$. We denote $\mathcal{H}^{\alpha,\beta,\gamma}\big([0,T]\times[0,R]\times [0,K]\big)$ the Banach H\"{o}lder space whose elements are continuous functions $(t,x,l)\mapsto u(t,x,l)$ defined in $[0,T]\times[0,R]\times [0,K]$, and valued in $\R$.
The norm $\|\cdot\|_{\mathcal{H}^{\alpha,\beta,\gamma}([0,T]\times[0,R]\times [0,K])}$ is defined for all $u\in \mathcal{H}^{\alpha,\beta,\gamma}\big([0,T]\times[0,R]\times [0,K]\big)$ by:
\begin{eqnarray*}
&\|u\|_{\mathcal{H}^{\alpha,\beta,\gamma}([0,T]\times[0,R]\times [0,K])}~~=~~|u|_{[0,T]\times[0,R]\times [0,K]}~~+~~|u|^{\alpha}_{t,[0,T]\times[0,R]\times [0,K]}\\
&+|u|^{\beta}_{x,[0,T]\times[0,R]\times [0,K]}+|u|^{\gamma}_{l,[0,T]\times[0,R]\times [0,K]},
\end{eqnarray*}
with:
\begin{align*}
&|u|_{[0,T]\times[0,R]\times [0,K]}:=\sup\Big\{|u(t,x,l)|,~~(t,x,l)\in [0,T]\times[0,R]\times [0,K]\Big\},\\
&|u|^{\alpha}_{t,[0,T]\times[0,R]\times [0,K]}:=\sup\Big\{\ds \frac{|u(t,x,l)-u(s,x,l)|}{|t-s|^\alpha},~~t\neq s,~~|t-s|\leq 1,\\
&\hspace{4.5 cm}(t,s,x,l)\in [0,T]^2\times[0,R]\times [0,K]\Big\},\\ 
&|u|^{\beta}_{x,[0,T]\times[0,R]\times [0,K]}:=\sup\Big\{\ds \frac{|u(t,x,l)-u(t,y,l)|}{|x-y|^\beta},~~x\neq y,~~|x-y|\leq 1,\\
&\hspace{4.5 cm}(t,x,y,l)\in [0,T]\times[0,R]^2\times [0,K]\Big\},\\ 
&|u|^{\alpha}_{l,[0,T]\times[0,R]\times [0,K]}:=\sup\Big\{\ds \frac{|u(t,x,l)-u(s,x,q)|}{|l-q|^\gamma},~~l\neq q,~~|l-q|\leq 1,\\
&\hspace{4.5 cm}(t,x,l,q)\in [0,T]\times[0,R]\times [0,K]^2\Big\} .
\end{align*}
For $q\in [1,+\infty)$, we denote by $L^q((0,T)\times (0,R)\times (0,K))$ the usual space of measurable real maps defined on $[0,T]\times[0,R]\times [0,K]$, with $p$-th power of their absolute value Lebesgue integrable, endowed with the following norm $\|\|_{L^q((0,T)\times (0,R)\times (0,K))}$, defined for any $f\in L^q((0,T)\times (0,R)\times (0,K))$ by:
\begin{eqnarray*}
 \|f\|_{L^q((0,T)\times (0,R)\times (0,K))}=\Big(\int_0^T\int_0^R\int_0^K|f(t,x,l)|^qdtdxdl\Big)^{\frac{1}{q}}.    
\end{eqnarray*}
On the other hand, $L_{\infty}((0,T)\times (0,R)\times (0,K))$ is the set of bounded real measurable maps defined defined on $[0,T]\times[0,R]\times [0,K]$, endowed with the standard infinite norm $|\cdot|_\infty$, defined by:
$$\forall f \in L_{\infty}((0,T)\times (0,R)\times (0,K)),~~|f|_{\infty}=\sup~\{|f(t,x,l)|,~~(t,x,l)\in [0,T]\times[0,R]\times [0,K]\}$$
We denote by $W^{1,\infty}([0,T]\times[0,R]\times [0,K])$ the set of real valued bounded Lipschitz functions defined on $[0,T]\times[0,R]\times [0,K]$, endowed with the norm $|\cdot|_{W^{1,\infty}([0,T]\times[0,R]\times [0,K])}$. We write for any $f\in W^{1,\infty}([0,T]\times[0,R]\times [0,K])$:
$$|f|_{\lfloor 
W^{1,\infty}([0,T]\times[0,R]\times [0,K])\rfloor}=|\partial f|_\infty$$
the best Lipschitz constant for $f$ and $$
|f|_{W^{1,\infty}([0,T]\times[0,R]\times [0,K])}=|f|_\infty + |\partial f|_\infty.
$$
The set $\mathcal{C}_c^\infty((0,T)\times (0,R)\times(0,K))$ consists of infinite continuous real differentiable functions on $(0,T)\times (0,R)\times (0,K)$, with compact support strictly included in $(0,T)\times (0,R)\times (0,K)$.  

Finally, $W_2^{1,2,0}\big((0,T)\times (0,R)\times (0,K)\big)$ is the Sobolev space, consisting of all elements $f\in L^2\big((0,T)\times (0,R)\times (0,K)\big)$ satisfying:
\begin{align*}
&\exists (g,h,v)\in L^2\big((0,T)\times (0,R)\times (0,K)\big)^3,~~\text{such that:}~~\forall  \phi \in \mathcal{C}_c^\infty((0,T)\times (0,R)\times(0,K)):\\
&\ds \int_0^T\int_0^R\int_0^Kf(t,x,l)\partial_t\phi(t,x,l)dtdxdl=-\int_0^T\int_0^R\int_0^Kg(t,x,l)\phi(t,x,l)dtdxdl,\\
&\ds \int_0^T\int_0^R\int_0^Kf(t,x,l)\partial_x\phi(t,x,l)dtdxdl=-\int_0^T\int_0^R\int_0^Kh(t,x,l)\phi(t,x,l)dtdxdl,\\
&\ds \int_0^T\int_0^R\int_0^Kf(t,x,l)\partial^2_x\phi(t,x,l)dtdxdl=\int_0^T\int_0^R\int_0^Kv(t,x,l)\phi(t,x,l)dtdxdl.
\end{align*}
As usual, we identify almost surely $\partial_tf$ (resp. $\partial_xf$ and $\partial_x^2f$) with $g$ (resp. $h$ and $v$) and we endow $W_2^{1,2,0}\big((0,T)\times (0,R)\times (0,K)\big)$ with the following norm, defined for any $f\in W_2^{1,2,0}\big((0,T)\times (0,R)\times (0,K)\big)$ by:
\begin{align*}
 &\|f\|_{W_2^{1,2,0}((0,T)\times (0,R)\times (0,K))}=\|f\|_{L^2((0,T)\times (0,R)\times (0,K))}+\|\partial_t f\|_{L^2((0,T)\times (0,R)\times (0,K))}+\\
 &\|\partial_xf\|_{L^2((0,T)\times (0,R)\times (0,K))}+\|\partial_x^2 f\|_{L^2((0,T)\times (0,R)\times (0,K))}.
\end{align*}

Using these notations, we refer the reader to Definition \ref{definition-espace-solutions} and the paragraph following it for the explanation of the regularity classes
$\mathfrak{C}^{1+\frac{\alpha}{2},2+\alpha,\frac{\alpha}{2}}_{\{0\}} \big(\Omega_T\big)$, $\mathfrak{C}^{1,2,0}_{\{0\}} \big(\Omega_T\big)$ and $\mathfrak{Lip}^{2,0}_{\{0\}} \big(\Omega\big)$.

\end{document}